\documentclass[11pt,oneside,leqno]{amsart}
\usepackage[utf8]{inputenc}
\usepackage[english]{babel}
\usepackage{amssymb,amsmath,latexsym,amsthm,amsfonts,bbm,dsfont}
\usepackage{geometry}
\usepackage{float}
\usepackage{url}
\usepackage[colorlinks=true,citecolor=red,linkcolor=blue,pdfpagetransition=Blinds]{hyperref}
\numberwithin{equation}{section}
\usepackage{mathtools, stmaryrd}
\usepackage{natbib}
\usepackage{textgreek}
\usepackage{xcolor}
\numberwithin{equation}{section}
\usepackage{graphicx}

\usepackage{comment}
\newtheorem{theorem}{Theorem}[section]
\newtheorem{proposition}[theorem]{Proposition}
\newtheorem{lemma}[theorem]{Lemma}
\newtheorem{corollary}[theorem]{Corollary}
\newtheorem{remark}[theorem]{Remark}

\newcommand{\trp}{\mathtt t^+}
\newcommand{\trm}{\mathtt t^-}
\newcommand{\trpm}{\mathtt t^\pm}

\begin{document}
\title[Asymptotic behavior of Carleman weight functions]{Asymptotic behavior of Carleman weight functions and application to controllability}
\author{Ariel A. P\'erez}
\address{Departmento de Ingenier\'ia Matem\'atica, Universidad de Concepci\'on, Concepci\'on, Chile}
\email{arperez@udec.cl}
\begin{abstract}
 In the development of controllability and inverse problem results for semi-discrete systems, by using Carleman estimates, it is required to estimate of the discrete operators applied to Carleman weight functions. This work aims to establish the asymptotic behavior of Carleman weight functions under these discrete operators. We provide a characterization of the error term in arbitrary order and dimension, extending previously known results. This generalization is of independent interest due to its applications in deriving Carleman estimates for semi-discrete deterministic and stochastic operators. The aforementioned estimates hold for Carleman weight functions used for parabolic, hyperbolic, and elliptic operators, which are applied to obtain control and inverse problems results for those operators.  As an application, we derive a Carleman estimate for an adjoint
fully discrete parabolic operator in arbitrary dimension, from which a
relaxed observability inequality and a $\phi$-null controllability result,
with controls bounded uniformly with respect to the discretization
parameters, are obtained. 
\end{abstract}
\maketitle
\section{Introduction}
 Let $d\geq 1$ and $T,L_{1}$, $\ldots$,  $L_{d}$ be positive real numbers; and $\Omega:=\Pi_{i=1}^{d}(0,L_{i})$, with $\omega\Subset \Omega$. Consider the following parabolic problem
\begin{equation}\label{system:control}
\begin{cases}
\partial_{t} y - \sum_{i=1}^{d}\partial_{x_{i}}( \gamma_{i}\partial_{x_{i}}y)  = \mathbbm{1}_\omega v, & (x,t) \in \Omega \times (0,T),\\
y(x,t) = 0, & (x,t) \in \partial\Omega \times (0,T),\\
y(x,0) = g(x), & x \in \Omega,
\end{cases}
\end{equation}
where $\gamma_{i}(x)>0$ for all $x\in\Omega$, and $\mathbbm{1}_\omega$ stands for the indicator function of the set $\omega$. Given an initial datum $g\in L^{2}(\Omega)$, the goal is to find $v\in L^{2}((0,T)\times\Omega)$ such that $y(T)=0$. In this case, it is said that the system \eqref{system:control} is null-controllable. The aforementioned controllability problem was solved independently by Lebeau and Robbiano \cite{lebeau1995controle} and Fursikov and Imanuvilov \cite{fursikov-1996}. The usual methodology is that the null-controllability result for the system \eqref{system:control} is equivalent, by a duality argument, to the following observability inequality
\begin{equation}\label{adjoint:continuous}
    \|q(0)\|^{2}_{L^{2}(\Omega)}\leq C_{\text{obs}}^{2}\left\| q\right\|^{2}_{L^{2}(\omega\times(0,T))},
\end{equation}
where $q$ is a solution to the adjoint system
\begin{equation}\label{system:control:adjoint}
\begin{cases}
\partial_{t} q + \sum_{i=1}^{d}\partial_{x_{i}}( \gamma_{i}\partial_{x_{i}}q)  = 0, & (x,t) \in \Omega \times (0,T),\\
q(x,t) = 0, & (x,t) \in \partial\Omega \times (0,T),\\
q(x,T)=q_{T}(x),&x\in \Omega.
\end{cases}
\end{equation}
A popular strategy to prove the inequality \eqref{adjoint:continuous} is based on a Carleman estimate for the system \eqref{system:control:adjoint}. The first application of these estimates was to prove unique continuation properties (see \cite{carleman1939}). Later, for control and inverse problems Carleman estimates have become a fundamental tool. We refer \cite{fu-2019} and the references therein for a unified approach to Carleman estimates for second-order partial differential equations.
\par Unfortunately, the null-controllability of \eqref{system:control} is not inherited, in general, by its discretization, even when it is null-controllable for fixed discretization parameters, the observability constants (and hence the norms of the controls) are in general not uniform as the discretization parameters tend to zero; see \cite{ervedosa-puel-ulhman,zuazua:2005} . The next example illustrates this situation by building a non-observable function of a discretization of the adjoint system \eqref{system:control:adjoint}. To this end, let $M,N\in \mathbb{N}$, $T>0$, and define $h=1/(M+1)$ and $\Delta t=T/N$. Then, the uniform meshes in space and time are given by follows
\begin{align*}
    \mathcal{K}=\{x_i:=ih\,;\, i\in \{ 0,\ldots,M+1 \}\},\quad
    \overline{\mathcal{N}}=\{t_{k}:=k\Delta t\, ;\, k\in \{ 0,\ldots,N \}\}. 
\end{align*}
Thus, we consider a regular partition of the square $(0,1)^{2}$ given by $\mathcal{M}^{2}$ where $\mathcal{M}:=(0,1)\cap\mathcal{K}$. In this example we write $q_{i.j}(t)=q(x_{i},x_{j},t_{k})$ for $x_{i},x_{j}\in\mathcal{K}$ and $t_{k}\in\overline{\mathcal{N}}$. Using the finite-difference scheme, the adjoint discrete system \eqref{system:control:adjoint} of the heat equation, for $d=2$, $\gamma_{i}(x)=1$ and $L_{i}=1
$, can be written as
\begin{equation}\label{sys:discrete:adjoint}
\begin{aligned}
    D_{t}q_{i,j}(t_{k})+\Delta_{h}\trp q_{i,j}(t_{k})&=0, \quad  (x_{i},x_{j},t_{k})\in\mathcal{M}^{2}\times\mathcal{N}^{\ast},\\
    q_{i,0}(t_{k})=q_{i,M+1}(t_{k})&=0,\  \ \ (x_{i},t_{k})\in\mathcal{M}\times \overline{\mathcal{N}},\\
    q_{0,j}(t_{k})=q_{M+1,j}(t_{k})&=0, \ \ \ (x_{j},t_{k})\in \mathcal{M}\times\overline{\mathcal{N}},\\
    \end{aligned}
\end{equation}
where the five-point approximation of the Laplacian operator in uniform meshes is utilized, as given by
\begin{equation}
    \Delta_{h}q_{i,j}:=\frac{1}{h^{2}}\left(q_{i+1,j}+q_{i,j+1}+q_{i-1,j}+q_{i,j-1}-4q_{i,j}\right),
\end{equation}
and $D_{t} q:=\dfrac{\trp q-\trm q}{\Delta t}$ in $\mathcal{N}^{\ast}:=\{(k-1/2)\Delta t\,;\, k\in \{1,\ldots,N\}\}$, where $\trpm q(t_{k}):=q(t_{k}\pm \Delta t/2)$. Thus, in \eqref{sys:discrete:adjoint} the unknown $q_{i,j}$ is defined on the primal mesh $\overline{\mathcal{N}}$, while the equation, that is, $D_{t}q_{i,j}$ and $\Delta_{h}\trp q_{i,j}$ is evaluated on the dual time grid $\mathcal{N}^{\ast}$. Combining the ideas of \cite{zuazua:2005} and \cite{ervedoza:2019}, we can build a non-observable solution of the system \eqref{sys:discrete:adjoint}. 
Indeed, notice that for
\begin{equation}
    \tilde{q}_{i,j}:=\begin{cases}
    \ \ 1,& i=j \text{ even,}\\
    \ \ 0,& i\ne j,\\
    -1,& i=j\text{ odd},\\
    \end{cases}
\end{equation}
it follows that 
\begin{equation}
    \Delta_{h}\tilde{q}_{i,j}=\begin{cases}
    -\frac{4}{h^{2}},& i=j \text{ even,}\\
    \ \ \  0,& i\ne j, \\
    \ \ \frac{4}{h^{2}},& i=j\text{ odd,}.\\
    \end{cases}=-\frac{4}{h^{2}}\tilde{q}_{i,j}.
\end{equation}
Then $\tilde{q}_{i,j}$ is an eigenfunction of $\Delta_{h}$, with eigenvalue $-4/h^{2}$. In turn, for $\alpha \Delta t<1$ and with 
$\displaystyle \mathfrak{a}:=\left( \frac{1}{1-\alpha\Delta t}\right)^{1/\alpha\Delta t}$, we obtain that $D_{t}\mathfrak{a}^{\alpha t}=\alpha\trp \mathfrak{a}^{\alpha t}$. Thus, 
considering $\alpha=4/h^{2}$ such that $\alpha \Delta t<1$ it follows that $\overline{q}_{i,j}:=\mathfrak{a}^{-4(T-t)/h^{2}}\tilde{q}_{i,j}$ solves \eqref{sys:discrete:adjoint} since $D_{t}\overline{q}=\frac{4}{h^{2}}\trp\overline{q}_{i,j}$.  The condition $\alpha \Delta t<1$, that is, $4\Delta t<h^{2}$, it means that $\mathfrak{a}$ is well defined. Therefore, if the discrete set $\omega$ contains no diagonal nodes $(x_{i},x_{i})$, then $\overline{q}$ vanishes identically on $\omega\times\overline{\mathcal{N}}$, whereas $\overline{q}(\cdot,0)$ is nonzero, with norm of size $\mathfrak{a}^{-4T/h^{2}}$. Consequently, for such $\omega$, no observability inequality like \eqref{adjoint:continuous} holds for \eqref{sys:discrete:adjoint}, nor even with a constant depending on $h$ and $\Delta t$, and the corresponding fully discrete control system is not null-controllable for any fixed $h$ and $\Delta t$.
\par Due to the aforementioned obstructions, uniform null controllability cannot be expected in general in the fully or semi-discrete setting. For discretized parabolic systems, an alternative objective is to build controls, bounded uniformly with respect to the discretization parameter, such that at a given time the norm of the solution decreases exponentially as $\Delta t$ and/or $h$ tend towards zero. That is, in the case of space semi-discrete system, to build semi-discrete control $u$ such that at time $T$ the solution of the system $y$ (with an appropriate norm on the space of solutions $X$) verifies
\begin{equation}
    \left\| y(x,T)\right\|_{X} \approx C\sqrt{\phi(h)},
\end{equation}
with $C>0$, $\left\| u\right\|_{Y}\leq C$ uniformly respect to $h$ (with the appropriate norm in the space of controls $Y$), and
where $h\longrightarrow \phi(h)$ is any given function of the discretization parameter such that
\begin{equation}\label{eq:phi}
    \liminf_{h\rightarrow 0}\frac{\phi(h)}{e^{-c/h}}>0.
\end{equation}
This approach is known as $\phi$-null controllability for the space semi-discrete system, and also can be formulated for time-discrete systems and/or fully-discrete systems.
 
\par In several settings, the $\phi$-null controllability has been studied and the analysis is based on the penalization Hilbert uniqueness method, we refer to the work \cite{boyer-canum} for an overview of the method for space semi-discrete parabolic setting. In particular, the works \cite{boyer-2014,boyer-2010-1d-elliptic,BHLR:2010} establish $\phi$-null controllability results for semi-discrete parabolic operators in the spatial semi-discrete setting. For discontinuous diffusion coefficient we refer to \cite{Thuy} and to \cite{Thuy2} for the setting of space semi-discrete parabolic operator in Banach space. Also in the space-discrete setting in \cite{BHSDT:2019}, the authors examined insensitizing controls for a semilinear parabolic equation in one dimension. For higher order operator we refer to \cite{CLTP-2022} and \cite{kumar:hal-04828698}, where the boundary $\phi$-null controllability is studied. Likewise, in  \cite{BHS:2020} is presented the $\varphi(\Delta t)$-null controllability within the time-discrete framework for the time-discrete heat equation and also for some parabolic systems. The same kind of controllability for a simplified stabilized Kuramoto-Sivashinsky system was also addressed in \cite{HS:2023} and recently for parabolic system with Kirchhoff boundary condition in \cite{bhandari:hal-04597449}. As for fully-discrete systems, an important contribution was made in \cite{GC-HS-2021} for null Dirichlet boundary conditions and in \cite{LMPZ:2023} for dynamic boundary conditions both of them in the one-dimensional setting.

\par In all the aforementioned works, their results follow from suitable Carleman estimates, and they need an estimate of the discrete difference and average operator applied on the Carleman weight functions. It is worth noticing that, in arbitrary dimension, there are no results for general computation of the discrete operators over the continuous Carleman weight functions; see \cite{LOP-2020} for the one-dimensional case. Hence, in this work, we aim to cover that gap by giving general results about the asymptotic behavior of the Carleman weight functions; see Proposition \ref{prop:weight}, Theorems \ref{theo:weight:estimates}, \ref{theo:time:derivate} and \ref{theo:fully:weight:estimates}. We emphasize that our introduction is focusing on controllability issues but our scope also reaches semi-discrete inverse problem and unique continuation properties for hyperbolic and elliptic operators respectively; see Section \ref{sec:concluding} for some remarks in these directions. \\
 \subsection{Discrete setting and main result}\label{subsec:discrete:setting}
From now on, $\Omega=\prod_{i=1}^{d}(0,L_{i})$ and, for simplicity, we consider a uniform mesh of size $h$ in each direction, i.e. $L_{i}=(M_{i}+1)h$ with $M_{i}\in\mathbb{N}$. 
We now use the vector notation $x=(x_{1},\ldots,x_{d})$ for the nodes instead of the index notation of \eqref{sys:discrete:adjoint}; for instance, $q_{i,j}(t)$ in \eqref{sys:discrete:adjoint} corresponds to $q(x,t)$ with $x=(x_{i},x_{j})$. We follow the notation of \cite{BHLR:2010,LDOP-2021,LPP:2026,LPP:2026b,LLRP:2025}. The primal mesh and its closure are
\begin{equation*}
\mathcal{M}:=\prod_{j=1}^{d}\{kh\,;\,k=1,\ldots,M_{j}\},\qquad \overline{\mathcal{M}}:=\prod_{j=1}^{d}\{kh\,;\,k=0,\ldots,M_{j}+1\},
\end{equation*}
and, for $i\neq j$, the dual meshes in the direction $e_{i}$ and in the directions $e_{i},e_{j}$ are
\begin{equation*}
\begin{aligned}
\mathcal{M}_{i}^{\ast}&:=\big\{x+\tfrac{h}{2}e_{i}\,;\,x\in\overline{\mathcal{M}},\ x_{i}\neq L_{i},\ x_{k}\notin\{0,L_{k}\}\ \forall k\neq i\big\},\\
\overline{\mathcal{M}}_{ij}&:=\big\{x+\tfrac{h}{2}(e_{i}+e_{j})\,;\,x\in\overline{\mathcal{M}},\ x_{i}\neq L_{i},\ x_{j}\neq L_{j},\ x_{k}\notin\{0,L_{k}\}\ \forall k\notin\{i,j\}\big\}.
\end{aligned}
\end{equation*}
The boundary of $\mathcal{M}$ in the direction $e_{i}$ is $\partial_{i}\mathcal{M}:=\{x\in\overline{\mathcal{M}}\,;\,x_{i}\in\{0,L_{i}\},\ x_{k}\notin\{0,L_{k}\}\ \forall k\neq i\}$, and $\partial\mathcal{M}:=\bigcup_{i=1}^{d}\partial_{i}\mathcal{M}$. On $\partial_{i}\mathcal{M}$ we set $n_{i}(x):=-1$ if $x_{i}=0$ and $n_{i}(x):=1$ if $x_{i}=L_{i}$ and, for $u$ defined on $\mathcal{M}_{i}^{\ast}$, the trace $t_{r}^{i}(u)(x):=u(x-\frac{h}{2}n_{i}(x)e_{i})$. The difference and average operators $D_{i}$, $A_{i}$ are defined in Section~\ref{preliminary}; they map functions defined on $\overline{\mathcal{M}}$ to functions defined on $\mathcal{M}_{i}^{\ast}$, and conversely. We also write $D_{ij}^{2}:=D_{i}D_{j}$ and $A_{ij}^{2}:=A_{i}A_{j}$.
For the time variable, with $\Delta t=T/N$, $N\in\mathbb{N}$, we consider the primal and dual time meshes
\begin{equation*}
\begin{aligned}
&\mathcal{N}:=\{k\Delta t\,;\,k=1,\ldots,N\},\quad \overline{\mathcal{N}}:=\mathcal{N}\cup\{0\}, \\
&\mathcal{N}^{\ast}:=\{(k-\tfrac12)\Delta t\,;\,k=1,\ldots,N\},\quad \overline{\mathcal{N}}^{\ast}:=\mathcal{N}^{\ast}\cup\{T+\tfrac{\Delta t}{2}\},
\end{aligned}
\end{equation*}
the translations $\trpm y(t):=y(t\pm\Delta t/2)$ and $D_{t}y:=(\trp y-\trm y)/\Delta t$, which maps functions defined on $\overline{\mathcal{N}}$ (resp. $\overline{\mathcal{N}}^{\ast}$) to functions defined on $\mathcal{N}^{\ast}$ (resp. $\mathcal{N}$). We also write $D_{it}^{2}:=D_{i}D_{t}=D_{t}D_{i}$. The time boundary is $\partial\mathcal{N}:=\{0,T\}$, with $n(0):=-1$, $n(T):=1$, and $\partial\mathcal{N}^{+}:=\{T\}$.
The discrete integrals are defined by $\displaystyle\int_{\mathcal{X}}u:=h^{d}\sum_{x\in\mathcal{X}}u(x)$ for $\mathcal{X}\in\{\mathcal{M},\mathcal{M}_{i}^{\ast},\overline{\mathcal{M}}_{ij}\}$ (or a subset of them), $\displaystyle\int_{\partial_{i}\mathcal{M}}u:=h^{d-1}\sum_{x\in\partial_{i}\mathcal{M}}u(x)$, $\displaystyle\int_{\mathcal{N}}u:=\Delta t\sum_{t\in\mathcal{N}}u(t)$, $\displaystyle\int_{\mathcal{N}^{\ast}}u:=\Delta t\sum_{t\in\mathcal{N}^{\ast}}u(t)$, $\displaystyle\int_{\partial\mathcal{N}}u:=u(0)+u(T)$, and $\displaystyle\int_{\mathcal{X}\times\mathcal{Y}}u:=\int_{\mathcal{Y}}\int_{\mathcal{X}}u$ for products of spatial and time meshes.
Accordingly, $\displaystyle\|u\|_{L^{2}_{h}(\mathcal{X})}^{2}:=\int_{\mathcal{X}}|u|^{2}$, $\displaystyle\|u\|_{L^{2}_{h}(\mathcal{X}\times\mathcal{Y})}^{2}:=\int_{\mathcal{X}\times\mathcal{Y}}|u|^{2}$ and $\|a\|_{L^{\infty}_{h}}:=\max_{\mathcal{M}\times\mathcal{N}}|a|$. Functions defined on $\overline{\Omega}$ or on $\overline{\Omega}\times[0,T]$, such as the coefficients $\gamma_{i}$ and the Carleman weights, are sampled at the nodes where they are evaluated; for instance, in $D_{i}(\gamma_{i}D_{i}y)$ the coefficient $\gamma_{i}$ is evaluated on $\mathcal{M}_{i}^{\ast}$. Finally, given an open set $\omega\subset\Omega$, the discrete control set $\omega\cap\mathcal{M}$ is still denoted by $\omega$.
\par Almost all the fully-discrete $\phi(h)$-controllability results are in 1-D setting. The main difficulty to extend these results into higher-dimension is that there is no general asymptotic behavior of Carleman weight function. Then, adopting the same notation from \cite{LDOP-2021,LPP:2026,LPP:2026b,LLRP:2025} for the sets $\mathcal{M}$, $\mathcal{M}_{i}^{\ast}$, $\partial_{i}\mathcal{M}$, $\mathcal{N}$, the operators $D_{i}$, $A_{i}$, $t_{r}^{i}$, we show how to use Theorems \ref{theo:weight:estimates}, \ref{theo:time:derivate} and \ref{theo:fully:weight:estimates} to obtain a fully discrete Carleman estimate which implies $\phi-$controllability result for the following system
\begin{equation}\label{eq:controlled}
\begin{cases}
D_{t}y-\sum_{i=1}^{d}D_{i}(\gamma_{i}D_{i}\trp y)+\trp(a\,y)=1_{\omega}v,&\forall(x,t)\in\mathcal{M}\times\mathcal{N}^{\ast},\\
y\big|_{\partial\mathcal{M}}=0,&\forall t\in\mathcal{N}^{\ast},\\
y(x,0)=y_{0}(x),&\forall x\in\mathcal{M},
\end{cases}
\end{equation}
where  $y\in C(\overline{\mathcal{M}}\times\overline{\mathcal{N}})$, $y_{0}\in C(\mathcal{M})$ is the initial datum, $v\in C(\omega\times\mathcal{N}^{\ast})$ is the control function acting on a discrete subset $\omega\subset\mathcal{M}$, $a=a(x,t)$ is a bounded potential defined on $\mathcal{M}\times\mathcal{N}$ so that $\trp(aq)$ is evaluated on $\mathcal{M}\times\mathcal{N}^{\ast}$;  $\Gamma(x):= \mathrm{Diag}(\gamma_{1}(x),\gamma_{2}(x),\ldots, \gamma_{d}(x))$, where $\gamma_{i}\in W^{1,\infty}(\Omega)$, $\gamma_i>0$ for all $i=1,\ldots,d$, and 
$$
 \mbox{reg}(\gamma):= \text{ess}\sup_{\substack{x\in \mathcal{M} \\ i=1,\ldots,d}}\left(\gamma_i+\frac{1}{\gamma_i}+{ \sum_{j=1}^{d}|\partial_{j}\gamma_i|^2}\right)<c_0.
$$
for some constant $c_0> 0$. Since each $\gamma_{i}$ is Lipschitz continuous on $\overline{\Omega}$ its values at the primal and dual meshes are well defined. Following the penalized HUM strategy, see \cite{boyer-canum}, we have the $\phi$-controllability result for the system \eqref{eq:controlled}
\begin{theorem}\label{thm:phi:controllability}
Let $T>0$, $\mu\geq 1$, let $\omega\subset \Omega$ a nonempty open set, let $\gamma$ satisfy $\text{reg}(\gamma)<c_{0}$ and let $a$ be a bounded potential. There exist positive constants $h_{0}$, $C_{0}$, $C$ and $C_{1}$, depending only on $d,\Omega,\omega$ and $c_{0}$ such that the following holds. Set $m_{\mu}:=\min\{\mu/4,1\}$, 
\begin{equation*}
    h_{1}:=C_{0}\left(1+\frac{1}{T}+\left\|a \right\|^{2/3}_{L_{h}^{\infty}} \right)^{-\max\{1,4/\mu\}},\quad C_{\text{obs}}:=e^{C\left( 1+\frac{1}{T}+\left\| a\right\|_{L_{h}}^{2/3}+T\left\|a \right\|_{L_{h}^{\infty}}\right)}
    \end{equation*}
and assume that
\begin{equation*}
    h\leq\min\{h_{0},h_{1}\},\quad 0<\Delta t\leq \min\{T^{-2}h^{\mu},(4\left| a\right|_{L^{\infty}})^{-1}\}.
\end{equation*}
Then, for any initial datum $y_{0}\in C(\mathcal{M})$, there exists a control function $v$ satisfying
\begin{equation}\label{eq:control:bound}
\|v\|_{L^{2}_{h}(\omega\times\mathcal{N}^{\ast})}\leq\sqrt{C_{\text{obs}}}\,\|y_{0}\|_{L^{2}_{h}(\mathcal{M})},
\end{equation}
and the solution of the controlled system \eqref{eq:controlled} satisfies
\begin{equation}\label{eq:phi:null}
\|y(T)\|_{L^{2}_{h}(\mathcal{M})}\leq\sqrt{C_{\text{obs}}}\,\sqrt{\phi(h)}\,\|y_{0}\|_{L^{2}_{h}(\mathcal{M})},\qquad\phi(h)=e^{-C_{1}/h^{\min\{\mu/4,1\}}}.
\end{equation}
\end{theorem}

\begin{remark}\label{rem:convergence}
The scheme \eqref{eq:controlled} is the standard $(2d{+}1)$-point finite-difference
discretization in space combined with the implicit Euler method
in time. It is
therefore consistent of order $\mathcal O(h^{2}+\Delta t)$ for smooth solutions
and unconditionally stable; the restriction $\Delta t\le T^{-2}h^{\mu}$ in
Theorem \ref{thm:phi:controllability} originates in the Carleman estimate of Theorem \ref{thm:Carleman:q}, and
$\Delta t\|a\|_{L^{\infty}_h}\le1/4$ is only used in the discrete Gronwall
argument, neither is a CFL condition. Theorem \ref{thm:phi:controllability} yields a family
$(v_{h,\Delta t},y_{h,\Delta t})$ with
$\|v_{h,\Delta t}\|_{L^{2}_h(\omega\times\mathcal N^{*})}\le C\|y_0\|_{L^2_h}$
uniformly in $(h,\Delta t)$ and
$\|y_{h,\Delta t}(T)\|_{L^{2}_h}\le C\sqrt{\varphi(h)}\,\|y_0\|_{L^2_h}\to0$ as
$h\to0$. The convergences of the controls in the time-discrete setting, through the
notion of solution by transposition, is studied in \cite{BDK:2026}, and we refer to \cite{boyer-canum} and the
references therein for the space semi-discrete case. Since the object of
this paper is the asymptotics of the weight functions, we leave the fully
discrete convergence analysis in arbitrary dimension for future work.
\end{remark}

\begin{remark}\label{rem:conditions}
Neither of the conditions on $\Delta t$ in Theorem~\ref{thm:phi:controllability} is a CFL condition, since the implicit Euler scheme \eqref{eq:controlled} is unconditionally stable. The condition $\Delta t\|a\|_{L^{\infty}_{h}}\leq 1/4$ guarantees that each implicit step is uniquely solvable for a potential of arbitrary sign, and it is used in the discrete Gronwall argument \eqref{eq:energy:monotonicity}. The condition $\Delta t\leq T^{-2}h^{\mu}$ is a smallness condition inherited from the Carleman estimate of Theorem \ref{thm:Carleman:q}, which requires $\tau h(\delta T^{2})^{-1}\leq\epsilon_{0}$ and $\tau^{4}\Delta t(\delta^{4}T^{6})^{-1}\leq\epsilon_{0}$, see \eqref{eq:CFL:conditions}. In the proof of Theorem \ref{thm:observability}, the parameter $\tau$ is taken as large as these conditions allow, namely $\tau=cT^{2}h^{-m_{\mu}}$ with $c>0$ small: if $\mu\geq4$, then $\tau h(\delta T^{2})^{-1}=c/\delta$ and $\tau^{4}\Delta t(\delta^{4}T^{6})^{-1}\leq c^{4}h^{\mu-4}/\delta^{4}$, while if $\mu<4$, then $\tau h(\delta T^{2})^{-1}=ch^{1-\mu/4}/\delta$ and $\tau^{4}\Delta t(\delta^{4}T^{6})^{-1}\leq c^{4}/\delta^{4}$. The resulting factor $e^{-C\tau/T^{2}}=e^{-Cc/h^{m_{\mu}}}$ produces $\phi(h)$. Hence $\mu$ quantifies how small $\Delta t$ is compared with $h$: the larger $\mu$, the smaller $\phi(h)$, up to the saturation $m_{\mu}=1$ for $\mu\geq4$, which comes from the spatial condition $\tau h\lesssim\delta T^{2}$.
\end{remark}

In \cite[section 4]{LMPZ:2023} is the strategy to follow to prove Theorem \ref{thm:phi:controllability}. Moreover, Theorem \ref{thm:phi:controllability} can be extended to more general functions $\phi$ satisfying
\begin{equation*}
\liminf_{h\to 0}\phi(h)\,e^{C_{1}/h^{\min\{\mu/4,1\}}}>0.
\end{equation*}
Indeed, for such $\phi$, there exists $h^{\ast}>0$ such that $e^{-C_{1}/h^{m_{\mu}}}\leq\phi(h)$ for all $h\leq h^{\ast}$. Then, Theorem~\ref{thm:observability} holds with $\phi(h)$ replacing $e^{-C_{1}/h^{m_{\mu}}}$, and the same penalised HUM argument applies. Here, Theorem \ref{thm:observability} is the observability inequality of  the corresponding adjoint system is given by
\begin{equation}\label{eq:adjoint}
\begin{cases}
D_{t}q+\sum_{i=1}^{d}D_{i}(\gamma_{i}D_{i}\trm q)-\trm(a\,q)=0,&\forall(x,t)\in\mathcal{M}\times\mathcal{N},\\
q\big|_{\partial\mathcal{M}}=0,&\forall t\in\mathcal{N},\\
q(x,T+\Delta t/2)=q_{T}(x),&\forall x\in\mathcal{M}.
\end{cases}
\end{equation}
 for $q\in C(\overline{\mathcal{M}}\times\overline{\mathcal{N}}^{\ast})$. System \eqref{eq:adjoint} is the discrete adjoint of \eqref{eq:controlled} with respect to the inner product $\displaystyle\langle u,w\rangle:=\int_{\mathcal{M}}u\,w$ of $L^{2}_{h}(\mathcal{M})$. Indeed, $y$ is defined on the primal time mesh $\overline{\mathcal{N}}$ and $q$ on the dual time mesh $\overline{\mathcal{N}}^{\ast}$. Multiplying the first equation of \eqref{eq:controlled} by $q$, integrating over $\mathcal{M}\times\mathcal{N}^{\ast}$, and using the discrete integration by parts formulas in time and in space \cite[Proposition 2]{LMPZ:2023}, \cite[Lemma 2.2]{LDOP-2021}, namely
\begin{equation*}
\begin{aligned}
\int_{\mathcal{M}\times\mathcal{N}^{\ast}}D_{t}y\,q&=\langle y(T),q(T+\tfrac{\Delta t}{2})\rangle-\langle y(0),\trp q(0)\rangle-\int_{\mathcal{M}\times\mathcal{N}}y\,D_{t}q,\\
\int_{\mathcal{M}\times\mathcal{N}^{\ast}}D_{i}(\gamma_{i}D_{i}\trp y)\,q&=\int_{\mathcal{M}\times\mathcal{N}}y\,D_{i}(\gamma_{i}D_{i}\trm q),\qquad
\int_{\mathcal{M}\times\mathcal{N}^{\ast}}\trp(a\,y)\,q=\int_{\mathcal{M}\times\mathcal{N}}a\,y\,\trm q,
\end{aligned}
\end{equation*}
where the last two identities follow from the relabelling $t\mapsto t+\Delta t/2$, which maps $\mathcal{N}^{\ast}$ onto $\mathcal{N}$, and from $y=q=0$ on $\partial\mathcal{M}$, we obtain, for solutions of \eqref{eq:controlled} and \eqref{eq:adjoint},
\begin{equation*}
\int_{\omega\times\mathcal{N}^{\ast}}v\,q=\langle y(T),q_{T}\rangle-\langle y_{0},\trp q(0)\rangle.
\end{equation*}
This identity, on which the penalized HUM argument relies, explains the changes between \eqref{eq:controlled} and \eqref{eq:adjoint}; since the dual time mesh is shifted by $\Delta t/2$ with respect to the primal one, the evaluation $\trp$ in \eqref{eq:controlled}, which is implicit in the forward time direction, becomes the evaluation $\trm$ in \eqref{eq:adjoint}, which is implicit in the backward time direction, and the sign of the potential term changes as in the continuous setting. The spatial operators are unchanged, since $D_{i}(\gamma_{i}D_{i}\,\cdot)$ is symmetric in $L^{2}_{h}(\mathcal{M})$ for functions vanishing on $\partial\mathcal{M}$.\\

\par The organization of this work is as follows. In Section \ref{preliminary}, we provided some results concerning the average and difference operators and the asymptotic behavior of smooth functions. In section \ref{sec:asymp:weight}, we apply the results from the previous section to the Carleman weight function. In Section \ref{sec:Carleman}, we state the Carleman estimate for the system \ref{eq:adjoint} and sketch the proof of Theorem \ref{thm:observability}; the proof of the Carleman estimate is given in Section \ref{sec:proof:carleman}, and the computations of the cross-product terms are collect in Section \ref{sec:cross}. Section \ref{sec:concluding} provides concluding remarks about the potential applicability of the results presented in section \ref{sec:asymp:weight} for deterministic and stochastic operators.
\section{Asymptotic behavior of smooth functions}\label{preliminary}
In this section, we aim to establish the asymptotic behavior of a smooth function which will be applied in particular in the next section to Carleman weight functions. To this end, we consider functions continuously defined in a neighborhood of $\overline{\Omega}$. We define the average and difference operators, in the direction $e_{i}$, with $\{e_{i}\}_{i=1}^{d}$ being
the usual basis of $\mathbb{R}^{d}$, as 
\begin{equation}
\begin{split}
    A_{i}u(x)&:=\frac{1}{2}\left( u(x+\frac{h}{2}e_{i})+u(x-\frac{h}{2}e_{i})\right),\\
    D_{i}u(x)&:=\frac{1}{h}\left(u(x+\frac{h}{2}e_{i})-u(x-\frac{h}{2}e_{i}) \right),
    \end{split}
\end{equation}
  For these operators, we have the following product rule.
\begin{lemma}[{\cite[Lemma 2.1 and 2.2]{BHLR:2010}}] \label{base:case}
    Let the functions $u$ and $v$ be continuously defined in a neighborhood of $\overline{\Omega}$. For $i\in\{1,\ldots,d\}$, we have the difference and average operators satisfy
    \begin{align}
        D_{i}(uv)=&D_{i}u\,A_{i}v+D_{i}v\,A_{i}u,\label{D:product}\\
        A_{i}(uv)=&A_{i}u\,A_{i}v+\frac{h^{2}}{4}D_{i}u\,D_{i}v.\label{A:product}
    \end{align}
\end{lemma}
It is possible to generalize the above product rule for the difference and average operators.
\begin{proposition}
Let $u$ and $v$ be two smooth functions and $n\in\mathbb{N}$. Then, it holds a Leibniz rule for the difference operator given by
\begin{equation}\label{eq:discreteleibniz}
    D_{i}^{n}(u\,v)=\sum_{k=0}^{n}\binom{n}{k}D_{i}^{n-k}A_{i}^{k}u\,A^{n-k}_{i}D^{k}_{i}v.
\end{equation}
Moreover, for the average operator, we have the operator identity 
\begin{equation}\label{eq:discrete:leibniz:average}
    A_{i}^{2m}(f)=\sum_{j=0}^{m}\binom{m}{j}\frac{h^{2j}}{4^{j}}D_{i}^{2j}(f),\quad\forall m\in\mathbb{N}.
\end{equation}
for every function $f$ continuously defined in a neighborhood of $\overline{\Omega}$.
\end{proposition}
\begin{proof}
     We first prove \eqref{eq:discreteleibniz} by induction on $n$. For $n=1$, \eqref{eq:discreteleibniz} is \eqref{D:product}. Assume that \eqref{eq:discreteleibniz} holds for some $n\geq 1$. Since the operator $D_{i}$ and $A_{i}$ commute
    , applying $D_{i}$ to \eqref{eq:discreteleibniz} and using \eqref{D:product} we obtain
    \begin{equation*}
        D_{i}^{n+1}(uv)=\sum_{k=0}^{n}\left( D_{i}^{n+1-k}A_{i}^{k}u\,A_{i}^{n+1-k}D_{i}^{k}v+D_{i}^{n-k}A_{i}^{k+1}u\,A_{i}^{n-k}D_{i}^{k+1}v\right).
    \end{equation*}
Shifting the index $k\mapsto k-1$ in the second sum and using $\binom{n}{k}+\binom{n}{k-1}=\binom{n+1}{k}$, we obtain \eqref{eq:discreteleibniz} for $n+1$. On the other hand, \eqref{eq:discrete:leibniz:average} is a consequence of the identity $A^{2}_{i}=I+\frac{h^{2}}{4}D_{i}^{2}$ and the binomial formula.
    
\end{proof}
\begin{remark}
    Identity \eqref{eq:discrete:leibniz:average} is stated for even power. For odd power, since $A_{i}$ commutes with $D_{i}$ it gives \begin{equation*}
        A_{i}^{2m+1}f=\sum_{j=0}^{m}\binom{m}{j}\frac{h^{2j}}{4^{j}}D_{i}^{2j}A_{i}f,\forall m\in\mathbb{N}\cup\{0\}.
    \end{equation*}
\end{remark}
Our first result states the asymptotic behavior of a smooth function when several iterations of the discrete operator are applied. It extends \cite[Proposition 4.1]{LOP-2020} stated in one space dimension, to functions of $d$ variables. Moreover, Corollary \ref{cor:difference:average} provides the asymptotic behavior of operator acting in two different direction and also we provide explicit integral expressions of the remainder terms for arbitrary order $n$. This results are then used in Section \ref{sec:asymp:weight}.

\begin{proposition}\label{prop:difference:average}
Let $f$ be a smooth function defined in a neighborhood of $\overline{\Omega}$. For 
$i\in\{ 1,\ldots,d\}$ we have
\begin{equation}\label{eq:ave:diff}
    \begin{split}
        A^{n}_{i}f=&f+R_{A_{i}^{n}}(f),\quad \forall n\in \mathbb{N},\\
        D^{n}_{i}f=&\partial_{i}^{n}f
        +R_{D_{i}^{n}}(f),\quad\forall n\in\mathbb{N
        }.
    \end{split}
\end{equation}
where  $\displaystyle R_{D_{i}^{n}}(f):=h^{2}\sum_{k=0}^{n}\binom{n}{k}(-1)^{k}\left(\frac{n-2k}{2} \right)^{n+2}\int_{0}^{1}\frac{(1-\sigma)^{n+1}}{(n+1)!}\partial_{i}^{n+2}f(\cdot+\tilde{e}_{i}\sigma)d\sigma$ and \\ $\displaystyle R_{A_{i}^{n}}(f):=\frac{h^{2}}{2^{n}}\sum_{k=0}^{n}\binom{n}{k}\frac{(n-2k)^{2}}{4}\int_{0}^{1}(1-\sigma)\partial_{i}^{2}f(\cdot+\tilde{e}_{i}\sigma)d\sigma$,
with $\tilde{e}_{i}:=\frac{(n-2k)h}{2}e_{i}$ and $\sigma\in[0,1]$.
\end{proposition}
\begin{proof}
The proof follows a similar strategy developed in the proof of \cite[Proposition 4.1]{LOP-2020} with the Taylor expansion \begin{equation}\label{eq:taylor}
 f(x+\eta)=\sum_{j=0}^{l-1}\frac{1}{j!}f^{(j)}(x;\eta,\ldots, \eta)+\int_{0}^{1}\frac{(1-\sigma)^{l-1}}{(l-1)!}f^{(l)}(x+\sigma \eta;\eta,\ldots,\eta)d \sigma.
\end{equation}
Here, for $x,\eta\in\mathbb{R}^{d}$, $f^{(j)}(x;\eta,\ldots,\eta)$ denotes the $j$-th order differential of $f$ at $x$ applied to $j$ copies of the vector $\eta$, that is,
\begin{equation*}
    f^{(j)}(x;\eta,\ldots,\eta)=\sum_{|\beta|=j}\frac{j!}{\beta!}\partial^{\beta}f(x)\eta^{\beta}.
\end{equation*}
In particular, for $\eta=\eta_{i}e_{i}$, $f^{(j)}(x;\eta,\ldots,\eta)=\eta^{j}_{i}\partial_{i}^{j}f(x)$. Denoting $\displaystyle\tau_{(n-2k)i}f(x):=f\left(x+\frac{(n-2k)h}{2}e_{i}\right)$, we use \eqref{eq:taylor} with $l=2$, $\displaystyle\eta=\frac{(n-2k)h}{2}e_{i}$ to obtain
\begin{equation*}
    \tau_{(n-2k)i}f=f+\frac{(n-2k)h}{2}\partial_{i}f+\frac{(n-2k)^{2}h^{2}}{4}\int_{0}^{1}(1-\sigma)\partial_{i}^{2}f(\cdot+\frac{(n-2k)h}{2}e_{i}\sigma)d\sigma.
\end{equation*}
Then, using the above identity for the average operator, we have 
\begin{align*}
    A_{i}^{n}f=&\frac{1}{2^{n}}\left(\tau_{i}+\tau_{-i} \right) ^{n}f\\
        =&\frac{1}{2^{n}}\sum_{k=0}^{n}\binom{n}{k}\tau_{(n-2k)i}f\\
        =&\frac{1}{2^{n}}\sum_{k=0}^{n}\binom{n}{k}\left(f+\frac{(n-2k)h}{2}\partial_{i}f\right)+R_{A_{i}^{n}}(f)
\end{align*}
Thus, using that $\displaystyle\sum_{k=0}^{n}\binom{n}{k}=2^{n}$ and $\displaystyle\sum_{k=0}^{n}\binom{n}{k}k=n2^{n-1}$ we conclude that
\begin{equation*}
    A_{i}^{n}f=f+R_{A_{i}^{n}}(f),
\end{equation*}
which is the first identity from \eqref{eq:ave:diff}.\\    In turn, for the second identity from \eqref{eq:ave:diff}, we proceed similarly. Using \eqref{eq:taylor}, with $l=n+2$ and $\displaystyle\eta=\frac{(n-2k)h}{2}e_{i}$, it follows that
\begin{align}\label{eq:Df}
    D^{n}_{i}f=&\frac{1}{h^{n}}\left(\tau_{i}-\tau_{-i} \right)^{n}f\notag\\
        =&\frac{1}{h^{n}}\sum_{k=0}^{n}\binom{n}{k}(-1)^{k}\tau_{(n-2k)i}f\\
        =&\frac{1}{h^{n}}\sum_{k=0}^{n}\binom{n}{k}(-1)^{k}\sum_{j=0}^{n+1}\frac{1}{j!}\left(\frac{(n-2k)h}{2} \right)^{j}\partial_{i}^{j}f+R_{D_{i}^{n}}(f).\notag
        \end{align}
Now, using from \cite[Equation 2.2]{egorychev} the Tepper identity  
$$\displaystyle\sum_{k=0}^{n}(-1)^{k}\binom{n}{k}(x-k)^{r}=\begin{cases}
    0,&\ 0\leq r<n, \\
    n!,&\ r=n, 
\end{cases}$$ 
we notice that the non-zero terms, in the latter expression of \eqref{eq:Df}, are only for $j=n$ and $j=n+1$. Moreover, we have 
  \begin{equation}\label{eq:tepper:n+1}
      \mathcal{S}:=\sum_{k=0}^{n}(-1)^{k}\binom{n}{k}\left(\frac{n}{2}-k\right)^{n+1}=0.   
  \end{equation}
  Indeed, notice that applying the substitution $k\mapsto n-k$, we obtain
  \begin{equation*}
      \mathcal{S}=\sum_{k=0}^{n}(-1)^{n-k}\binom{n}{n-k}\left(\frac{n}{2}-(n-k)\right)^{n+1}=(-1)^{n}(-1)^{n+1}\sum_{k=0}^{n}(-1)^{k}\binom{n}{k}\left(\frac{n}{2}-k\right)^{n+1}=-\mathcal{S}
  \end{equation*}
  Hence, \eqref{eq:tepper:n+1} holds.
  Thus, by using \eqref{eq:tepper:n+1} it follows that \eqref{eq:Df} becomes
        \begin{equation}
        \begin{split}
        D_{i}^{n}f=&\frac{1}{h^{n}}\sum_{k=0}^{n}\binom{n}{k}(-1)^{k}\sum_{j=0}^{n+1}\frac{1}{j!}\left(\frac{(n-2k)h}{2} \right)^{j}\partial_{i}^{j}f+R_{D_{i}^{n}}(f)\\
        =&\frac{1}{h^{n}}\sum_{j=0}^{n+1}\frac{1}{j!}h^{j}\sum_{k=0}^{n}\binom{n}{k}(-1)^{k}\left(\frac{n}{2}-k \right)^{j}\partial_{i}^{j}f+R_{D_{i}^{n}}(f)\\
        =&\frac{1}{h^{n}n!}h^{n}\sum_{k=0}^{n}\binom{n}{k}(-1)^{k}\left(\frac{n}{2}-k \right)^{n}\partial_{i}^{n}f\\
        &+\frac{1}{h^{n}}\frac{1}{(n+1)!}h^{n+1}\sum_{k=0}^{n}\binom{n}{k}(-1)^{k}\left(\frac{n}{2}-k \right)^{n+1}\partial_{i}^{n+1}f+R_{D_{i}^{n}}(f)\\
        =&\partial_{i}^{n}f+R_{D_{i}^{n}}(f),
    \end{split}
\end{equation} 
which completes the proof.
\end{proof}
A combination of the estimates from \eqref{eq:ave:diff} yields estimates when the composition of the discrete operators is made in different directions.
\begin{corollary}\label{cor:difference:average}
Let $f$ be a smooth function defined in a neighborhood of $\overline{\Omega}$. Let $i,j\in\{1,\ldots,d\}$ with
$i\ne j\in\{ 1,\ldots,d\}$ and $m,n\in\mathbb{N}$. Then 
\begin{equation}
    \begin{split}
        A_{j}^{m}A_{i}^{n}f=&f+R_{A_{j}^{m}}(f)+R_{A_{i}^{n}}(f)+R_{A_{j}^{m}A_{i}^{n}}(f),\\
        D_{j}^{m}D_{i}^{n}f=&\partial_{j}^{m}\partial_{i}^{n}f+R_{D_{j}^{m}}(\partial_{i}^{n}f)+R_{D_{i}^{n}}(\partial_{j}^{m}f)+R_{D_{j}^{m}D_{i}^{n}}(f),
    \end{split}
\end{equation}
and 
\begin{equation}
D_{i}^{n}A_{j}^{m}f=\partial_{i}^{n}f+R_{D_{i}^{n}}(f)+R_{A_{j}^{m}}(\partial^{n}_{i}f)+R_{D_{i}^{n}A_{j}^{m}}(f);
\end{equation}
where $R_{A^{m}_{j}A_{i}^{n}}:=R_{A_{j}^{m}}\circ R_{A_{i}^{n}}$, $R_{D_{j}^{m}D_{i}^{n}}:=R_{D_{j}^{m}}\circ R_{D^{n}_{i}}$ and $R_{D_{i}^{n}A_{j}^{m}}:=R_{A_{j}^{m}}\circ R_{D_{i}^{n}}$ and given by
\begin{equation*}
    R_{A^{m}_{j}A^{n}_{i}}(f):=\sum_{k=0}^{n}\sum_{k'=0}^{m}a_{k,k'}\int_{0}^{1}\int_{0}^{1}(1-\sigma)(1-\sigma^{\prime})\partial_{j}^{2}\partial_{i}^{2}f\left(\cdot+\tilde{e}_{ij}\right)d\sigma'd\sigma,
\end{equation*}
\begin{equation*}
    R_{D^{m}_{j}D^{n}_{i}}(f):=\sum_{k=0}^{n}\sum_{k'=0}^{m}b_{k,k'}\int_{0}^{1}\int_{0}^{1}(1-\sigma)^{n+1}(1-\sigma')^{m+1}\partial_{j}^{m+2}\partial_{i}^{n+2}f\left(\cdot+\tilde{e}_{ij}\right)d\sigma'd\sigma,
\end{equation*}
 and
 \begin{equation*}
    R_{D^{n}_{i}A^{m}_{j}}(f):=\sum_{k=0}^{n}\sum_{k'=0}^{m}c_{k,k'}\int_{0}^{1}\int_{0}^{1}(1-\sigma)^{n+1}(1-\sigma')\partial_{i}^{n+2}\partial_{j}^{2}f\left(\cdot+\tilde{e}_{ij}\right)d\sigma'd\sigma;
\end{equation*}
with $\tilde{e}_{ij}=\tilde{e}_{ij}(n,m,k,k',h,\sigma,\sigma'):=(n-2k)\frac{h}{2}\sigma e_{i}+(m-2k')\frac{h}{2}\sigma'e_{j}$ for $\sigma,\sigma'\in[0,1]$ and  $$a_{k,k'}:=\frac{h^{4}}{2^{m+n}}\binom{m}{k'}\binom{n}{k}\frac{(n-2k)^{2}}{4}\frac{(m-2k')^{2}}{4},$$
$$b_{k,k'}:=\frac{(-1)^{k+k'}h^{4}}{(n+1)!(m+1)!}\binom{m}{k'}\binom{n}{k}\left(\frac{n-2k}{2}\right)^{n+2}\left(\frac{m-2k'}{2}\right)^{m+2},$$
and 
$$c_{k,k'}:=\frac{(-1)^{k}h^{4}}{(n+1)!2^{m}}\binom{m}{k'}\binom{n}{k}\left(\frac{n-2k}{2}\right)^{n+2}\left(\frac{m-2k'}{2}\right)^{2}.$$
\end{corollary}
\begin{proof}
     
    Notice that for the case $i=j$ the Proposition \ref{prop:difference:average} applies directly. Since $i\neq j$, translations in the direction $e_{i}$ commute with translations in the direction $e_{j}$ and with $\partial_{j}$. As $D_{j}$, $A_{j}$ are linear combinations of translations in the direction $e_{j}$, and $R_{D_{i}^{n}}$, $R_{A_{i}^{n}}$ are linear combinations of integrals of translates in the direction $e_{i}$ of $\partial_{i}^{n+2}f$ and $\partial_{i}^{2}f$, respectively, the operators $D_{j}^{m}$, $A_{j}^{m}$, $\partial_{j}^{m}$ commute with $R_{D_{i}^{n}}$ and $R_{A_{i}^{n}}$. Hence, by Proposition \ref{prop:difference:average} applied first in the direction $e_{i}$ and then in the direction $e_{j}$,
\begin{equation*}
\begin{aligned}
D_{j}^{m}D_{i}^{n}f&=D_{j}^{m}\partial_{i}^{n}f+D_{j}^{m}R_{D_{i}^{n}}(f)=\partial_{j}^{m}\partial_{i}^{n}f+R_{D_{j}^{m}}(\partial_{i}^{n}f)+R_{D_{i}^{n}}(D_{j}^{m}f)\\
&=\partial_{j}^{m}\partial_{i}^{n}f+R_{D_{j}^{m}}(\partial_{i}^{n}f)+R_{D_{i}^{n}}(\partial_{j}^{m}f)+R_{D_{i}^{n}}(R_{D_{j}^{m}}(f)),
\end{aligned}
\end{equation*}
and $R_{D_{i}^{n}}\circ R_{D_{j}^{m}}=R_{D_{j}^{m}}\circ R_{D_{i}^{n}}$. The other two identities are obtained in the same way, noting that $A_{j}^{m}$ reproduces $f$ at leading order, so that $A_{j}^{m}A_{i}^{n}f=A_{j}^{m}f+A_{j}^{m}R_{A_{i}^{n}}(f)=f+R_{A_{j}^{m}}(f)+R_{A_{i}^{n}}(f)+R_{A_{j}^{m}}(R_{A_{i}^{n}}(f))$ and $D_{i}^{n}A_{j}^{m}f=A_{j}^{m}(\partial_{i}^{n}f+R_{D_{i}^{n}}(f))=\partial_{i}^{n}f+R_{A_{j}^{m}}(\partial_{i}^{n}f)+R_{D_{i}^{n}}(f)+R_{A_{j}^{m}}(R_{D_{i}^{n}}(f))$. Finally, the explicit expressions of the composed remainders follow by inserting the integral expressions of Proposition \ref{prop:difference:average} into each other and using Fubini's theorem, the translation in the direction $e_{i}$ (resp. $e_{j}$) being associated with the variable $\sigma$ (resp. $\sigma'$).
    
\end{proof}
\section{Asymptotic behavior of Carleman weight functions}\label{sec:asymp:weight}
This section is devoted to extending the results on discrete operations performed on the Carleman weight functions, presented in \cite{LOP-2020}, to arbitrary dimensions. We set the weight function of the form $r = e^{s\varphi}$ and $\rho = r^{-1}$, with $s\geq 1$ a parameter; where    $\varphi(x)=e^{\lambda\psi(x)}$ and the function $\psi$ fulfills suitable assumptions concerning the specific issue; that is, controllability or inverse problems. Our main objective is to obtain an asymptotic behavior describing the dependence on $s, h$ in the following basic estimates in domain $\Omega$ when the discrete operators are applied.
\par  Throughout this paper, $\mathcal{O}(1)$ denotes a function whose absolute value is bounded by a constant independent of $s$ (hence of $\tau$ and $\theta$), $\lambda$, $h$, $\Delta t$, $T$ and $\delta$, while $\mathcal{O}_{\lambda}(1)$ denotes a function bounded by a constant which may depend on $\lambda$, but is independent of $s$, $\tau$, $h$, $\Delta t$, $T$ and $\delta$. Both constants may depend on $d$, $\Omega$, $\psi$, $\mathrm{reg}(\gamma)$ and on the orders of the multi-indices and of the discrete operators involved. For a nonnegative quantity $X$, we write $\mathcal{O}(X):=X\,\mathcal{O}(1)$ and $\mathcal{O}_{\lambda}(X):=X\,\mathcal{O}_{\lambda}(1)$.
\par Let us recall that a multi-index has the form $\alpha:=(\alpha_{1},\ldots,\alpha_{d})\in\mathbb{N}^{d}_{0}$, where $\mathbb{N}_{0}:=\mathbb{N}\cup \{0\}$. Then, for $\alpha\in\mathbb{N}^{d}_{0}$,  and $\eta\in\mathbb{R}^{d}$ we write the order of the multi-index as $|\alpha|:=\alpha_{1}+\cdots+\alpha_{n}$, the compact notation of a higher-order operator as $
\partial^{\alpha}:=\partial_{x_{1}}^{\alpha_{1}}\cdots\partial_{x_{d}}^{\alpha_{d}}$ and we also adopt the notation
$\eta^{\alpha}:=\eta_{1}^{\alpha_{1}}\cdots\eta_{d}^{\alpha_{d}}$.

Now, we consider two fundamental estimates for our weight function. 
\begin{lemma}[{\cite[Lemma 3.7]{BHLR:2010}}] \label{lem:derivative:wrt:x}
Let $\alpha$ and $\beta$ multi-indices. We have
\begin{align*}
\partial^{\beta}(r\partial^{\alpha} \rho)=&|\alpha|^{|\beta|}(-s\varphi)^{|\alpha|}\lambda^{|\alpha+\beta|}(\nabla \psi)^{\alpha+\beta}+|\alpha||\beta|(s\varphi)^{|\alpha|}\lambda^{|\alpha+\beta|-1}\mathcal{O}(1)\\
&+s^{|\alpha|-1}|\alpha|(|\alpha|-1)\mathcal{O}_{\lambda}(1)\\
=&\mathcal{O}_{\lambda}(s^{|\alpha|}).
\end{align*}
Let $\sigma\in [-1,1]$ and $i\in\{1,\ldots,d\}$. We have
\begin{equation}
\begin{split}
\partial^{\beta}\left( r(x)(\partial^{\alpha}\rho)(x+\sigma he_{i})\right)=\mathcal{O}_{\lambda}(s^{|\alpha|}(1+(sh)^{|\beta|}))e^{\mathcal{O}_{\lambda}(sh)}.
\end{split}
\end{equation}
Provided $sh\leq1$ we have $\partial^{\beta}\left( r(x)(\partial^{\alpha}\rho)(x+\sigma he_{i})\right)=\mathcal{O}_{\lambda}(s^{|\alpha|})$. The same expressions hold with $r$ and $\rho$ interchanged and with $s$ changed to $-s$.
\end{lemma}
Another useful estimate is the following. 
\begin{corollary}[{\cite[Corollary 3.8]{BHLR:2010}}]\label{Cor:derivative:wrt:x:2r}
Let $\alpha,\beta$ and $\delta$ multi-indices. We have
\begin{equation}
\begin{split}
\partial^{\delta}\left(r^{2}\,\partial^{\alpha}\rho\, \partial^{\beta}\rho\right)=&|\alpha+\beta|^{|\delta|}(-s\varphi)^{|\alpha+\beta|}\lambda^{|\alpha+\beta+\delta|}(\nabla\psi)^{\alpha+\beta+\delta}\\
&+|\delta||\alpha+\beta|(s\varphi)^{|\alpha+\beta|}\lambda^{|\alpha+\beta+\delta|-1}\mathcal{O}(1)\\
&+s^{|\alpha+\beta|-1}\left( |\alpha|(|\alpha|-1)+|\beta|(|\beta|-1)\right)\mathcal{O}(1)\\
&=\mathcal{O}_{\lambda}(s^{|\alpha+\beta|}).
\end{split}
\end{equation}
\end{corollary}
The proofs of the above results can be found in \cite{BHLR:2010}. We aim to obtain a general estimate for the discrete operators applied to the Carleman weight function.
 Mimicking the notation for multi-indices in the continuous setting, we define, for two directions $i\ne j$ and for the two-dimensional index $k=(k_{i},k_{j})$, with $k_{i},k_{j}\in\mathbb{N}_{0}$, the discrete operator $\mathbf{D}_{h}^{k}:=D_{i}^{k_{i}}D_{j}^{k_{j}}$. Analogously, for the average operator, we define $\mathbf{A}_{h}^{l}:=A_{i}^{l_{i}}A_{j}^{l_{j}}$ with $l:=(l_{i},l_{j})$ being a two-dimensional multi-index. Considering the previous notation and Proposition \ref{prop:difference:average}, we can extend \cite[Proposition 4.5]{LOP-2020} to arbitrary dimensions.
\begin{proposition}\label{prop:weight}
    Let $\alpha$ be a multi-index and $k,l$ be two-dimensional multi-indices. Let $i,j\in\{1,\ldots,d\} $; provided that $sh\leq 1$, we have
    \begin{equation}\label{eq:average:difference:weight}
\begin{aligned}
r\mathbf{D}_{h}^{k}\mathbf{A}_{h}^{l}\partial^{\alpha}\rho &= r\partial^{k+\alpha}\rho + s^{|\alpha|+k_{i}}\mathcal{O}_{\lambda}((sh)^{2}) \\
&\quad + s^{|\alpha|+k_{j}}\mathcal{O}_{\lambda}((sh)^{2}) + s^{|\alpha+k|}\mathcal{O}_{\lambda}((sh)^{2}).
\end{aligned}
\end{equation}
\end{proposition}
\begin{proof}
Applying Corollary \ref{cor:difference:average} to $\mathbf{A}_{h}^{l}\partial^{\alpha}\rho$, we write
\begin{equation}\label{eq:estimation:DA}
\begin{aligned}
r\mathbf{D}_{h}^{k}\mathbf{A}_{h}^{l}\partial^{\alpha}\rho=& r\partial^{k}(\mathbf{A}_{h}^{l}\partial^{\alpha}\rho)+rR_{D^{k_{j}}_{j}}(\partial_{i}^{k_{i}}\mathbf{A}_{h}^{l}\partial^{\alpha}\rho)+rR_{D^{k_{i
}}_{i}}(\partial_{j}^{k_{j}}\mathbf{A}_{h}^{l}\partial^{\alpha}\rho)\\
&+rR_{\mathbf{D}_{h}^{k}}(\mathbf{A}_{h}^{l}\partial^{\alpha}\rho).
\end{aligned}
\end{equation}
Then, using the first identity of Corollary \ref{cor:difference:average}, it follows for the first term on the right-hand side of \eqref{eq:estimation:DA} that
\begin{align*}
r\partial^{k}(\mathbf{A}_{h}^{l}\partial^{\alpha}\rho)=&r\partial^{k+\alpha}\rho+rR_{A_{j}^{l_{j}}}(\partial^{\alpha+k}\rho)+rR_{A_{i}^{l_{i}}}(\partial^{\alpha+k}\rho)+rR_{\mathbf{A}_{h}^{l}}(\partial^{\alpha+k}\rho)
\end{align*}
Now, our next task is to estimate the last three terms on the right-hand side above. Using Corollary \ref{cor:difference:average} and Lemma \ref{lem:derivative:wrt:x} we obtain the estimates
\begin{align*}
    rR_{A_{j}^{l_{j}}}(\partial^{\alpha+k}\rho)=&s^{|\alpha|+|k|}\mathcal{O}_{\lambda}((sh)^{2}),\\
    rR_{A_{i}^{l_{i}}}(\partial^{\alpha+k}\rho)=&s^{|\alpha|+|k|}\mathcal{O}_{\lambda}((sh)^{2}),\\
    rR_{\mathbf{A}_{h}^{l}}(\partial^{\alpha+k}\rho)=&s^{|\alpha|+|k|}\mathcal{O}_{\lambda}((sh)^{4}),
\end{align*}
which  yields 
\begin{align}\label{eq:estimation}
    r\partial^{k}(\mathbf{A}_{h}^{l}\partial^{\alpha}\rho)=r\partial^{k+\alpha}\rho+s^{|\alpha|+|k|}\mathcal{O}_{\lambda}((sh)^{2}),
\end{align} 
where we have used $(sh)^{4}\leq (sh)^{2}$, since $sh\leq 1$. Similarly, since $R_{D_{j}^{k_{j}}}(\partial_{i}^{k_{i}}\mathbf{A}_{h}^{l}\partial^{\alpha}\rho)$ is linear combination of terms of the form $h^{2}\partial_{j}^{k_{j}}\partial_{i}^{k_{i}}\partial^{\alpha}\rho$ evaluated at points $x+\eta$ with $|\eta|\leq Ch$, we have the estimates
\begin{equation}\label{eq:estimations}
\begin{aligned}
    rR_{D^{k_{j}}_{j}}(\partial_{i}^{k_{i}}\mathbf{A}_{h}^{l}\partial^{\alpha}\rho)=&s^{|k|+|\alpha|}\mathcal{O}_{\lambda}((sh)^{2}),\\
    rR_{D^{k_{i}}_{i}}(\partial_{j}^{k_{j}}\mathbf{A}_{h}^{l}\partial^{\alpha}\rho)=&s^{|k|+|\alpha|}\mathcal{O}_{\lambda}((sh)^{2}),\\
    rR_{\mathbf{D}_{h}^{k}}(\mathbf{A}_{h}^{l}\partial^{\alpha}\rho)=&s^{|k|+|\alpha|}\mathcal{O}_{\lambda}((sh)^{4}).
\end{aligned}
\end{equation}
Thus, combining the estimates \eqref{eq:estimations} and \eqref{eq:estimation} in \eqref{eq:estimation:DA} we obtain \eqref{eq:average:difference:weight}.
\end{proof}

Another useful estimate is the following lemma.
\begin{lemma}\label{lem:difference:average:2}
Let $\alpha$, $\beta$ and $\delta$ be multi-indices and let $k,l$ be two-dimensional multi-indices. Provided $sh\leq 1$, we have  for $i\ne j$,
\begin{equation}\label{lem:difference:average}
\mathbf{D}_{h}^{k}\mathbf{A}_{h}^{l}\left(\partial^{\beta}\left( r\partial^{\alpha}\rho\right) \right)=\partial_{i}^{k_{i}}\partial_{j}^{k_{j}}\partial^{\beta}\left(r\partial^{\alpha}\rho \right)+h^{2}\mathcal{O}_{\lambda}(s^{|\alpha|}).
\end{equation}
Moreover, under the same condition $sh\leq 1$ we have
\begin{equation}\label{eq0:lem:difference:average:3}
\mathbf{D}_{h}^{k}\mathbf{A}_{h}^{l}\partial^{\delta}\left(r^{2}\left(\partial^{\alpha}\rho\right)\partial^{\beta}\rho\right)=\partial^{k+\delta}\left(r^{2}(\partial^{\alpha}\rho)\partial^{\beta}\rho \right)+h^{2}\mathcal{O}_{\lambda}(s^{|\alpha|+|\beta|}),
\end{equation}
and, still for $sh\leq 1$ and for $\sigma,\sigma'\in[0,1]$ we get
\begin{align}
\mathbf{D}^{k}_{h}\mathbf{A}_{h}^{l}\partial^{\beta}\left(r(x)\partial^{\alpha}\rho(x+\sigma he_{i
} \right))&=\mathcal{O}_{\lambda}(s^{|\alpha|}),\label{eq1:lem:difference:average}\\
\mathbf{D}_{h}^{k}\mathbf{A}_{h}^{l}\partial^{\delta}\left(r^{2}(x)\left( \partial^{\alpha}\rho(x+\sigma he_{i})\right)\partial^{\beta}\rho(x+\sigma' he_{j})\right)&=\mathcal{O}_{\lambda}(s^{|\alpha|+|\beta|}).\label{eq2:lem:difference:average}
\end{align}
\end{lemma}

\begin{proof} The method of proof is analogous to that used in demonstrating Proposition \ref{prop:weight}. Specifically, Corollary \ref{cor:difference:average} is applied to $\partial^{\beta}(r\partial^{\alpha}\rho)$, followed by the application of Lemma \ref{lem:derivative:wrt:x} to establish result \eqref{lem:difference:average}. In a similar fashion, equations \eqref{eq0:lem:difference:average:3} and \eqref{eq2:lem:difference:average} are derived using the same approach, with the conclusion utilizing Corollary \ref{Cor:derivative:wrt:x:2r}.
\end{proof}
Now, we state the following Theorem that provides us an estimate for several computations of the discrete average and derivative operators applied to our Carleman weight function. 
\begin{theorem}\label{theo:weight:estimates}
Let $\alpha$ be a multi-index and $l,k,n,m,p$ and $q$ be two-dimensional multi-indices. For $i,j\in\{ 1,\ldots, d\}$, and for $sh\leq 1$, we have
\begin{equation}\label{eq1:theo:weight:estimates}
\begin{aligned}
       \mathbf{A}_{h}^{l}\mathbf{D}_{h}^{k}\partial^{\alpha}(r\mathbf{A}_{h}^{m}\mathbf{D}_{h}^{n}\rho)=&\partial^{k+\alpha}(r\partial^{n}\rho)+s^{|\alpha|+|n|}\mathcal{O}_{\lambda}((sh)^{2})
\end{aligned}       
\end{equation}
and 
\begin{equation}\label{eq2:theo:weight:estimates}
\begin{aligned}
\mathbf{A}_{h}^{l}\mathbf{D}_{h}^{k}\partial^{\beta}(r^{2}\mathbf{A}^{p}_{h}\mathbf{D}^{q}_{h}(\partial^{\alpha}\rho)\mathbf{A}_{h}^{m}\mathbf{D}_{h}^{n}\rho)=&\partial^{k+\beta}(r^{2}(\partial^{q+\alpha}\rho)\partial^{n}\rho)+s^{|\alpha+n+q|}\mathcal{O}_{\lambda}((sh)^{2}).
\end{aligned}
\end{equation}
\end{theorem}
\begin{proof}
From Corollary \ref{cor:difference:average} we write
\begin{equation*}
    r\mathbf{A}_{h}^{m}\mathbf{D}_{h}^{n}\rho= r\partial^{n}(\mathbf{A}_{h}^{m}\rho)+rR_{D^{n_{j}}_{j}}(\partial_{i}^{n_{i}}\mathbf{A}_{h}^{m}\rho)+rR_{D^{n_{i}}_{i}}(\partial_{j}^{n_{j}}\mathbf{A}_{h}^{m}\rho)+rR_{\mathbf{D}_{h}^{n}}(\mathbf{A}_{h}^{m}\rho).
\end{equation*}
Then, applying the operator $\mathbf{A}_{h}^{l}\mathbf{D}_{h}^{k}\partial^{\alpha}$ we have
\begin{align}\label{eq:prop:difference}
    \mathbf{A}_{h}^{l}\mathbf{D}_{h}^{k}\partial^{\alpha}  (r\mathbf{A}_{h}^{m}\mathbf{D}_{h}^{n}\rho)=&  \mathbf{A}_{h}^{l}\mathbf{D}_{h}^{k}\partial^{\alpha}(r\partial^{n}(\mathbf{A}_{h}^{m}\rho))+ \mathbf{A}_{h}^{l}\mathbf{D}_{h}^{k}\partial^{\alpha}(rR_{D^{n_{j}}_{j}}(\mathbf{A}_{h}^{m}\rho))\notag\\
    &+ \mathbf{A}_{h}^{l}\mathbf{D}_{h}^{k}\partial^{\alpha}(rR_{D^{n_{i}}_{i}}(\mathbf{A}_{h}^{m}\rho))+ \mathbf{A}_{h}^{l}\mathbf{D}_{h}^{k}\partial^{\alpha}(rR_{\mathbf{D}_{h}^{n}}(\mathbf{A}_{h}^{m}\rho)).
\end{align}
Let us focus on the term $\mathbf{A}_{h}^{l}\mathbf{D}_{h}^{k}\partial^{\alpha}(r\partial^{n}(\mathbf{A}_{h}^{m}\rho))$. We notice that, for this term, using the first equation from Corollary \ref{cor:difference:average} applied to $\mathbf{A}_{h}^{m}\rho$ gives
\begin{align*}
    \mathbf{A}_{h}^{l}\mathbf{D}_{h}^{k}\partial^{\alpha}(r\partial^{n}(\mathbf{A}_{h}^{m}\rho))=&\mathbf{A}_{h}^{l}\mathbf{D}_{h}^{k}\partial^{\alpha}(r\partial^{n}\rho)+\mathbf{A}_{h}^{l}\mathbf{D}_{h}^{k}\partial^{\alpha}(r\partial^{n}R_{A_{i}^{m_{i}}}(\rho))\\
    &+\mathbf{A}_{h}^{l}\mathbf{D}_{h}^{k}\partial^{\alpha}(r\partial^{n}R_{A_{j}^{m_{j}}}(\rho))+\mathbf{A}_{h}^{l}\mathbf{D}_{h}^{k}\partial^{\alpha}(r\partial^{n}R_{\mathbf{A}_{h}^{m}}(\rho)).
    \end{align*}
    Thus, applying Proposition \ref{prop:weight} to the first term from the right-hand side above, and then \eqref{eq1:lem:difference:average} from Lemma \ref{lem:difference:average:2} we can estimate 
    \begin{align*}
    \mathbf{A}_{h}^{l}\mathbf{D}_{h}^{k}\partial^{\alpha}(r\partial^{n}(\mathbf{A}_{h}^{m}\rho))=&\partial^{k+\alpha}(r\partial^{n}\rho)+s^{|n|}\mathcal{O}_{\lambda}((sh)^{2}).
    \end{align*}
     Indeed, by Lemma \ref{lem:derivative:wrt:x} and \eqref{eq1:lem:difference:average}, the discrete and continuous derivatives of $r\partial^{\gamma}\rho$, evaluated with shifts of size $\mathcal{O}(h)$, are $\mathcal{O}_{\lambda}(s^{|\gamma|})$, so that the factor $s^{|\alpha|}$ does not appear; the terms involving $R_{A_{i}^{m_{i}}}$, $R_{A_{j}^{m_{j}}}$ and $R_{\mathbf{A}_{h}^{m}}$ are $h^{2}\mathcal{O}_{\lambda}(s^{|n|+2})=s^{|n|}\mathcal{O}_{\lambda}((sh)^{2})$, and \eqref{lem:difference:average} gives $\mathbf{A}_{h}^{l}\mathbf{D}_{h}^{k}\partial^{\alpha}(r\partial^{n}\rho)=\partial^{k+\alpha}(r\partial^{n}\rho)+h^{2}\mathcal{O}_{\lambda}(s^{|n|})$. Similarly, we can follow the same steps above to estimate the last three terms of \eqref{eq:prop:difference}, which are of order $\mathcal{O}_{\lambda}((sh)^{2})$ and $\mathcal{O}_{\lambda}((sh)^{4})$, and conclude the proof of \eqref{eq1:theo:weight:estimates}. To prove \eqref{eq2:theo:weight:estimates}, we note that thanks to Corollary \ref{cor:difference:average} we write
   \begin{equation*}
    r\mathbf{A}_{h}^{m}\mathbf{D}_{h}^{n}\rho= r\partial^{n}(\mathbf{A}_{h}^{m}\rho)+rR_{D^{n_{j}}_{j}}(\mathbf{A}_{h}^{m}\rho)+rR_{D^{n_{i}}_{i}}(\mathbf{A}_{h}^{m}\rho)+rR_{\mathbf{D}_{h}^{n}}(\mathbf{A}_{h}^{m}\rho)
\end{equation*}
and 
\begin{equation*}
    r\mathbf{A}_{h}^{p}\mathbf{D}_{h}^{q}(\partial^{\alpha}\rho)= r\partial^{q}(\mathbf{A}_{h}^{p}\partial^{\alpha}\rho)+rR_{D^{q_{j}}_{j}}(\partial_{i}^{q_{i}}\mathbf{A}_{h}^{p}\partial^{\alpha}\rho)+rR_{D^{q_{i}}_{i}}(\partial_{j}^{q_{j}}\mathbf{A}_{h}^{p}\partial^{\alpha}\rho)+rR_{\mathbf{D}_{h}^{q}}(\mathbf{A}_{h}^{p}\partial^{\alpha}\rho).
\end{equation*}
Thus, multiplying the above expressions by $\mathbf{A}_{h}^{m}\mathbf{D}_{h}^{n}\rho$ and arguing as in the proof of \eqref{eq1:theo:weight:estimates} the estimate of \eqref{eq2:theo:weight:estimates} follows.
\end{proof}
\begin{remark}
We remark that the estimates in this section remain valid if we interchange $\rho$ and $r$. This is because Lemma \ref{lem:derivative:wrt:x} and Corollary \ref{Cor:derivative:wrt:x:2r} still hold with this change, and we can replace $s$ by $-s$ in these statements. 
\end{remark}

It is worth noting that the previous results can be extended to the case where the weight function is time-dependent. For instance, if we consider a weight function of the form $r(x,t)=e^{\tau\theta(t)\varphi(x)}$, the condition $sh\leq 1$ needs to be replaced by $\tau h(\max_{[0,T]}\theta(t))\leq 1$, which implies $\tau\theta(t)h\leq 1$. This is typically the case for parabolic operators where the Carleman weight function has the form $e^{\tau\theta(t)\varphi(x)}$ and the function $\theta(t)$ is given by 
\begin{equation}\label{theta-delta}
    \theta(t)=\frac{1}{(t+\delta T)(T+\delta T-t)},\quad t\in [0,T],
\end{equation}
for some $0<\delta \leq 1/2$. The parameter $\delta$ is chosen to avoid singularities at time $t=0$ and $t=T$ since $\theta$ must be bounded on the interval $[0,T]$ to obtain an estimate similar to Theorem \ref{theo:weight:estimates} for the time-dependent case. Moreover,  we collect here the bounds on $\theta$ used throughout the paper; for $t\in [0,T]$ and $j\in\mathbb{N}_{0}$,
\begin{equation}\label{eq:theta} 
\begin{gathered}\frac{1}{T^{2}}\leq \frac{4}{(1+2\delta)^{2}T^{2}}\leq \theta(t)\leq \max_{t\in [0,T]} \theta(t)=\theta(0)=\theta(T)=\frac{1}{T^{2}\delta(1+\delta)}\leq \frac{1}{T^{2}\delta},\\
|\theta'(t)|\leq T\theta^{2}(t),\quad |\theta''(t)|\leq 4T^{2}\theta^{3}(t),\quad|\theta^{(j)}(t)|\leq \frac{2j!}{\delta^{j+1}T^{j+2}}.
\end{gathered}
\end{equation}
Thus, the condition $sh\leq 1$ for $\theta$ as in \eqref{theta-delta} is replaced by $\tau h(T^{2}\delta)^{-1}\leq 1$. This situation is documented in the works \cite{boyer-2014,GC-HS-2021,LMPZ:2023}. Therefore, Theorem \ref{theo:weight:estimates} is a generalization of the estimates presented in those works.
 
\par For the aforementioned Carleman weight function, we obtain, in the case of time derivative, the following asymptotic behavior, which stands for a generalization into an arbitrary dimension of Theorem 2.9 from \cite{CLTP-2022}.  
\begin{theorem}\label{theo:time:derivate}
Let $\alpha$ be a multi-index and let $m,n,j,k$ be two-dimensional multi-indices.  Let $r(x,t)=e^{s(t)\varphi(x)}$ and $\rho=r^{-1}$, with $s(t)=\tau\theta(t)$, $\tau\geq 1$ and $\theta$ given by \eqref{theta-delta}. Then, provided $\tau h(\delta T^{2})^{-1}\leq 1$, we have
\begin{equation}\label{eq:temporal:estimate}
    \begin{aligned}\partial_{t}\left( r \mathbf{D}_{h}^{n}\mathbf{A}_{h}^{m}\partial^{\alpha}\rho\right)=&\partial_{t}\left(r\partial^{n+\alpha}\rho \right)+Ts^{|\alpha|+|n|}\theta\mathcal{O}_{\lambda}((sh)^{2}),
    \end{aligned}
\end{equation}
and
\begin{equation}\label{eq:temporal:estimation}
\begin{split}
    \partial_{t} \mathbf{D}_{h}^{k}\mathbf{A}_{h}^{j}\left(r\mathbf{D}_{h}^{n}\mathbf{A}_{h}^{m}\partial^{\alpha}\rho\right)
    =&T\theta s^{|\alpha|+|n|}\mathcal{O}_{\lambda}(1).
    \end{split}
\end{equation}
\end{theorem}
\begin{proof}
We begin as in the proof of Theorem \ref{theo:weight:estimates}, by using Corollary \ref{cor:difference:average} we have 
\begin{equation*}
\begin{aligned}
   \partial_{t}(r\mathbf{D}_{h}^{n}\mathbf{A}_{h}^{m}\partial^{\alpha}\rho)=&\partial_{t}(r\partial^{n}(\mathbf{A}_{h}^{m}\partial^{\alpha}\rho))+\partial_{t}(rR_{D^{n_{j}}_{j}}(\partial_{i}^{n_{i}}\mathbf{A}_{h}^{m}\partial^{\alpha}\rho))\\
   &+\partial_{t}(rR_{D^{n_{i}}_{i}}(\partial_{j}^{n_{j}}\mathbf{A}_{h}^{m}\partial^{\alpha}\rho))+\partial_{t}(rR_{\mathbf{D}_{h}^{n}}(\mathbf{A}_{h}^{m}\partial^{\alpha}\rho)).
   \end{aligned}
\end{equation*}
Note that applying Corollary \ref{cor:difference:average} to $\partial_{t}(r\partial^{n}(\mathbf{A}_{h}^{m}\partial^{\alpha}\rho))$ we write
\begin{equation}\label{estimate:theo:time1}
\begin{aligned}
    \partial_{t}(r\partial^{n}(\mathbf{A}_{h}^{m}\partial^{\alpha}\rho))=&\partial_{t}(r\partial^{n+\alpha}\rho)+\partial_{t}(rR_{A_{i}^{m_{i}}}(\partial^{n+\alpha}\rho))\\
    &+\partial_{t}(rR_{A_{j}^{m_{j}}}(\partial^{n+\alpha}\rho))+\partial_{t}(rR_{\mathbf{A}_{h}^{m}}(\partial^{n+\alpha}\rho)).
    \end{aligned}
\end{equation}
Thus, our next task is to estimate the last three terms from the right-hand side above. First, focusing on $\partial_{t}(rR_{A_{i}^{m_{i}}}(\partial^{n+\alpha}\rho))$, thanks to the first estimate from \cite[Proposition 2.14]{boyer-2014} it follows that
\begin{equation}\label{eq:coro:rho}
\partial_{t}(r(x,t)(\partial_{i}^{2}\partial^{n+\alpha}\rho)(x+(m_{i}-2k)e_{i}h\sigma/2,t))=Ts^{|\alpha|+|n|+2}\theta(t)\mathcal{O}_{\lambda}(1).
\end{equation}
Then, with the above estimate and the asymptotic expansion of $R_{A_{i}^{m_{i}}}$ given by Proposition \ref{prop:difference:average} we obtain $\partial_{t}(rR_{A_{i}^{m_{i}}}(\partial^{n+\alpha}\rho))=Ts^{|\alpha|+|n|
}\theta(t)\mathcal{O}_{\lambda}((sh)^{2})$. A similar procedure for the last two terms on the right-hand side of \eqref{estimate:theo:time1} allows to estimate
\begin{align*}
    \partial_{t}(rR_{A_{j}^{m_{j}}}(\partial^{n+\alpha}\rho))=&Ts^{|\alpha|+|n|
}\theta(t)\mathcal{O}_{\lambda}((sh)^{2}),\\
    \partial_{t}(rR_{\mathbf{A}_{h}^{m}}(\partial^{n+\alpha}\rho))=&Ts^{|\alpha|+|n|
}\theta(t)\mathcal{O}_{\lambda}((sh)^{4}).
\end{align*}
Therefore, the previous last three estimates, together with $(sh)^{4}\leq (sh)^{2}$,  yield that \eqref{estimate:theo:time1} can be estimated as 
\begin{equation*}
    \partial_{t}(r\partial^{n}(\mathbf{A}_{h}^{m}\partial^{\alpha}\rho))= \partial_{t}(r\partial^{n+\alpha}\rho)+Ts^{|\alpha|+|n|
}\theta(t)\mathcal{O}_{\lambda}((sh)^{2}).
\end{equation*}

Similarly,  since the remainders of Corollary \ref{cor:difference:average} are linear combinations of $h^{2}\partial_{j}^{n_{j}+2}\partial_{i}^{n_{i}}\partial^{\alpha}\rho$, $h^{2}\partial_{i}^{n_{i}+2}\partial_{j}^{n_{j}}\partial^{\alpha}\rho$ and $h^{4}\partial_{i}^{n_{i}+2}\partial_{j}^{n_{j}+2}\partial^{\alpha}\rho$ evaluated at shifted points, the estimate \eqref{eq:coro:rho} gives
\begin{align*}
    \partial_{t}(rR_{D^{n_{j}}_{j}}(\partial_{i}^{n_{i}}\mathbf{A}_{h}^{m}\partial^{\alpha}\rho))=&Ts^{|\alpha|+|n|}\theta(t)\mathcal{O}_{\lambda}((sh)^{2}),\\
   \partial_{t}(rR_{D^{n_{i}}_{i}}(\partial_{j}^{n_{j}}\mathbf{A}_{h}^{m}\partial^{\alpha}\rho))=&Ts^{|\alpha|+|n|}\theta(t)\mathcal{O}_{\lambda}((sh)^{2}),\\
   \partial_{t}(rR_{\mathbf{D}_{h}^{n}}(\mathbf{A}_{h}^{m}\partial^{\alpha}\rho))=&Ts^{|\alpha|+|n|}\theta(t)\mathcal{O}_{\lambda}((sh)^{2}).
\end{align*}
Combining these last four estimates, we establish \eqref{eq:temporal:estimate}.\\
The estimate \eqref{eq:temporal:estimation} can be obtained following a similar strategy.
\end{proof}
\begin{remark}
    Theorem \ref{theo:time:derivate} is a generalization of the second and third estimates from \cite[Proposition 2.14]{boyer-2014} and \cite{HS:2023}.  
\end{remark}

\begin{remark}\label{rmk:1} In the study of inverse problem for  hyperbolic operators \cite{BEO-2015,BE-2013} and \cite{B-B-E-O:2021} 
 is considered the following Carleman weight function $ e^{\lambda\hat{\varphi}}$, with
 \begin{equation}\label{eq:Carleman:wave}
\hat{\psi}(t,x) = |x - x^{\ast}|^2 - \beta t^2 + c_0, \quad \hat{\varphi}(t,x) = e^{\mu \psi(t,x)},
\end{equation}
where $x^{\ast}\notin [0,1]^{2}$, and $c_0 > 0$ is such that $\hat{\psi} \geq 1$ in $[-T,T] \times [0,1]^2$, $\beta\in(0,1)$, $\lambda\geq 1$ is the large parameter and $\mu \geq 1$ is a suitable parameter. Although $\hat{\psi}$ mixed the variables $t$ and $x$, the weight has the separated form of this section, since the exponential factorizes
\begin{equation*}
    e^{\lambda\hat{\varphi}},\quad s(t):=\lambda\tilde{\theta}(t),\quad\tilde{\theta}:=e^{-\mu\beta t^{2}},\quad\varphi(x):=e^{\mu(|x-x^{\ast}|+c_{0})},
\end{equation*}
where $\mu$ plays the role of the parameter $\lambda$ of this section and $\lambda$ that of $\tau$. Hence, since $0<\tilde{\theta}\leq 1$, Proposition \ref{prop:weight} and Theorem \ref{theo:weight:estimates} apply for each fixed $t$, under the condition $\lambda h\leq 1$. However, Theorem \ref{theo:time:derivate} cannot be applied as stated. The obstruction is neither the mixing of $t$ and $x$ in $\hat{\psi}$ nor the boundedness of the time factor, but the structure of the time-derivative error terms, that is, the factor $T\theta(t)$ in \eqref{eq:temporal:estimate} comes from $|\theta'/\theta|\leq T\theta$, a property of the parabolic function \eqref{theta-delta} (see \eqref{eq:theta}), which is used through \eqref{eq:coro:rho}. For $\tilde{\theta}$ one has instead $|\tilde{\theta}'/\tilde{\theta}|=2\mu\beta|t|$, which is not controlled by $T\tilde{\theta}$ uniformly in $\mu$. The time-derivative estimates for these weights must therefore be derived separately, with error term involving $\mu\beta|t|$ instead of $T\theta$. See also \cite{ZZ:2022} and \cite{ZZ:2023}, where a similar weight function is considered to study the coefficient inverse problems for damped wave and Schr\"{o}dinger equations, respectively. 
\end{remark}
\begin{remark}
In the same direction as the aforementioned remark, a Carleman weight function like \eqref{eq:Carleman:wave} is considered in \cite{WZ:2024} in the study of an inverse problem for a semi-discrete stochastic parabolic operator is considered, a Carleman weight function like \eqref{eq:Carleman:wave}. For this reason, the spatial estimates established above provide a useful tool to extend the results from \cite{WZ:2024} to arbitrary dimensions.
\end{remark}

Similar to the previous results, it is possible to obtain an asymptotic expansion considering the time and spatial discrete operators. The main task is to generalize to arbitrary dimensions the asymptotic behavior obtained in \cite{LMPZ:2023}. Now, we will introduce some notations to define the discretization of the time variable. To do this, we recall $\Delta t= T/N$, with $N\in \mathbb{N}$. We define
\begin{align*}
    \mathcal{N}=&\{j\Delta t\,;\, j\in \{ 1,\ldots, N\}\}, \quad \overline{\mathcal{N}}=\mathcal{N}\cup \{0\},\\
    \mathcal{N}^{\ast}=&\{(j-1/2)\Delta t\,;\, j\in \{1,\ldots,N\}\},\quad \overline{\mathcal{N}}^*=\mathcal{N}^{\ast} \cup \{T+\Delta t/2\}. 
\end{align*}
We define the time discrete derivative of a function $y:[0,T]\to \mathbb{R}$ sampled on $\mathcal{N}$ as follows
\begin{align*}
    D_{t} y:=\dfrac{\trp    y-\trm   y}{\Delta t}\quad \text{in } \mathcal{N}^{\ast},
\end{align*}
where the translation in time is defined by $\trpm y(t):=y(t\pm \Delta t/2)$, with $t\in \mathcal{N}$. For the time-discrete operator  we have the following asymptotic behavior
\begin{lemma}\label{lem:time:discrete:weight}
    Let $p\in\mathbb{N}_{0}$, $\tau\geq 1$, $s(t)=\tau\theta(t)$ with $\theta$ given by \eqref{theta-delta}, and let $\varphi$ be bounded. Provided $\Delta t\tau(T^{3}\delta^{2})^{-1}\leq 1$, we have
    \begin{equation}\label{eq:theta:D:t}
        \trm(\theta^{p}e^{s\varphi})D_{t}(\theta^{-p}e^{-s\varphi})=\trm(\theta^{p}e^{s\varphi}\partial_{t}(\theta^{-p}e^{-s\varphi}))+\mathcal{E}(\Delta t) \mathcal{O}_{\lambda}(1),
    \end{equation}
    where $\mathcal{E}(\Delta t):=\mathcal{E}(\Delta t, p,\delta,\tau,T):=p(p+1)\frac{\Delta t}{\delta^{p+4}T^{2}}+p\frac{\Delta t}{\delta^{p+3}T^{2}}
    +2p\tau\frac{\Delta t}{\delta^{p+4}T^{4}}+\tau^{2} \frac{\Delta t}{\delta^{p+4}T^{6}}+\tau\frac{\Delta t}{\delta^{p+3}T^{4}}$.
\end{lemma}
\begin{proof}
We use the first-order Taylor formula for  $f:=\theta^{-p}e^{-s\varphi}$ to write
\begin{equation}\label{eq:time:taylor}
    \trp D_{t}(f)=\partial_{t}f+\Delta t\int_{0}^{1}(1-\gamma)\partial_{t}^{2}f(t+\gamma \Delta t)d\gamma.
\end{equation}
Let us focus on the integral remainder of the above expression. We have
\begin{equation*}
\begin{aligned}
    \partial_{t}^{2}f=&(p(p+1)\theta^{-p-2}(\partial_{t}\theta)^{2}-p\theta^{-p-1}\partial_{t}^{2}\theta)e^{-s\varphi}+2\tau\varphi d\theta^{-p-1}(\partial_{t}\theta)^{2}e^{-s\varphi}\\
    &+\theta^{-p}\left[ \tau^{2}\varphi^{2}(\partial_{t}\theta)^{2}e^{-s\varphi}-\tau\partial_{t}^{2}\theta\varphi e^{-s\varphi}\right].
      \end{aligned}
\end{equation*}
We notice that there exists a constant $C>0$, only depending on $p$, such that
\begin{equation}\label{ine:theta:d:-d}
    \max_{t\in [0,T]} \theta^{p}(t)\theta^{-p}(t+\gamma\Delta t)\leq \frac{C}{\delta^{p}}.
\end{equation} Indeed, for $p\in\{0,1,2,\ldots\}$ it follows that
\begin{equation}\label{ine:max:theta}
    \max_{t\in[0,T]}\theta^{-p}(t+\gamma\Delta t)\leq   \frac{T^{2p}(1+2\delta)^{2p}}{4^{p}}\leq T^{2p},
    \end{equation}
    and 
    \begin{equation}\label{ine:theta:d}
    \max_{t\in[0,T]} \theta^{p}(t)\leq \frac{C}{\delta^{p}T^{2p}}.
\end{equation} 
So \eqref{ine:theta:d:-d} holds. And we also have
\begin{equation}\label{ine:theta:d:2}
    \max_{t\in[0,T]}\theta^{(j)}(t)\leq \frac{2!}{\delta^{j+1}T^{j+2}},\quad j=0,1,\ldots
\end{equation}
Thus, multiplying \eqref{eq:time:taylor} by $\theta^{p}e^{s\varphi}$, and then applying \eqref{ine:theta:d:-d}-\eqref{ine:theta:d:2} to the result we obtain
\begin{equation}
\begin{aligned}
    \theta^{p}e^{s\varphi}\trp D_{t}(\theta^{-p}e^{-s\varphi})=&\theta^{p}e^{s\varphi}\partial_{t}(\theta^{-p}e^{-s\varphi})+p(p+1)\frac{\Delta t}{\delta^{p+4}T^{2}}\mathcal{O}_{\lambda}(1)+p\frac{\Delta t}{\delta^{p+3}T^{2}}\mathcal{O}_{\lambda}(1)\\
    &+2p\tau\frac{\Delta t}{\delta^{p+4}T^{4}}\mathcal{O}_{\lambda}(1)+\tau^{2} \frac{\Delta t}{\delta^{p+4}T^{6}}\mathcal{O}_{\lambda}(1)+\tau\frac{\Delta t}{\delta^{p+3}T^{4}}\mathcal{O}_{\lambda}(1),
\end{aligned}
\end{equation}
where the conditions $\Delta t\tau(\delta^{2} T^{3})^{-1}\leq 1$ and $\|\varphi\|_{\infty}=\mathcal{O}_{\lambda}(1)$ are used to control the exponential term $e^{\tau(\theta(t)-\theta(t+\gamma\Delta t))\varphi}=\mathcal{O}_{\lambda}(1)$. Then, by shifting the above expression using $\trm$ the estimate \eqref{eq:theta:D:t} holds.
\end{proof}
The shift $\trm$ can be removed from the right-hand side of \eqref{eq:theta:D:t}. Indeed, setting $g:=\theta^{p}e^{s\varphi}\partial_{t}(\theta^{-p}e^{-s\varphi})=-p\theta'/\theta-\tau\theta'\varphi$, we have $|g(t)-\trm g(t)|\leq\frac{\Delta t}{2}\max|g'|$ with $g'=-p(\theta''/\theta-(\theta')^{2}/\theta^{2})-\tau\theta''\varphi$, and \eqref{eq:theta} gives $|g'|\leq C\big(p\,\delta^{-2}T^{-2}+\tau\delta^{-3}T^{-4}\big)\mathcal{O}_{\lambda}(1)$. These terms are bounded by the second and last terms of $\mathcal{E}(\Delta t)$, since $\delta\leq1$. Hence \eqref{eq:theta:D:t} also holds with $\trm(\theta^{p}e^{s\varphi}\partial_{t}(\theta^{-p}e^{-s\varphi}))$ replaced by $\theta^{p}e^{s\varphi}\partial_{t}(\theta^{-p}e^{-s\varphi})$, with the same form of $\mathcal{E}(\Delta t)$ up to a constant; the same observation applies to \eqref{eq:theta:D:t:general}. We keep the shifted form, since it is the one used in \eqref{eq:time:conjugate}.
\begin{remark}
    Following the same strategy, let $p\geq 1$ and $l\geq 0$ be integers and $s(t)=\tau\theta^{p}(t)$, where $\tau\geq 1$. Provided 
    \begin{equation}\label{eq:hypothesis:general}
    \Delta t\,\tau\,(\delta^{p+1}T^{2p+1})^{-1}\leq 1,
    \end{equation}
    we have
    \begin{equation}\label{eq:theta:D:t:general}
        \trm(\theta^{l}e^{s\varphi})D_{t}(\theta^{-l}e^{-s\varphi})=\trm(\theta^{l}e^{s\varphi}\partial_{t}(\theta^{-l}e^{-s\varphi}))+\mathcal{E}(\Delta t) \mathcal{O}_{\lambda}(1),
    \end{equation}
    where $\mathcal{E}(\Delta t):=\mathcal{E}(\Delta t, p, l,\delta,\tau,T)$ is given by
    \begin{equation*}
    \begin{aligned}
    \mathcal{E}(\Delta t):=\Delta t\bigg(&\frac{l(l+1)}{\delta^{l+4}T^{2}}+\frac{l}{\delta^{l+3}T^{2}}+\frac{C(l,p)\,\tau}{\delta^{\kappa(l,p)}T^{2p+2}}+\frac{p^{2}\tau^{2}}{\delta^{l+2p+2}T^{4p+2}}\bigg),
    \end{aligned}
    \end{equation*}
    with $C(l,p):=|2lp-p(p-1)|+p$ and $\kappa(l,p):=l+4$ if $p=1$, $\kappa(l,p):=l+p+2$ if $p\geq 2$.\\
    Moreover, \eqref{eq:theta:D:t:general} generalizes the estimates \cite[Lemma B.4]{BHS:2020}.
\end{remark}

Thus, we will focus on the Carleman weight function of the form $r(x,t):=e^{s(t)\varphi(x)}$ with $\theta$ defined in \eqref{theta-delta}, $\rho=:r^{-1}$ and $s(t)=\tau\theta(t)$.
In \cite{GC-HS-2021} an asymptotic behavior of the Carleman weight function used for the heat equation is established in the one-dimensional setting with limited order in the space-discrete operators.
Moreover, in \cite[Theorem 5]{LMPZ:2023} this limitation is overcome with respect to the order of the spatial operator, although still in the 1-D case. Thus, our task now is to generalize the aforementioned results to arbitrary dimensions. We emphasize that the methodology of the next result adapts the proof of \cite[Lemma A.16]{GC-HS-2021} and \cite[Theorem 5]{LOP-2020}.
\begin{theorem}\label{theo:fully:weight:estimates}
Let $\alpha$ and $m,n$ be a multi-index and two-dimensional multi-indices in the space variable, respectively. Let $r=e^{s(t)\varphi}$ and $\rho=r^{-1}$, with $s(t)=\tau\theta(t)$ and $\theta$ given by \eqref{theta-delta}. Provided $\tau h(\delta T^{2})^{-1}\leq 1$, $\Delta t \tau(T^{3}\delta^{2})^{-1}\leq 1/2$ and $\sigma$ bounded, we have
\begin{equation}\label{eq:prop:time:estimate}
    \begin{aligned}
    D_{t}(r\mathbf{D}_{h}^{n}\mathbf{A}_{h}^{m}(\partial^{\alpha}\rho))=&\partial_{t}(r\partial^{n+\alpha}\rho)+T(sh)^{2}s^{|n+\alpha|}\theta\mathcal{O}_{\lambda}(1)\\
    &+\Delta t \trp((sh)^{2}T^{2}s^{|n+\alpha|}\theta^{2})\mathcal{O}_{\lambda}(1)\\
    &+\Delta t T^{2}\trp(s^{|n+\alpha|}\theta^{2})\mathcal{O}_{\lambda}(1).
\end{aligned}
\end{equation}
\end{theorem}
\begin{proof}
   To prove \eqref{eq:prop:time:estimate} we follow a similar strategy developed in the proof of Theorem \ref{theo:weight:estimates}. Thanks to Corollary \ref{cor:difference:average} we write
\begin{equation}\label{estimate:theo:fully}
\begin{aligned}
    r\mathbf{D}_{h}^{n}\mathbf{A}_{h}^{m}\partial^{\alpha}\rho=& r\partial^{n}(\mathbf{A}_{h}^{m}\partial^{\alpha}\rho)+rR_{D^{n_{j}}_{j}}(\partial_{i}^{n_{i}}\mathbf{A}_{h}^{m}\partial^{\alpha}\rho)+rR_{D^{n_{i
}}_{i}}(\partial_{j}^{n_{j}}\mathbf{A}_{h}^{m}\partial^{\alpha}\rho)\\
&+rR_{\mathbf{D}_{h}^{n}}(\mathbf{A}_{h}^{m}\partial^{\alpha}\rho).
    \end{aligned}
\end{equation}
Notice that the first term from the above right-hand side can be written as
\begin{align*}
r\partial^{n}(\mathbf{A}_{h}^{m}\partial^{\alpha}\rho)=&r\partial^{n+\alpha}\rho+rR_{A_{j}^{m_{j}}}(\partial^{\alpha+n}\rho)+rR_{A_{i}^{m_{i}}}(\partial^{\alpha+n}\rho)+rR_{\mathbf{A}_{h}^{m}}(\partial^{\alpha+n}\rho)
\end{align*}
where we have used Corollary \ref{cor:difference:average}.
Then, the first-order Taylor formula in the time variable yields
\begin{equation}\label{eq:almost:estimate}
\begin{aligned}
    D_{t}(r\partial^{n}(\mathbf{A}_{h}^{m}\partial^{\alpha}\rho))=&\partial_{t}(r\partial^{n+\alpha}\rho)+\partial_{t}(rR_{A_{j}^{m_{j}}}(\partial^{n+\alpha}\rho))\\
    &+\partial_{t}(rR_{A_{i}^{m_{i}}}(\partial^{\alpha+n}\rho))+\partial_{t}(rR_{\mathbf{A}_{h}^{m}}(\partial^{\alpha+n}\rho))\\
    &+\Delta t\int_{0}^{1}(1-\gamma)\partial_{t}^{2}\left( r(x,t+\gamma\Delta t)\partial^{n}\mathbf{A}_{h}^{m}(\partial^{\alpha}\rho(x,t+\gamma\Delta t))\right)d\gamma.
    \end{aligned}
\end{equation}
 Arguing as in the proof of Theorem \ref{theo:time:derivate}, that is, using \eqref{eq:coro:rho} and the expression of the remainders given in Proposition \ref{prop:difference:average}, we estimate
\begin{equation}\label{eq:estimate:R}
\begin{aligned}
\partial_{t}(rR_{A_{j}^{m_{j}}}(\partial^{n+\alpha}\rho))=&T(sh)^{2}s^{|\alpha|+|n|}\theta\mathcal{O}_{\lambda}(1),\\ \partial_{t}(rR_{A_{i}^{m_{i}}}(\partial^{n+\alpha}\rho))=&T(sh)^{2}s^{|\alpha|+|n|}\theta\mathcal{O}_{\lambda}(1),
\end{aligned}
\end{equation}
and 
\begin{equation}\label{eq:estimate:R:01}
\partial_{t}(rR_{\mathbf{A}_{h}^{m}}(\partial^{\alpha+n}\rho))=T(sh)^{4}s^{|\alpha|+|n|}\theta\mathcal{O}_{\lambda}(1).
\end{equation}
According to the first estimate from \cite[Lemma A.16]{GC-HS-2021} and Proposition \ref{prop:weight}, the last term from the right-hand side of \eqref{eq:almost:estimate} can be estimated as
\begin{equation}\label{integral:remainder}
\begin{split}
    \int_{0}^{1}(1-\gamma)\partial_{t}^{2}\left( r(x,t+\gamma\Delta t) \partial^{n}\mathbf{A}_{h}^{m}(\partial^{\alpha}\rho(x,t+\gamma\Delta t))\right)d\gamma=& T^{2}\trp(s^{|n|+|\alpha|}\theta^{2})\mathcal{O}_{\lambda}(1)\\
    &+ \trp((sh)^{2}T^{2}s^{|n|+|\alpha|}\theta^{2})\mathcal{O}_{\lambda}(1).
    \end{split}
\end{equation}
Hence, combining \eqref{eq:estimate:R}-\eqref{integral:remainder}, for \eqref{eq:almost:estimate} we obtain 
\begin{equation}
\begin{aligned}
    D_{t}(r\partial^{n}(\mathbf{A}_{h}^{m}\partial^{\alpha}\rho))=&\partial_{t}(r\partial^{n+\alpha}\rho)+T(sh)^{2}s^{|n|+|\alpha|}\theta\mathcal{O}_{\lambda}(1)\\
    &+\Delta t \trp((sh)^{2}T^{2}s^{|n|+|\alpha|}\theta^{2})\mathcal{O}_{\lambda}(1)\\
    &+\Delta t T^{2}\trp(s^{|n|+|\alpha|}\theta^{2})\mathcal{O}_{\lambda}(1).
    \end{aligned}
\end{equation}

Mimicking the same steps as before, it follows the estimates for the time-discrete operaror applied on the remainder terms of \eqref{estimate:theo:fully}; which yields \eqref{eq:prop:time:estimate}. 
\end{proof}

\section{Carleman estimate}\label{sec:Carleman}


\par \subsection{Weight functions}\label{subsec:weights}
In this subsection we collect the definitions and properties of the weight functions used in this section and in Sections \ref{sec:proof:carleman} and \ref{sec:cross}; later on, we refer back to this subsection.
We recall that $\Omega=\prod_{i=1}^{d}(0,L_{i})$ is the continuous domain, while $\mathcal{M}$, $\mathcal{M}_{i}^{\ast}$ and $\partial_{i}\mathcal{M}$ are the discrete meshes introduced in Section \ref{subsec:discrete:setting}. For $i\in\{1,\ldots,d\}$, $\partial_{i}\Omega:=\{x\in\partial\Omega\,;\,x_{i}\in\{0,L_{i}\}\}$ is the union of the two faces of $\Omega$ orthogonal to $e_{i}$, on which the outward unit normal is $-e_{i}$ (where $x_{i}=0$) or $e_{i}$ (where $x_{i}=L_{i}$). Let $\tilde{\Omega}$ be a smooth, open and connected neighborhood of $\overline{\Omega}$ in $\mathbb{R}^{d}$, and let $V_{\partial_{i}\Omega}$ be a sufficiently small neighborhood of $\partial_{i}\Omega$ in $\tilde{\Omega}$, in which the outward normal $n_{i}$ to $\partial_{i}\Omega$ is extended as the constant vector $\pm e_{i}$. The assumption on the Carleman weight function is standard for this problem, since it has been used in \cite{boyer-2014,LPP:2026,LLRP:2025,BHLR:2010,LPP:2026b}.\\ \textbf{Assumption:} Let $\overline{\omega_0}\subset \omega$ be an arbitrary fixed subdomain of $\Omega$. Let $\tilde{\Omega}$ be a smooth, open, and connected neighborhood of $\overline{\Omega}$ in $\mathbb{R}^d$. The function $x \mapsto \psi(x)$ belongs to $\mathcal{C}^p(\tilde{\Omega}, \mathbb{R})$, for sufficiently large $p$, and satisfies, for some $c > 0$,
\begin{equation}\label{assumtion:psi}
\psi > 0 \quad \text{in } \tilde{\Omega}, \quad |\nabla \psi| \geq c \quad \text{in } \tilde{\Omega} \setminus \omega_0, \quad \text{and} \quad \partial_{n_i} \psi(x) \leq -c < 0 \quad \text{for } x \in V_{\partial_i \Omega},
\end{equation}
 Although $\Omega$ has corners, the last condition is imposed separately near each family of faces $\partial_{i}\Omega$, with respect to the constant normal $\pm e_{i}$, and this assumption is the one used for the same rectangular geometry in \cite{BHLR:2010,boyer-2014}, where the existence of such a function $\psi$ is discussed in \cite{BHLR:2010}
For $\lambda\geq 1$ and $K>\|\psi\|_{\infty}$, we introduce the functions
\begin{align}\label{funcion-peso-2}
\varphi(x)=e^{\lambda\psi(x)}-e^{\lambda K}<0,\quad \xi(x):=e^{\lambda\psi}\geq 1,
\end{align}
 and, for $0<\delta \leq 1/2$,
\begin{equation}
    \theta(t)=\frac{1}{(t+\delta T)(T+\delta T-t)},\quad t\in [0,T].
\end{equation}
Given $\tau\geq 1$, we set
\begin{equation}\label{eq1}
s(t)=\tau\theta(t),\quad r(x,t):=e^{s(t)\varphi(x)},\quad\rho:=r^{-1}.
\end{equation}
 Note that $\partial^{\alpha}\varphi=\partial^{\alpha}\xi$ for $|\alpha|\geq1$, so that $\xi$ is the positive function appearing in the leading terms of the estimates of Section \ref{sec:asymp:weight} (where it is denoted by $\varphi$), and that $k_{0}:=\min_{\overline{\Omega}}(-\varphi)=e^{\lambda K}-\max_{\overline{\Omega}}e^{\lambda\psi}>0$. The weights $r$, $\rho$ and $\xi$ are continuous functions on $\overline{\Omega}\times[0,T+\Delta t]$ (resp. $\overline{\Omega}$), which are sampled on the meshes where they are evaluated, e.g. on $\mathcal{M}\times\mathcal{N}^{\ast}$ or on $\mathcal{M}_{i}^{\ast}\times\mathcal{N}^{\ast}$. Finally, with $\Gamma=\mathrm{Diag}(\gamma_{1},\ldots,\gamma_{d})$, we set $\nabla_{\Gamma}\psi:=\Gamma^{1/2}\nabla\psi$, so that $|\nabla_{\Gamma}\psi|^{2}=\sum_{j=1}^{d}\gamma_{j}(\partial_{j}\psi)^{2}$, and $\Delta_{\Gamma}\xi:=\sum_{j=1}^{d}\gamma_{j}\partial_{j}^{2}\xi=\lambda\xi(\lambda|\nabla_{\Gamma}\psi|^{2}+\Delta_{\Gamma}\psi)$.
\par Applying the asymptotic behavior of the Carleman weight functions given by Theorems \ref{theo:weight:estimates} and \ref{theo:fully:weight:estimates} and following the strategies used in \cite{boyer-2014,BHS:2020,GC-HS-2021,LMPZ:2023,LLRP:2025} we prove a fully-discrete Carleman estimate in arbitrary dimensions. Indeed, denoting  $\mathcal{P}q:=D_{t}q+\sum_{i=1}^{d}D_{i}(\gamma_{i}D_{i}\trm q)$ and 
\begin{equation*}
\begin{aligned}
\mathcal{J}(q):=&\int_{\mathcal{M}\times\mathcal{N}}\trm(s^{-1}r^{2})|D_{t}q|^{2}+\sum_{i=1}^{d}\int_{\mathcal{M}\times\mathcal{N}^{\ast}}s^{-1}r^{2}|D_{i}(\gamma_{i}D_{i}q)|^{2}+\sum_{i=1}^{d}\int_{\mathcal{M}_{i}^{\ast}\times\mathcal{N}^{\ast}}s\gamma_{i}\xi\,r^{2}|D_{i}q|^{2}\\
&+\int_{\mathcal{M}\times\mathcal{N}^{\ast}}s^{3}r^{2}\xi^{3}|\nabla_{\Gamma}\psi|^{4}|q|^{2}+\sum_{i=1}^{d}\int_{\mathcal{M}\times\mathcal{N}^{\ast}}s\gamma_{i}\xi\,r^{2}|A_{i}D_{i}q|^{2}
\end{aligned}
\end{equation*}
we have
\begin{theorem}\label{thm:Carleman:q}
Let $\omega$ be an open subset of $\mathcal{M}$, whose discrete counterpart $\omega\cap\mathcal{M}$ is still denoted by $\omega$, $\omega_{0}\Subset \omega$ a nonempty open set, $\psi$ satisfying \eqref{assumtion:psi}. For $\lambda\geq 1$ sufficiently large, there exist $C>0$, $\tau_{0}\geq 1$, $\epsilon_{0}>0$ depending on $\omega$, $\omega_{0}$, $T$, $\lambda$, $\Gamma$ such that
\begin{equation}\label{eq:Carleman:q}
\begin{aligned}
&\mathcal{J}(q)+\sum_{i=1}^{d}\int_{\partial_{i}\mathcal{M}\times\mathcal{N}^{\ast}}s\gamma_{i}\xi r^{2}\,t_{r}^{i}(|D_{i}q|^{2})+\sum_{i=1}^{d}\int_{\partial_{i}\mathcal{M}\times\mathcal{N}^{\ast}}s^{3}r^{2}\xi^{3}t_{r}^{i}(A_{i}|q|^{2})+\int_{\partial\mathcal{M}\times\mathcal{N}}(\trm r)^{2}|D_{t}q|^{2}\\
&\leq C_{\lambda}\bigg(\int_{\mathcal{M}\times\mathcal{N}}(\trm r)^{2}|\mathcal{P}q|^{2}+\int_{\omega\times\mathcal{N}^{\ast}}s^{3}r^{2}\xi^{3}|q|^{2}+h^{-2}\int_{\partial\mathcal{M}\times\partial\mathcal{N}}|\trp(rq)|^{2}+h^{-2}\int_{\mathcal{M}\times\partial\mathcal{N}}|\trp(rq)|^{2}\bigg),
\end{aligned}
\end{equation}
for all $\tau\geq\tau_{0}(T+T^{2})$, $0<h<h_{0}$, $\Delta t>0$ and $0<\delta\leq 1/2$ such that
\begin{equation}\label{eq:CFL:conditions}
\frac{\tau h}{\delta T^{2}}\leq\epsilon_{0}\quad\text{and}\quad\frac{\tau^{4}\Delta t}{\delta^{4}T^{6}}\leq\epsilon_{0}.
\end{equation}
\end{theorem}
This Carleman estimate implies an relaxed observability inequality for the system \eqref{eq:adjoint}.
\begin{theorem}\label{thm:observability}
Let $T>0$, $\mu\geq 1$, and $h_{0}>0$ given by Theorem~\ref{thm:Carleman:q}. There exist constants $C_{0},C,C_{1},C_{\mathrm{obs}}>0$ such that for all $ q_{T}\in L^{2}_{h}(\mathcal{M})$, the solution $q\in C(\overline{\mathcal{M}}\times\overline{\mathcal{N}}^{\ast})$ of \eqref{eq:adjoint} satisfies
\begin{equation}\label{eq:observability}
\|\trp q(0)\|^{2}_{L^{2}_{h}(\mathcal{M})}\leq C_{\mathrm{obs}}\bigg(\int_{\omega\times\mathcal{N}^{\ast}}|q|^{2}+e^{-\frac{C_{1}}{h^{\min\{\mu/4,1\}}}}\|q_{T}\|^{2}_{L^{2}_{h}(\mathcal{M})}\bigg),
\end{equation}
provided $h\leq\min\{h_{0},h_{1}\}$ and $\Delta t\leq\min\{T^{-2}h^{\mu},(4\|a\|_{L^{\infty}_{h}})^{-1}\}$, where
\begin{equation*}
h_{1}:=C_{0}\bigg(1+\frac{1}{T}+\|a\|_{L^{\infty}_{h}}^{2/3}\bigg)^{-\max\{1,4/\mu\}}\quad\text{and}\quad C_{\mathrm{obs}}:=e^{C\big(1+\frac{1}{T}+\|a\|_{L^{\infty}_{h}}^{2/3}+T\|a\|_{L^{\infty}_{h}}\big)}.
\end{equation*}
\end{theorem}
Its proof mimics the steps developed in \cite[section 4]{LMPZ:2023}. Indeed, let $q$ be a solution of \eqref{eq:adjoint}. Since $-\sum_{i}D_{i}(\gamma_{i}D_{i}\,\cdot)$ is nonnegative in $L^{2}_{h}(\mathcal{M})$ for functions vanishing on $\partial\mathcal{M}$, taking the inner product of the first equation of \eqref{eq:adjoint} at $t\in\mathcal{N}$ with $\trm q(t)$ gives $(1-\Delta t\|a\|_{L^{\infty}_{h}})\|\trm q(t)\|_{L^{2}_{h}(\mathcal{M})}\leq\|\trp q(t)\|_{L^{2}_{h}(\mathcal{M})}$. Iterating this discrete Gronwall-type inequality and using $(1-x)^{-1}\leq e^{2x}$ for $0\leq x\leq1/2$, we obtain, for all $t\in\overline{\mathcal{N}}$,
\begin{equation}\label{eq:energy:monotonicity}
\|\trp q(0)\|^{2}_{L^{2}_{h}(\mathcal{M})}\leq e^{4\|a\|_{L^{\infty}_{h}}t}\,\|\trp q(t)\|^{2}_{L^{2}_{h}(\mathcal{M})},
\end{equation}
provided $\Delta t\,\|a\|_{L_{h}^{\infty}}\leq1/2$. Moreover, every $t\in\mathcal{N}^{\ast}$ satisfies $q(t)=\trp q(t-\Delta t/2)$ with $t-\Delta t/2\in\overline{\mathcal{N}}$. Hence, defining $\eta(t)=1$ for $t\in\mathcal{N}^{\ast}\cap(T/4,3T/4)$ and $\eta(t)=0$ otherwise, and noting that $\int_{\mathcal{N}^{\ast}}\eta\geq T/4$ for $\Delta t\leq T/4$, \eqref{eq:energy:monotonicity} yields
\begin{equation}\label{eq:initial:from:energy}
\|\trp q(0)\|^{2}_{L^{2}_{h}(\mathcal{M})}\leq\frac{4}{T}\,e^{4T\|a\|_{L^{\infty}_{h}}}\int_{\mathcal{M}\times\mathcal{N}^{\ast}}\eta(t)|q|^{2}.
\end{equation}
On the support of $\eta$, we have $T^{-2}\leq\theta\leq 16/(3T^{2})$, so that $s^{3}\xi^{3}\geq\tau^{3}T^{-6}$ and $r^{2}=e^{2\tau\theta\varphi}\geq e^{-C\tau/T^{2}}$ with $C=\frac{32}{3}\|\varphi\|_{\infty}$. Therefore
\begin{equation}\label{eq:LHS:controls:initial}
\tau^{3}\int_{\mathcal{M}\times\mathcal{N}^{\ast}}\theta^{3}e^{2\tau\theta\varphi}\xi^{3}|q|^{2}\geq \frac{\tau^{3}}{4T^{5}}\,e^{-\frac{C\tau}{T^{2}}-4T\|a\|_{L^{\infty}_{h}}}\|\trp q(0)\|^{2}_{L^{2}_{h}(\mathcal{M})}.
\end{equation}
For the time boundary terms of \eqref{eq:Carleman:q}, the weight is evaluated at $t=\Delta t/2$ and $t=T+\Delta t/2$, and $q=0$ on $\partial\mathcal{M}$. We note that $\theta(0)=\theta(T)=\frac{1}{\delta(1+\delta)T^{2}}\leq\frac{1}{\delta T^{2}}$; however, since $\varphi<0$, bounding $e^{2\tau\theta\varphi}$ from above requires a lower bound on $\theta$ at these points. If $\Delta t\leq\delta T$, then $\theta(\Delta t/2)\geq\frac{4}{9\delta T^{2}}$ and $\theta(T+\Delta t/2)\geq\frac{1}{2\delta T^{2}}$, hence, with $k_{0}=\min_{\overline{\Omega}}(-\varphi)>0$,
\begin{equation*}
e^{2\tau\theta(\Delta t/2)\varphi}\leq e^{-\frac{8\tau k_{0}}{9\delta T^{2}}},\qquad e^{2\tau\theta(T+\Delta t/2)\varphi}\leq e^{-\frac{8\tau k_{0}}{9\delta T^{2}}}.
\end{equation*}
Combining these bounds with \eqref{eq:energy:monotonicity} for $t=T$, we obtain
\begin{equation}\label{eq:time:boundary:bound}
h^{-2}\int_{\mathcal{M}\times\partial\mathcal{N}}|\trp(rq)|^{2}\leq 2\,h^{-2}\,e^{-\frac{8\tau k_{0}}{9\delta T^{2}}+4T\|a\|_{L_{h}^{\infty}}}\|q_{T}\|^{2}_{L^{2}_{h}(\mathcal{M})}.
\end{equation}
Finally, the potential term is absorbed in the Carleman estimate, since $|a\,\trm q|\leq\|a\|_{L^{\infty}_{h}}|\trm q|$, and \eqref{eq:observability} follows by combining \eqref{eq:Carleman:q}, \eqref{eq:LHS:controls:initial} and \eqref{eq:time:boundary:bound} with the choice of $\tau$ described in Remark \ref{rem:conditions}; we refer to \cite[Section 4]{LMPZ:2023} for the details.
\section{Concluding remarks and perspectives}\label{sec:concluding}
In this work, we established the asymptotic behavior of the Carleman weight functions in arbitrary dimensions for spatial and time discrete operators, which are fundamental tools in developing Carleman estimates for semi-discrete and fully discrete operators. Recent works, such as \cite{GC-HS-2021} and \cite{LMPZ:2023}, have used Carleman estimates to analyze the controllability of fully discrete 1-D heat equations with Dirichlet and dynamic boundary conditions respectively. Hence, this work extends the results to arbitrary dimensions presented in \cite{GC-HS-2021}. For this reason, the results of this paper, such as Theorems \ref{theo:weight:estimates} and \ref{theo:fully:weight:estimates}, seem to be a promising tool to extend the results obtained in \cite{LMPZ:2023} to arbitrary dimensions. Moreover, as is mentioned in remark \ref{rmk:1}, the spatial asymptotic behavior also holds for the weight functions used in Carleman estimate for semi-discrete hyperbolic operators; hence, the approach used for the corresponding inverse problems could be adapted to arbitrary dimensions for these semi-discrete operators.
\par  A further natural question is the convergence, as $(h,\Delta t)\to(0,0)$, of the controls provided
by Theorem \ref{thm:phi:controllability} towards a null control of \eqref{system:control}; see Remark \ref{rem:convergence}, and \cite{bhandari:hal-04597449} for the
time-discrete case.
\par In a similar direction, another unexplored research topic where Carleman estimates are fundamental is the discrete or semi-discrete inverse problem for parabolic operators. In recent years, results for semi-discrete hyperbolic operators \cite{BEO-2015,BE-2013} have provided stability result for inverse problems. Nevertheless, there is only one stability estimate results for semi-discrete deterministic parabolic operators \cite{LLRP:2025}. Further, \cite{Wu_2024} considers an inverse problem for a semi-discrete stochastic parabolic equation in the one-dimensional setting. Thus, the results presented in this work are important tools for studying those semi-discrete inverse problems, as Carleman estimates are the main tool to achieving these results, see Remark \ref{rmk:1}. Indeed,  three Carleman estimates for semi-discrete stochastic parabolic operators in arbitrary dimensions are obtained in \cite{LPP:2026b} by using Theorems \ref{theo:weight:estimates} and \ref{theo:fully:weight:estimates}. These Carleman estimates are applied to study inverse random source and Cauchy problems, see \cite[Theorem 1.1]{LPP:2026b} and \cite[Theorem 1.3]{LPP:2026b}.
\par Recently, several works have been considered the controllability problems for semi-discrete stochastic partial differential equations.  The $\phi$-null controllability of the second- and fourth-order semi-discrete stochastic parabolic operators is established in \cite{zhao:2024}, \cite{WZ:2025} and \cite{WZ:2024}  (see \cite{boyer-canum} for details of $\phi$-null controllability for semi-discrete systems). Moreover, in \cite{LL:2026} the authors obtain $\phi$-null controllability results for a cascade system of backward stochastic semi-discrete fourth- and second-order parabolic equations. Concerning the controllability issue, the task is to establish a relaxed observability inequality through a  Carleman estimate for the respective semi-discrete backward adjoint system, then to obtain $\phi$-null controllability result for its respective forward semi-discrete stochastic system. Using this methodology,  controllability results for spatial semi-discrete approximations of $1$-$D$ stochastic parabolic operators are obtained in \cite{zhao:2024}, \cite{WZ:2025} and  \cite{WZ:2024}. It is worth noting that spatial terms are key to obtaining the Carleman estimates present in those works. For this reason, the estimates related to the Carleman weight function proved in the present paper are fundamental in extending the aforementioned results to arbitrary dimensions. Moreover, due to the difficulty of managing the estimation on the Carleman weight function, there are only two works concerning a higher order operator in 1-D, see \cite{CLTP-2022} and \cite{WZ:2024} for controllability results of a deterministic and stochastic fourth-order operator, respectively. These difficulties are overcome in \cite{LPP:2026} and \cite{LPP:2026C} by using Theorem \ref{theo:weight:estimates}, where the authors obtain $\phi$-controllability result for semi-discrete stochastic parabolic operators in arbitrary dimension which generalize the results obtained in \cite{zhao:2024}.
\par In turn,  a stability estimate of the discrete Calder\'on problem is obtained in \cite{ZY-2024} with partial data measured on an arbitrary portion of the boundary where a suitable Carleman estimate is the main tool. This is an improvement of the stability result presented in \cite{LDOP-2021} and also needs to estimate repeated application of the discrete spatial operator on the Carleman weight function to obtain a weaker version, with respect to its continuous framework, of the unique continuation property. In the same direction, a discrete version of the three spheres inequality is obtained in \cite{FB:R:R-2021}, and the unique continuation properties of the discrete fractional Laplacian operator is studied in \cite{FB:R:R:2024}, both of via discrete Carleman estimates, with an approach that mimics the methodology from \cite{ervedoza-2011}. Hence several computations related with trigonometric identities are needed in the development of the discrete Carleman estimate in the works \cite{FB:R:R-2021,ervedoza-2011,LDOP-2021}. Thus, the estimates from Proposition \ref{prop:weight} and Theorem \ref{theo:weight:estimates} give an alternative way to develop of Carleman estimates for discrete elliptic operators. 

\section*{Acknowledgment}
Part of this work was completed during the author's research stay at Escuela de Matem\'aticas, Universidad Industrial de Santander, Colombia. The author extends heartfelt thanks to Professor Juan L\'opez-R\'ios for his kind hospitality.
\section*{Conflict of interest}
 The corresponding author states that there is no conflict of interest.
 \section*{Funding} 
 This work was partially supported by FONDECYT Grant 11250805.

\section{Proof of the Carleman estimate}\label{sec:proof:carleman}

 \subsection*{Outline of the proof}
We first give the main steps of the proof, emphasizing the differences with the one-dimensional and the semi-discrete cases.
\begin{enumerate}
\item[(1)] \emph{Conjugation.} With $q=\rho z$, the conjugated operator $\trm(r)\mathcal{P}(\rho z)$ is split into $Az+Bz$ plus a remainder $Rz$, see \eqref{eq:innner:product}. The spatial splitting relies on the discrete product rules \eqref{D:product}--\eqref{A:product} and the time splitting on Lemma \ref{lem:time:discrete:weight}.
\item[(2)] \emph{Cross-product.} The main term is $2\int Az\cdot Bz=2\sum_{i,j}I_{ij}$, see \eqref{conmu}. Each $I_{ij}$ is computed in Section \ref{sec:cross} by discrete integrations by parts in space and time, and all the coefficients involving the weights are expanded by Theorems \ref{theo:weight:estimates} and \ref{theo:fully:weight:estimates}. Compared with the one-dimensional case, the new feature is the presence of crossed terms involving two directions $i\neq j$ (e.g. $I_{11}^{ij}$, $I_{21}^{ij}$), in which the operators $D_{i}$, $A_{i}$, $D_{j}$, $A_{j}$ act on the weights; these are handled by Corollary \ref{cor:difference:average} and Theorem \ref{theo:weight:estimates}. Compared with the semi-discrete case, the time-discrete operator produces additional terms of size $\Delta t$ (e.g. $\Delta t\int|D_{it}^{2}z|^{2}$) and time boundary terms, which are handled by Lemma \ref{lem:time:discrete:weight} and Theorem \ref{theo:fully:weight:estimates}.
\item[(3)] \emph{Positivity.} The leading terms combine into the positive quantities $\lambda^{4}\int s^{3}\xi^{3}|\nabla_{\Gamma}\psi|^{4}|z|^{2}$ and $\lambda^{2}\sum_{i}\int s\gamma_{i}\xi|\nabla_{\Gamma}\psi|^{2}|D_{i}z|^{2}$, see \eqref{eq:combined:recall} and the discussion after it; the remaining terms are absorbed under the conditions \eqref{eq:CFL:conditions}.
\item[(4)] \emph{Conclusion.} The observation terms are introduced, the terms $\|Az\|^{2}$, $\|Bz\|^{2}$ are used to recover the second-order and time-derivative terms, and one returns to the variable $q$.
\end{enumerate}
\subsection{Conjugated operator}\label{sec:conjugate}
 We use the notation of Section \ref{subsec:weights}: $r:=e^{s\varphi}$ and $\rho:=r^{-1}$, $s=\tau\theta$, $\varphi=e^{\lambda\psi}-e^{\lambda K}$, and $\xi=e^{\lambda\psi}$, where $\psi\in C^{k}$ for $k$ sufficiently large and $\lambda\geq 1$. The proof of some technical results can be found in the Section \ref{sec:cross}.
\par By regarding this weight function, the first step is to consider the change of variable $q=\rho z$. Our first task is to split the conjugate operator $\trm r\mathcal{P}(\rho z):=\trm (r)D_{t}(\rho z)+\trm( r\sum_{i=1}^{d}D_{i}\left(\gamma_{i}D_{i}(\rho z)\right))$ into several terms that we will estimate separately. By \eqref{eq:discreteleibniz}, we have
 \begin{equation}\label{eq:conjugate}
\trm r\left(\sum_{i=1}^{d} D_{i}\left(\gamma_{i}D_{i}(\rho z)\right)\right)=A_{1}z+A_{2}z+B_{1}z+C_{h}z,
\end{equation}
where $\displaystyle A_{1}z:=\sum_{i=1}^{d}\trm(rA_{i}^{2}\rho\,D_{i}(\gamma_{i}D_{i}z))$, $\displaystyle B_{1}z:=2\sum_{i=1}^{d}\gamma_{i}\trm(rD_{i}A_{i}\rho D_{i}A_{i}z)$,\\ $\displaystyle A_{2}z:=\sum_{i=1}^{d}\gamma_{i}\trm(rD_{i}^{2}\rho A_{i}^{2}z),$
and\begin{equation*}
    \begin{split}
        C_{h}z:=&\sum_{i=1}^{d}h\,\mathcal{O}(1)\trm(rD_{i}^{2}\rho A_{i}^{2}z)+\sum_{i=1}^{d}h\,\mathcal{O}(1)\trm(rD_{i}A_{i}\rho D_{i}A_{i}z)+\sum_{i=1}^{d}D_{i}\gamma_{i}\trm(r\,A_{i}D_{i}\rho A_{i}^{2}z)\\
        &+\sum_{i=1}^{d}\frac{h^{2}}{4}D_{i}\gamma_{i}\trm(r\,D_{i}^{2}\rho D_{i}A_{i}z)+\sum_{i=1}^{d}\frac{h^{2}}{4}D_{i}\gamma_{i}\trm(r\,D_{i}A_{i}\rho D_{i}^{2}z).
    \end{split}
\end{equation*}
Indeed, by \eqref{D:product}, $D_{i}(\rho z)=D_{i}\rho\,A_{i}z+A_{i}\rho\,D_{i}z$, and, applying \eqref{D:product} once more,
\begin{equation*}
D_{i}(\gamma_{i}D_{i}(\rho z))=D_{i}(\gamma_{i}D_{i}\rho)\,A_{i}^{2}z+A_{i}(\gamma_{i}D_{i}\rho)\,D_{i}A_{i}z+D_{i}(\gamma_{i}A_{i}\rho)\,A_{i}D_{i}z+A_{i}(\gamma_{i}A_{i}\rho)\,D_{i}^{2}z.
\end{equation*}
Expanding the coefficients by \eqref{D:product} and \eqref{A:product},
\begin{equation*}
\begin{aligned}
D_{i}(\gamma_{i}D_{i}\rho)&=D_{i}\gamma_{i}\,A_{i}D_{i}\rho+A_{i}\gamma_{i}\,D_{i}^{2}\rho, & A_{i}(\gamma_{i}D_{i}\rho)&=A_{i}\gamma_{i}\,A_{i}D_{i}\rho+\tfrac{h^{2}}{4}D_{i}\gamma_{i}\,D_{i}^{2}\rho,\\
D_{i}(\gamma_{i}A_{i}\rho)&=D_{i}\gamma_{i}\,A_{i}^{2}\rho+A_{i}\gamma_{i}\,D_{i}A_{i}\rho, & A_{i}(\gamma_{i}A_{i}\rho)&=A_{i}\gamma_{i}\,A_{i}^{2}\rho+\tfrac{h^{2}}{4}D_{i}\gamma_{i}\,D_{i}A_{i}\rho,
\end{aligned}
\end{equation*}
and using $D_{i}(\gamma_{i}D_{i}z)=D_{i}\gamma_{i}\,A_{i}D_{i}z+A_{i}\gamma_{i}\,D_{i}^{2}z$ and $A_{i}D_{i}=D_{i}A_{i}$, we obtain
\begin{equation*}
\begin{aligned}
rD_{i}(\gamma_{i}D_{i}(\rho z))=&\;rA_{i}^{2}\rho\,D_{i}(\gamma_{i}D_{i}z)+2A_{i}\gamma_{i}\,rD_{i}A_{i}\rho\,D_{i}A_{i}z+A_{i}\gamma_{i}\,rD_{i}^{2}\rho\,A_{i}^{2}z+D_{i}\gamma_{i}\,rA_{i}D_{i}\rho\,A_{i}^{2}z\\
&+\tfrac{h^{2}}{4}D_{i}\gamma_{i}\,rD_{i}^{2}\rho\,D_{i}A_{i}z+\tfrac{h^{2}}{4}D_{i}\gamma_{i}\,rD_{i}A_{i}\rho\,D_{i}^{2}z.
\end{aligned}
\end{equation*}
Now, using that $\trp z=\trm z+\Delta t D_{t}z$,  the product rule $D_{t}(\rho z)=D_{t}\rho\trp z+\trm\rho D_{t}z$ and Lemma \ref{lem:time:discrete:weight} with $p=0$, provided $\Delta t \tau (T^3 \delta^2)^{-1}\leq 1$ it follows that
\begin{equation}\label{eq:time:conjugate}
\begin{split}
    \trm(r)\,D_{t}(\rho z)
    =&D_{t}z-\tau\varphi\trm(\theta'z)+C_{\Delta t}z,
    \end{split}
\end{equation}
with $C_{\Delta t}z:=\Delta t\left(\frac{\tau}{\delta^{3}T^{4}}+\frac{\tau^{2}}{\delta^{4}T^{6}} \mathcal{O}_{\lambda}(1)\right)\trp z-\tau\Delta t\trm(\theta')\varphi D_{t}z$. Now, we set $A_{3}z:=-\tau\varphi\trm(\theta'z)$, $B_{2}z:=-2\trm(s\Delta_{\Gamma}\xi z)$ and $B_{3}z:=D_{t}z$, where $\Delta_{\Gamma}\xi=\sum_{i=1}^{d}\gamma_{i}\partial_{x_{i}}^{2}\xi$. Thus, combining \eqref{eq:conjugate} and \eqref{eq:time:conjugate} we have the following identity in $\mathcal{M}\times\mathcal{N}$
\begin{equation}\label{eq:innner:product}
    Az+Bz=Rz,
\end{equation}
where $Az=A_{1}z+A_{2}z+A_{3}z$, $Bz:=B_{1}z+B_{2}z+B_{3}z$ and $Rz:=\trm r\mathcal{P}(\rho z)-C_{h}z-C_{\Delta t}z+B_{2}z$. Hence, taking the $L_{h}^{2}(\mathcal{M}\times\mathcal{N})$ in \eqref{eq:innner:product} yields 
\begin{equation}\label{eq:inner:product:2}
    \left\| Az\right\|^{2}+\left\| Bz\right\|^{2}+2\int_{\mathcal{M}\times\mathcal{N}}Az\cdot Bz=\left\| Rz\right\|^{2}.
\end{equation}
The main task to obtain a Carleman estimate is to estimate the following term 
\begin{equation}\label{conmu}
    \int_{\mathcal{M}\times\mathcal{N}}Az\cdot B z=\sum_{i,j=1}^{3}\int_{\mathcal{M}\times\mathcal{N}}A_{i}z\cdot B_{j}z=:\sum_{i,j=1}^{3}I_{ij}. 
\end{equation}
 For every term of (\ref{conmu}) we present its respective estimate in Section \ref{sec:cross}. Collecting all these estimates, namely \eqref{eq:I11:final}, \eqref{eq:I12:final}, \eqref{eq:I13:final}, the estimate of $I_{21}$, \eqref{eq:I22:final}, \eqref{eq:I23:total} and the estimates of $I_{31}$, $I_{32}$, $I_{33}$, we get
 \begin{equation}\label{eq:combined:recall}
\begin{aligned}
2\int_{\mathcal{M}\times\mathcal{N}}Az\cdot Bz\geq\;&\lambda^{4}\int_{\mathcal{M}\times\mathcal{N}^{\ast}}s^{3}\xi^{3}|\nabla_{\Gamma}\psi|^{4}\,|z|^{2}\\
&+\lambda^{2}\sum_{i=1}^{d}\int_{\mathcal{M}_{i}^{\ast}\times\mathcal{N}^{\ast}}s\gamma_{i}\xi\Big[2\gamma_{i}(\partial_{i}\psi)^{2}+|\nabla_{\Gamma}\psi|^{2}\Big]|D_{i}z|^{2}\\
&-2\lambda\sum_{i=1}^{d}\int_{\mathcal{M}_{i}^{\ast}\times\mathcal{N}^{\ast}}s\gamma_{i}\xi\,\partial_{i}\psi\,\partial_{i}\gamma_{i}\,|D_{i}z|^{2}\\
&+C_{\lambda}\Delta t\sum_{i=1}^{d}\int_{\mathcal{M}_{i}^{\ast}\times\mathcal{N}}\gamma_{i}\,|D_{it}^{2}z|^{2}+\sum_{i\ne j}V_{11}^{ij}-\mathcal{X}-\mathcal{Y},
\end{aligned}
\end{equation}
$D_{it}^{2}:=D_{i}D_{t}$, $C_{\lambda}>0$, the coefficient $1$ of the first term is $3-2$, from $I_{21}$ and $I_{22}$ (see the remark at the end of Section \ref{lem:I22}), $\mathcal{X}$ collects the volume remainders ($W_{11}$, $X_{12}$, $X_{21}$, $X_{22}$, $X_{23}^{i}$, and the lower-order terms of $I_{13}$ and $I_{3j}$) and $\mathcal{Y}$ the boundary remainders ($Y_{11}$, $Y_{21}$, $Y_{23}^{i}$ and the time boundary terms of $I_{13}$). The terms $V_{11}^{ij}$, $i\neq j$, defined in \eqref{eq:V11}, come from the crossed directions in $I_{11}$. Their first part is not sign-definite by itself, but together with the diagonal terms it forms a square:
\begin{equation*}
2\lambda^{2}\sum_{i,j=1}^{d}\int_{\mathcal{M}\times\mathcal{N}^{\ast}}s\gamma_{i}\gamma_{j}\xi\,\partial_{i}\psi\,\partial_{j}\psi\,D_{i}A_{i}z\,D_{j}A_{j}z=2\lambda^{2}\int_{\mathcal{M}\times\mathcal{N}^{\ast}}s\xi\Big(\sum_{i=1}^{d}\gamma_{i}\partial_{i}\psi\,D_{i}A_{i}z\Big)^{2}\geq0,
\end{equation*}
while the second part of $V_{11}^{ij}$ is nonnegative. Since $|D_{i}A_{i}z|^{2}\leq A_{i}(|D_{i}z|^{2})$, $\displaystyle\int_{\mathcal{M}}u\,A_{i}v\leq\int_{\mathcal{M}_{i}^{\ast}}A_{i}u\,v$ for $u,v\geq0$, and $A_{i}(s\gamma_{i}^{2}\xi(\partial_{i}\psi)^{2})=s\gamma_{i}^{2}\xi(\partial_{i}\psi)^{2}+s\,\mathcal{O}_{\lambda}(h)$, we deduce
\begin{equation*}
\sum_{i\neq j}V_{11}^{ij}\geq-2\lambda^{2}\sum_{i=1}^{d}\int_{\mathcal{M}_{i}^{\ast}\times\mathcal{N}^{\ast}}s\gamma_{i}^{2}\xi(\partial_{i}\psi)^{2}|D_{i}z|^{2}-\sum_{i=1}^{d}\int_{\mathcal{M}_{i}^{\ast}\times\mathcal{N}^{\ast}}s\,\mathcal{O}_{\lambda}(h)|D_{i}z|^{2}.
\end{equation*}
Hence, the second and fifth terms on the right-hand side of \eqref{eq:combined:recall} give at least \break $\lambda^{2}\sum_{i}\int_{\mathcal{M}_{i}^{\ast}\times\mathcal{N}^{\ast}}s\gamma_{i}\xi|\nabla_{\Gamma}\psi|^{2}|D_{i}z|^{2}$, up to a term $s\,\mathcal{O}_{\lambda}(h)|D_{i}z|^{2}$ which is absorbed for $h$ small. The proof is then completed as follows.
\begin{enumerate}
\item[(i)] \emph{Local terms.} Since $|\nabla_{\Gamma}\psi|\geq c\,c_{0}^{-1/2}$ outside $\omega_{0}$ by \eqref{assumtion:psi}, we add $\lambda^{4}\int_{\omega_{0}\times\mathcal{N}^{\ast}}s^{3}\xi^{3}|z|^{2}$ and $\lambda^{2}\sum_{i}\int_{\omega_{0}\times\mathcal{N}^{\ast}}s\xi|D_{i}z|^{2}$ to both sides, so that the weights $|\nabla_{\Gamma}\psi|^{4}$ and $|\nabla_{\Gamma}\psi|^{2}$ are bounded from below on the whole mesh, following \cite[Lemma 2.12]{GC-HS-2021}.
\item[(ii)] \emph{Absorption.} The term $-2\lambda\int s\gamma_{i}\xi\partial_{i}\psi\partial_{i}\gamma_{i}|D_{i}z|^{2}$ and the terms of order $s\lambda$ and $s^{3}\lambda^{3}$ in $\mathcal{X}$ are absorbed by taking $\lambda$ large (depending on $c_{0}$). The terms carrying a factor $(sh)^{2}$ or $h$ are absorbed under the first condition in \eqref{eq:CFL:conditions}, since $sh\leq\tau h(\delta T^{2})^{-1}$. The terms involving $Ts^{2}\theta$ are of lower order with respect to $s^{3}$, since $T\theta=Ts/\tau\leq s/\tau_{0}$ for $\tau\geq\tau_{0}T$. The terms carrying a factor $\Delta t$ are absorbed under the second condition in \eqref{eq:CFL:conditions}; for instance, the coefficient $\Delta t\,s^{2}$ of $\int|D_{t}z|^{2}$ in $X_{23}^{i}$ satisfies, since $s\leq\tau(\delta T^{2})^{-1}$, $\Delta t\,s^{2}=s^{-1}\Delta t\,s^{3}\leq s^{-1}\Delta t\,\tau^{3}\delta^{-3}T^{-6}\leq\epsilon_{0}\,\delta\,\tau^{-1}s^{-1}$, so that it is absorbed by the term $\int s^{-1}|D_{t}z|^{2}$ recovered in step (iv).
\item[(iii)] \emph{Boundary terms.} The spatial boundary terms in $\mathcal{Y}$ are nonnegative at leading order by the condition $\partial_{n_{i}}\psi\leq-c$ in \eqref{assumtion:psi} (see the remark at the end of the estimate of $I_{11}$), and their remainders are absorbed as in (ii). The time boundary terms are bounded, using $|D_{i}u|^{2}\leq4h^{-2}\max|u|^{2}$, by the terms $h^{-2}\int_{\mathcal{M}\times\partial\mathcal{N}}|\trp(rq)|^{2}$ and $h^{-2}\int_{\partial\mathcal{M}\times\partial\mathcal{N}}|\trp(rq)|^{2}$ of \eqref{eq:Carleman:q}.
\item[(iv)] \emph{Second-order terms.} The terms $\int\trm(s^{-1}r^{2})|D_{t}q|^{2}$ and $\int s^{-1}r^{2}|D_{i}(\gamma_{i}D_{i}q)|^{2}$ of $\mathcal{J}(q)$ are recovered from $\|Az\|^{2}+\|Bz\|^{2}$ in \eqref{eq:inner:product:2}, following \cite[Theorem 1.2]{LLRP:2025}.
\item[(v)] \emph{Local gradient term and original variable.} The local term $\lambda^{2}\sum_{i}\int_{\omega_{0}\times\mathcal{N}^{\ast}}s\xi|D_{i}z|^{2}$ is removed by means of a cut-off function supported in $\omega$, as in \cite[Lemma 2.16]{GC-HS-2021}, and one returns to the variable $q=\rho z$ as in \cite[Lemma 9]{LMPZ:2023}, which yields \eqref{eq:Carleman:q}.
\end{enumerate}
\section{Estimate for the cross-product}\label{sec:cross}
In this section we provided all the estimates needed to obtain a bound for the cross-product $\int_{\mathcal{M}\times\mathcal{N}}Az\cdot Bz$. We begin recalling the  definition of the respective $I_{ij}$ for $i,j=1,2,3$, then usually a time-shifting allows to integrate on $\mathcal{M}\times\mathcal{N}^{\ast}$. In higher dimension this is the hardest part in the development of a semi(fully)-discrete Carleman estimate, since every estimate of $I_{ij}$  requires the asymptotic behavior of the Carleman weight functions under several discrete operators, and many additional estimates are needed to bound $I_{ij}$. For this reason, it becomes easier by using Theorem \ref{theo:weight:estimates} and \ref{theo:fully:weight:estimates} since they contain all the possible configuration to be estimated in the proof of a Carleman estimate. Thus, these Theorems not only reduce the computation also enable one to obtain Carleman estimate in higher-dimensions and higher-order. At the end of some subsections we give a remark about the role of the respective estimate in the Carleman estimate .
\subsection{Estimate of $I_{11}$ }
We set $\beta_{11}^{ij}:=\gamma_{j}r^{2}A_{i}^{2}\rho D_{j}A_{j}\rho$, then recalling that
\begin{equation*}
    \begin{split}
        A_{1}z&:=\sum_{i=1}^{d}\trm(rA_{i}^{2}\rho\,D_{i}(\gamma_{i}D_{i}z))\, \text{ and }\,B_{1}z:=2\sum_{i=1}^{d}\gamma_{i}\trm(rD_{i}A_{i}\rho D_{i}A_{i}z),
    \end{split}
\end{equation*}
we have after a shift with respect to the discrete-time variable
\begin{equation*}
    \begin{split}
     I_{11}=\sum_{i,j= 1}^{d}2\int_{\mathcal{M}\times\mathcal{N}} \trm(\beta_{11}^{ij}D_{i}(\gamma_{i}D_{i}z)\,D_{j}A_{j}z)=&\sum_{i,j=1}^{d}2\int_{\mathcal{M}\times\mathcal{N}^{\ast}} \beta_{11}^{ij}D_{i}(\gamma_{i}D_{i}z)\,D_{j}A_{j}z\\
    =:&\sum_{i,j=1}^{d}I_{11}^{ij}.
    \end{split}
\end{equation*}
\par Let us first consider the case $i=j$. Using the product rule \eqref{D:product} we have
\begin{equation*}
\begin{split}
    I_{11}^{ii}=&2\int_{\mathcal{M}\times\mathcal{N}^{\ast}} \beta_{11}^{ii}D_{i}\gamma_{i}|D_{i}A_{i}z|^{2}+\beta_{11}^{ii}A_{i}\gamma_{i}D_{i}^{2}zD_{i}A_{i}z\\
    =&2\int_{\mathcal{M}\times\mathcal{N}^{\ast}} \beta_{11}^{ii}D_{i}\gamma_{i}|D_{i}A_{i}z|^{2}+\int_{\mathcal{M}\times\mathcal{N}^{\ast}} \beta_{11}^{ii}A_{i}\gamma_{i}D_{i}(|D_{i}z|^{2})\\
    =:&I_{11}^{ii(a)}+I_{11}^{ii(b)}.
    \end{split}
\end{equation*}
Applying a discrete integration by parts with respect to the difference operator $D_{i}$ given by \cite[Lemma 2.2, equation (15)]{LDOP-2021} on the integral $I_{11}^{ii(b)}$, it follows that
\begin{equation*}
    I_{11}^{ii(b)}=-\int_{\mathcal{M}^{\ast}_{i}\times\mathcal{N}^{\ast}} D_{i}(\beta_{11}^{ii}A_{i}\gamma_{i})|D_{i}z|^{2}+\int_{\partial_{i}\mathcal{M}\times\mathcal{N}^{\ast}}\beta_{11}^{ii}A_{i}\gamma_{i}t_{r}^{i}(|D_{i}z|^{2})n_{i}.
\end{equation*}
Now, using Theorem \ref{theo:weight:estimates} we estimate $\beta_{11}^{ii}=-    s\lambda\gamma_{i}\xi\partial_{i}\psi+s\mathcal{O}_{\lambda}((sh)^{2})$ and $D_{i}(\beta_{11}^{ii}A_{i}\gamma_{i})=-s\lambda^{2}\gamma_{i}^{2}\xi(\partial_{i}\psi)^{2}+s\lambda\mathcal{O}_{\lambda}(1)+s\mathcal{O}_{\lambda}((sh)^{2})$. Then for $I_{11}^{ii(b)}$ we obtain
\begin{equation}\label{eq:J11b}
\begin{aligned}
I_{11}^{ii(b)}=&\int_{\mathcal{M}^{\ast}_{i}\times\mathcal{N}^{\ast}}\left(s\lambda^{2}\gamma_{i}^{2}\xi(\partial_{i}\psi)^{2}+s\lambda\mathcal{O}_{\lambda}(1)+s\mathcal{O}_{\lambda}((sh)^{2})\right)|D_{i}z|^{2}\\
&+\int_{\partial_{i}\mathcal{M}\times\mathcal{N}^{\ast}}(-    s\lambda\gamma_{i}\xi\partial_{i}\psi+s\mathcal{O}_{\lambda}((sh)^{2}))t_{r}^{i}(|D_{i}z|^{2})n_{i}.
\end{aligned}
\end{equation}
In turn, using the identity $A_{i}(|D_{i}z|^{2})=|D_{i}A_{i}z|^{2}+\frac{h^{2}}{4}|D_{i} ^{2}z|^{2}$ for $i\in\{1,\ldots,d\}$, we can rewrite the integral $I_{11}^{ii(a)}$ as
\begin{equation*}
    I_{11}^{ii(a)}=2\int_{\mathcal{M}\times\mathcal{N}^{\ast}}\beta_{11}^{ii}D_{i}\gamma_{i}A_{i}(|D_{i}z|^{2})-\frac{h^{2}}{2}\int_{\mathcal{M}\times\mathcal{N}^{\ast}}\beta_{11}^{ii}D_{i}\gamma_{i}|D_{i}^{2}z|^{2}.
\end{equation*}
    Then, applying \cite[Lemma 2.2, equation (16)]{LDOP-2021} on the first integral above yields
\begin{equation}
    I_{11}^{ii(a)}=2\int_{\mathcal{M}^{\ast}_{i}\times\mathcal{N}^{\ast}}A_{i}(\beta_{11}^{ii}D_{i}\gamma_{i})|D_{i}z|^{2}+h\int_{\partial_{i}\mathcal{M}\times\mathcal{N}^{\ast}}\beta_{11}^{ii}D_{i}\gamma_{i}t_{r}^{i}(|D_{i}z|^{2})-\frac{h^{2}}{2}\int_{\mathcal{M}\times\mathcal{N}^{\ast}}\beta_{11}^{ii}D_{i}\gamma_{i}|D^{2}_{i}z|^{2}.
\end{equation}
 Thus, by using Theorem \ref{theo:weight:estimates}, for $I_{11}^{ii(a)}$ we have 
\begin{equation}\label{eq:I11:ii:a:final}
\begin{aligned}
I_{11}^{ii(a)}=&\;-2\lambda\int_{\mathcal{M}^{\ast}_{i}\times\mathcal{N}^{\ast}}s
\xi\,\partial_{i}\psi\gamma_{i}\,\partial_{i}\gamma_{i}\,|D_{i}z|^{2}
+
\int_{\mathcal{M}^{\ast}_{i}\times\mathcal{N}^{\ast}}s\,\mathcal{O}_{\lambda}((sh)^{2})|D_{i}z|^{2}\\
&+
\int_{\partial_{i}\mathcal{M}\times\mathcal{N}^{\ast}}
\mathcal{O}_{\lambda}(sh)t_{r}^{i}(|D_{i}z|^{2})
+h\,
\int_{\mathcal{M}\times\mathcal{N}^{\ast}}\mathcal{O}_{\lambda}(sh)|D_{i}^{2}z|^{2}.
\end{aligned}
\end{equation}
 Hence, collecting \eqref{eq:J11b} and \eqref{eq:I11:ii:a:final} we get
\begin{equation}
    \begin{aligned}
I_{11}^{ii}=&\int_{\mathcal{M}^{\ast}_{i}\times\mathcal{N}^{\ast}}\left(s\lambda^{2}\gamma_{i}^{2}\xi(\partial_{i}\psi)^{2}-2\lambda s
\xi\,\partial_{i}\psi\gamma_{i}\,\partial_{i}\gamma_{i}+s\lambda\mathcal{O}_{\lambda}(1)+s\mathcal{O}_{\lambda}((sh)^{2})\right)|D_{i}z|^{2}\\
&+\int_{\partial_{i}\mathcal{M}\times\mathcal{N}^{\ast}}(-    s\lambda\gamma_{i}^{2}\xi\partial_{i}\psi+s\mathcal{O}_{\lambda}((sh)^{2}))t_{r}^{i}(|D_{i}z|^{2})n_{i}+h\int_{\mathcal{M}\times\mathcal{N}^{\ast}}\mathcal{O}_{\lambda}(sh)|D_{i}^{2}z|^{2}.
\end{aligned}
\end{equation}
\par Let us focus in the case $i\ne j$. Applying the discrete integration by parts \cite[Lemma 2.2, equation (15)]{LDOP-2021}, we obtain 
\begin{equation*}  I_{11}^{ij}=-2\int_{\mathcal{M}_{i}^{\ast}\times\mathcal{N}^{\ast}}D_{i}( \beta_{11}^{ij}\,D_{j}A_{j}z)\gamma_{i}D_{i}z,
\end{equation*}
since $z=0$ on $\partial\mathcal{M}$ implies $D_{j}A_{j}z=0$ on $\partial_{i}\mathcal{M}$. Now, by equation \eqref{D:product} we write
\begin{equation}
\begin{split}
    I_{11}^{ij}=&-2\int_{\mathcal{M}_{i}^{\ast}\times\mathcal{N}^{\ast}}\gamma_{i}D_{i}( \beta_{11}^{ij})\,D_{j}A_{ij}^{2}z\,D_{i}z-2\int_{\mathcal{M}_{i}^{\ast}\times\mathcal{N}^{\ast}}\gamma_{i}A_{i}(\beta_{11}^{ij})\,D_{ij}^{2}A_{j}z\,D_{i}z=:-I_{11}^{ij(a)}-I_{11}^{ij(b)}.
    \end{split}
\end{equation}
We can integrate by parts with respect to the average operator $A_{i}$ on $I_{11}^{ij(a)}$(\cite[Lemma 2.2, equation (16)]{LDOP-2021}) to obtain
\begin{equation}\label{eq:non:diagonal:I_{11a}}
\begin{split}
    I_{11}^{ij(a)}=&2\int_{\mathcal{M}\times\mathcal{N}^{\ast}}A_{i}(\gamma_{i}D_{i} \beta_{11}^{ij}\,D_{i}z)\,D_{j}A_{j}z\\
    =&2\int_{\mathcal{M}\times\mathcal{N}^{\ast}}A_{i}(\gamma_{i}D_{i} \beta_{11}^{ij})D_{i}A_{i}z\,D_{j}A_{j}z+\frac{h^{2}}{2}\int_{\mathcal{M}\times\mathcal{N}^{\ast}}D_{i}(\gamma_{i}D_{i}\beta_{11}^{ij})D_{i}^{2}zD_{j}A_{j}z.
\end{split}
\end{equation}
where we have used that $D_{j}A_{j}z=0$ on $\partial_{i}\mathcal{M}$ and the identity \eqref{A:product}.
\par Similarly, for $I_{11}^{ij(b)}$, a discrete integration by parts with respect to the average operator $A_{j}$ and product rule given by \eqref{D:product} yield
\begin{equation*}
\begin{split}
    I_{11}^{ij(b)}=&2\int_{\overline{\mathcal{M}}_{ij}\times\mathcal{N}^{\ast}}A_{j}(\gamma_{i}A_{i}\beta_{11}^{ij}\,D_{i}z)\,D_{ij}^{2}z\\
    =&2\int_{\overline{\mathcal{M}}_{ij}\times\mathcal{N}^{\ast}}A_{j}(\gamma_{i}A_{i}\beta_{11}^{ij})\,A_{j}D_{i}z\,D_{ij}^{2}z+\frac{h^{2}}{2}\int_{\overline{\mathcal{M}}_{ij}\times\mathcal{N}^{\ast}}D_{j}(\gamma_{i}A_{i}\beta_{11}^{ij})\,|D_{ij}^{2}z|^{2}.
    \end{split}
\end{equation*}
Now, by virtue of \eqref{D:product} we have $D_{j}\left( |D_{i}z|^{2}\right)=2A_{j}D_{i}z\,D_{ij}^{2}z$, then we can apply a discrete integration by parts with respect to the difference operator $D_{j}$ on the first term from $I_{11}^{ij(b)}$ to get
\begin{equation}\label{eq:non:diagonal:I_{11b}}
\begin{split}
I_{11}^{ij(b)}=&\int_{\overline{\mathcal{M}}_{ij}\times\mathcal{N}^{\ast}}A_{j}(\gamma_{i}A_{i}\beta_{11}^{ij})\,D_{j}\left( |D_{i}z|^{2}\right)+\frac{h^{2}}{2}\int_{\overline{\mathcal{M}}_{ij}\times\mathcal{N}^{\ast}}D_{j}(\gamma_{i}A_{i}\beta_{11}^{ij})\,|D_{ij}^{2}z|^{2}\\
=&-\int_{\mathcal{M}^{\ast}_{i}\times\mathcal{N}^{\ast}}D_{j}A_{j}(\gamma_{i}A_{i}\beta_{11}^{ij})\, |D_{i}z|^{2}+\frac{h^{2}}{2}\int_{\overline{\mathcal{M}}_{ij}\times\mathcal{N}^{\ast}}D_{j}(\gamma_{i}A_{i}\beta_{11}^{ij})\,|D_{ij}^{2}z|^{2},
    \end{split}
\end{equation}
since $D_{i}u=0$ on $\partial_{j}\mathcal{M}^{\ast}_{i}$. Thus, combining \eqref{eq:non:diagonal:I_{11a}} and \eqref{eq:non:diagonal:I_{11b}} we obtain that $I_{11}^{ij}$ can be written as
\begin{equation}
    \begin{split}
       I_{11}^{ij}=&-2\int_{\mathcal{M}\times\mathcal{N}^{\ast}}A_{i}(\gamma_{i}D_{i} \beta_{11}^{ij})D_{i}A_{i}z\,D_{j}A_{j}z-\frac{h^{2}}{2}\int_{\mathcal{M}\times\mathcal{N}^{\ast}}D_{i}(\gamma_{i}D_{i}\beta_{11}^{ij})D_{i}^{2}zD_{j}A_{j}z\\
      &+\int_{\mathcal{M}^{\ast}_{i}\times\mathcal{N}^{\ast}}D_{j}A_{j}(\gamma_{i}A_{i}\beta_{11}^{ij})\, |D_{i}z|^{2}-\frac{h^{2}}{2}\int_{\overline{\mathcal{M}}_{ij}\times\mathcal{N}^{\ast}}D_{j}(\gamma_{i}A_{i}\beta_{11}^{ij})\,|D_{ij}^{2}z|^{2}.
    \end{split}
\end{equation}
Recalling that $\beta_{11}^{ij}:=r^{2}A_{i}^{2}\rho\,D_{j}A_{j}\rho$, Theorem \ref{theo:weight:estimates} and Corollary \ref{Cor:derivative:wrt:x:2r} yield the following estimates
\begin{equation}
    \begin{split}
A_{i}(\gamma_{i}D_{i}\beta_{11}^{ij})&=-s\lambda^{2}\gamma_{i}\gamma_{j}\xi\,\partial_{i}\psi\,\partial_{j}\psi+s\lambda\gamma_{j}\xi\,\mathcal{O}(1)+s\mathcal{O}_{\lambda}((sh)^{2}),\\
        D_{i}(\gamma_{i}D_{i}\beta_{11}^{ij})&=s\mathcal{O}_{\lambda}(1),\\
        D_{j}A_{j}(\gamma_{i}A_{i}\beta_{11}^{ij})&=-s\lambda^{2}\gamma_{i}\gamma_{j}\xi(\partial_{j}\psi)^{2}+s\lambda\gamma_{j}\xi\,\mathcal{O}(1)+s\mathcal{O}_{\lambda}((sh)^{2}),\\
        D_{j}(\gamma_{i}A_{i}\beta_{11}^{ij})&=-s\lambda^{2}\gamma_{i}\gamma_{j}\xi(\partial_{j}\psi)^{2}+s\lambda\gamma_{j}\xi\,\mathcal{O}(1)+s\mathcal{O}_{\lambda}((sh)^{2}).
    \end{split}
\end{equation}
Therefore, by using the estimates above we get
\begin{equation}
    \begin{split}
        I_{11}^{ij}=-\int_{\mathcal{M}_{i}^{\ast}\times\mathcal{N}^{\ast}}s\lambda^{2}\gamma_{i}\gamma_{j}\xi(\partial_{j}\psi)^{2}|D_{i}z|^{2}+V_{11}^{ij}+W_{11}^{ij},
    \end{split}
\end{equation}
where
\begin{equation}
\begin{split}
    W_{11}^{ij}:=&\int_{\mathcal{M}\times\mathcal{N}^{\ast}}\left(s\lambda\varphi\mathcal{O}(1)+s\mathcal{O}_{\lambda}((sh)^{2})\right)\,A_{i}D_{i}z\,D_{j}A_{j}z+\int_{\mathcal{M}\times\mathcal{N}^{\ast}}h\mathcal{O}_{\lambda}(sh)\,D_{i}^{2}z\,D_{j}A_{j}z\\
    &+\int_{\overline{\mathcal{M}}_{ij}\times\mathcal{N}^{\ast}}\left(sh^{2}\lambda\xi\mathcal{O}(1)+sh^{2}\mathcal{O}_{\lambda}((sh)^{2})\right)\,|D_{ij}^{2}z|^{2}\\
    &+\int_{\mathcal{M}_{i}^{\ast}\times\mathcal{N}^{\ast}}(s\lambda\xi\mathcal{O}(1)+s\mathcal{O}_{\lambda}((sh)^{2}))\, |D_{i}z|^{2},
    \end{split}
\end{equation}
\begin{equation}
    V_{11}^{ij}:=\int_{\mathcal{M}\times\mathcal{N}^{\ast}}2s\lambda^{2}\gamma_{i}\gamma_{j}\xi\partial_{j}\psi\partial_{i}\psi\,A_{i}D_{i}z\,D_{j}A_{j}z+\frac{1}{2}\int_{\overline{\mathcal{M}}_{ij}\times\mathcal{N}^{\ast}} sh^{2}\gamma_{i}\gamma_{j}\lambda^{2}\xi(\partial_{j}\psi)^{2}\,|D_{ij}^{2}z|^{2},
\end{equation}

\par On the other hand, we note that 
\begin{equation}
    \begin{split}
        \int_{\mathcal{M}\times\mathcal{N}^{\ast}}h\mathcal{O}_{\lambda}(sh)\,D_{i}^{2}z\,D_{j}A_{j}z\leq &\int_{\mathcal{M}\times\mathcal{N}^{\ast}}h^{2}\mathcal{O}_{\lambda}(sh)\,|D_{i}^{2}z|^{2}+\int_{\mathcal{M}\times\mathcal{N}^{\ast}}\mathcal{O}_{\lambda}(sh)\,|D_{j}A_{j}z|^{2}.
    \end{split}
\end{equation}
Thus, $W_{11}^{ij}$ can be bounded as follows
\begin{equation}
\begin{split}
    |W_{11}^{ij}|:=&\int_{\mathcal{M}\times\mathcal{N}^{\ast}}\left(s\lambda\varphi\mathcal{O}(1)+\mathcal{O}_{\lambda}(sh)+s\mathcal{O}_{\lambda}((sh)^{2})\right)\,|D_{i}A_{i}z|^{2}+\int_{\mathcal{M}\times\mathcal{N}^{\ast}}h^{2}\mathcal{O}_{\lambda}(sh)\,|D_{i}^{2}z|^{2}\\
    &+\int_{\overline{\mathcal{M}}_{ij}\times\mathcal{N}^{\ast}}\left(sh^{2}\lambda\xi\mathcal{O}(1)+sh^{2}\mathcal{O}_{\lambda}((sh)^{2})\right)\,|D_{ij}^{2}z|^{2}\\
    &+\int_{\mathcal{M}^{\ast}_{i}\times\mathcal{N}^{\ast}}(s\lambda\xi\mathcal{O}(1)+s\mathcal{O}_{\lambda}((sh)^{2}))\, |D_{i}z|^{2}\\
    \leq &\int_{\mathcal{M}_{i}^{\ast}\times\mathcal{N}^{\ast}}\left(s\lambda\xi\mathcal{O}(1)+\mathcal{O}_{\lambda}(sh)+s\mathcal{O}_{\lambda}((sh)^{2})\right)\,|D_{i}z|^{2}-\int_{\partial_{i}\mathcal{M}\times\mathcal{N}^{\ast}}\mathcal{O}_{\lambda}(sh)t_{r}^{i}(|D_{i}z|^{2})n_{i}\\
    &+\int_{\mathcal{M}\times\mathcal{N}^{\ast}}h^{2}\mathcal{O}_{\lambda}(sh)\,|D_{i}^{2}z|^{2}+\int_{\overline{\mathcal{M}}_{ij}\times\mathcal{N}^{\ast}}\left(sh^{2}\lambda\varphi\mathcal{O}(1)+sh^{2}\mathcal{O}_{\lambda}((sh)^{2})\right)\,|D_{ij}^{2}z|^{2}\\
    :=&X_{11}^{ij}.
    \end{split}
\end{equation}
From $I_{11}^{ii}$, the leading coefficient of $|D_{i}z|^{2}$ is $s\lambda^{2}\gamma_{i}^{2}\xi(\partial_{i}\psi)^{2}$. From $I_{11}^{ij}$ for each $j\ne i$, the leading coefficient of $|D_{i}z|^{2}$ is $
    -s\lambda^{2}\gamma_{i}\gamma_{j}\xi(\partial_{j}\psi)^{2}$.
Summing over $j\ne i$ we get 
$\displaystyle
    -s\lambda^{2}\gamma_{i}\xi\sum_{j\ne i}\gamma_{j}(\partial_{j}\psi)^{2}$. Then, $\displaystyle s\lambda^{2}\gamma_{i}^{2}\xi(\partial_{i}\psi)^{2}-s\lambda^{2}\gamma_{i}\xi\sum_{j\ne i}\gamma_{j}(\partial_{j}\psi)^{2}=s\lambda^{2}\gamma_{i}\xi\bigg[2\gamma_{i}(\partial_{i}\psi)^{2}-\sum_{j=1}^{d}\gamma_{j}(\partial_{j}\psi)^{2}\bigg].$
Thus, using the notation $|\nabla_{\Gamma}\psi|^{2}:=\sum_{j=1}^{d}\gamma_{j}(\partial_{j}\psi)^{2}$, the final estimate for $I_{11}$ is
\begin{equation}\label{eq:I11:final}
\begin{aligned}
I_{11}=&\sum_{i=1}^{d}\int_{\mathcal{M}_{i}^{\ast}\times\mathcal{N}^{\ast}}s\lambda^{2}\gamma_{i}\xi\Big[2\gamma_{i}(\partial_{i}\psi)^{2}-|\nabla_{\Gamma}\psi|^{2}\Big]|D_{i}z|^{2}\\
&-2s\lambda\sum_{i=1}^{d}\int_{\mathcal{M}_{i}^{\ast}\times\mathcal{N}^{\ast}}\gamma_{i}\xi\,\partial_{i}\psi\,\partial_{i}\gamma_{i}\,|D_{i}z|^{2}+\sum_{i\ne j}V_{11}^{ij}+W_{11}+Y_{11},
\end{aligned}
\end{equation}
where  $V_{11}^{ij}$ is given by
\begin{equation}\label{eq:V11}
V_{11}^{ij}:=2\int_{\mathcal{M}\times\mathcal{N}^{\ast}}s\lambda^{2}\gamma_{i}\gamma_{j}\xi\,\partial_{i}\psi\,\partial_{j}\psi\,D_{i}A_{i}z\,D_{j}A_{j}z+\frac{h^{2}}{2}\int_{\overline{\mathcal{M}}_{ij}\times\mathcal{N}^{\ast}}s\lambda^{2}\gamma_{i}\gamma_{j}\xi(\partial_{j}\psi)^{2}\,|D_{ij}^{2}z|^{2},
\end{equation}
and 
\begin{equation}\label{eq:W11}
\begin{aligned}
W_{11}:=&\sum_{i=1}^{d}\int_{\mathcal{M}_{i}^{\ast}\times\mathcal{N}^{\ast}}\Big(s\lambda\,\mathcal{O}_{\lambda}(1)+s\,\mathcal{O}_{\lambda}((sh)^{2})\Big)|D_{i}z|^{2}+\sum_{i=1}^{d}h\int_{\mathcal{M}\times\mathcal{N}^{\ast}}\mathcal{O}_{\lambda}(sh)\,|D_{i}^{2}z|^{2}\\
&+\sum_{i\ne j}\int_{\mathcal{M}\times\mathcal{N}^{\ast}}\Big(s\lambda\gamma_{j}\xi\,\mathcal{O}(1)+s\,\mathcal{O}_{\lambda}((sh)^{2})\Big)D_{i}A_{i}z\,D_{j}A_{j}z\\
&+\sum_{i\ne j}\int_{\mathcal{M}\times\mathcal{N}^{\ast}}h\,\mathcal{O}_{\lambda}(sh)\,D_{i}^{2}z\,D_{j}A_{j}z\\
&+\sum_{i\ne j}\int_{\overline{\mathcal{M}}_{ij}\times\mathcal{N}^{\ast}}\Big(sh^{2}\lambda\gamma_{j}\xi\,\mathcal{O}(1)+sh^{2}\mathcal{O}_{\lambda}((sh)^{2})\Big)|D_{ij}^{2}z|^{2},
\end{aligned}
\end{equation}
and the boundary terms are
\begin{equation}\label{eq:Y11}
Y_{11}:=\sum_{i=1}^{d}\int_{\partial_{i}\mathcal{M}\times\mathcal{N}^{\ast}}\Big(-s\lambda\gamma_{i}^{2}\xi\,\partial_{i}\psi+\mathcal{O}_{\lambda}(sh)\Big)\,\mathrm{tr}^{i}(|D_{i}z|^{2})\,n_{i}.
\end{equation}

\begin{remark}
\begin{itemize}
    \item
    The coefficient of $|D_{i}z|^{2}$ after summing over $j$ is
    \begin{equation*}
        s\lambda^{2}\gamma_{i}\xi\bigg[2\gamma_{i}(\partial_{i}\psi)^{2}
        -\sum_{j=1}^{d}\gamma_{j}(\partial_{j}\psi)^{2}\bigg].
    \end{equation*}
    When $\gamma_{i}=\gamma$ for all $i$, this equals
    $s\lambda^{2}\gamma^{2}\xi[2(\partial_{i}\psi)^{2}-|\nabla\psi|^{2}]$.
    In dimension $d=1$ this is positive ($+s\lambda^{2}\gamma^{2}\xi|\psi'|^{2}$),
    but for $d\geq 2$ it is not sign-definite.  It becomes positive when it is combined with the leading term of $I_{12}$ (see the remark at the end of the estimate $I_{12}$), and the crossed terms $V_{11}^{ij}$ are treated after \eqref{eq:combined:recall}.
    \item The boundary term
    $Y_{11}$ has leading contribution
    $-s\lambda\gamma_{i}^{2}\xi(\partial_{i}\psi\cdot n_{i})\,
    \mathrm{tr}^{i}(|D_{i}z|^{2})$,
    which is positive under the geometric condition
    $\partial_{i}\psi\cdot n_{i}<0$ on $\partial_{i}\mathcal{M}$,
    for $\lambda$ large enough.
\end{itemize}
\end{remark}

\subsection{Estimate of $I_{12}$ }
We set $\beta_{12}^{i}:=-2s\,rA_{i}^{2}\rho\,\Delta_{\Gamma}\xi$. Recall that 
\begin{equation*}
    \begin{split}
        A_{1}z:=&\sum_{i=1}^{d}\trm(rA_{i}^{2}\rho\,D_{i}(\gamma_{i}D_{i}z))\,\text{ and } \,
        B_{2}z:=-2\trm(s\Delta_{\Gamma}\xi z).
    \end{split}
\end{equation*}
Thanks to \cite[Proposition 2, equation (2.4)]{LMPZ:2023}, we can shift the discrete 
time variable to obtain
\begin{equation*}
I_{12}=\sum_{i=1}^{d}\int_{\mathcal{M}\times\mathcal{N}^{\ast}}\beta_{12}^{i}\,z\,D_{i}(\gamma_{i}D_{i}z)=:\sum_{i=1}^{d}I^{i}_{12}.
\end{equation*}
An integration by parts with respect to the difference operator $D_{i}$ yields \break $\displaystyle 
    I_{12}^{i}=-\int_{\mathcal{M}^{\ast}_{i}\times\mathcal{N}^{\ast}}\gamma_{i}D_{i}( \beta_{12}^{i}\,z)\,D_{i}z$,
where we have used that $z=0$ on $\partial_{i}\mathcal{M}$. Now, by using  we write 
    \begin{equation*}
    \begin{split}
    I^{i}_{12}&=-\int_{\mathcal{M}^{\ast}_{i}\times\mathcal{N}^{\ast}}\gamma_{i}D_{i}\beta_{12}^{i}\,A_{i}z\,D_{i}z-\int_{\mathcal{M}^{\ast}_{i}\times\mathcal{N}^{\ast}}\gamma_{i}A_{i}\beta_{12}^{i}\,|D_{i}z|^{2}.
    \end{split}
\end{equation*}
Then, noting that $D_{i}(|z|^{2})=2A_{i}z\,D_{i}z$, $I_{12}^{i}$ can be written as follows
\begin{equation*}
    \begin{split}
    I_{12}^{i}&=-\frac{1}{2}\int_{\mathcal{M}^{\ast}_{i}\times\mathcal{N}^{\ast}}\gamma_{i}D_{i}\beta_{12}^{i}\,D_{i}(|z|^{2})-\int_{\mathcal{M}^{\ast}_{i}\times\mathcal{N}^{\ast}}\gamma_{i}A_{i}(\beta_{12}^{i})|D_{i}z|^{2}.
    \end{split}
\end{equation*}
An integration by parts with respect to the difference operator $D_{i}$(\cite[Lemma 2.2 equation (15)]{LDOP-2021}) on the first integral above gives
\begin{equation}
    \begin{split}
    I_{12}^{i}&=\frac{1}{2}\int_{\mathcal{M}\times\mathcal{N}^{\ast}}D_{i}(\gamma_{i}D_{i}\beta_{12}^{i})\,|z|^{2}-\int_{\mathcal{M}^{\ast}_{i}\times\mathcal{N}^{\ast}}\gamma_{i}A_{i}(\beta_{12}^{i})|D_{i}z|^{2}.
    \end{split}
\end{equation}
By Theorem \ref{theo:weight:estimates} we have $rA_{i}^{2}\rho=1+\mathcal{O}_{\lambda}((sh)^{2})$, then  $\beta_{12}^{ij}=-2s\lambda^{2}\xi|\nabla_{\Gamma}\psi|^{2}-2s\lambda\xi\Delta_{\Gamma} \psi+s\mathcal{O}_{\lambda}((sh)^{2})$ where we have used that $\Delta_{\Gamma}\xi=\lambda\xi(\lambda|\nabla_{\Gamma}\psi|^{2}+\Delta_{\Gamma}\psi)$. 
Moreover,
\begin{equation}\label{eq:beta12:derived}
    \begin{split}
        A_{i}(\beta_{12}^{i})&=-2s\lambda^{2}\xi\,|\nabla_{\Gamma}\psi|^{2}
        -2s\lambda\xi\,\Delta_{\Gamma}\psi+s\,\mathcal{O}_{\lambda}(h^{2}+(sh)^{2}),\\
        D_{i}(\beta_{12}^{i})&=s\,\mathcal{O}_{\lambda}(1),\\
        D_{i}(\gamma_{i}D_{i}\beta_{12}^{i})&=s\,\mathcal{O}_{\lambda}(1).
    \end{split}
\end{equation}
thanks to \cite[Lemma 2.1]{LDOP-2021} and Theorem \ref{theo:weight:estimates}. Thus, $I_{12}$ can be estimate as
\begin{equation}\label{eq:I12:final}
    I_{12}=2\sum_{i=1}^{d}\int_{\mathcal{M}_{i}^{\ast}\times\mathcal{N}^{\ast}}s\lambda^{2}
    \gamma_{i}\xi\,|\nabla_{\Gamma}\psi|^{2}\,|D_{i}z|^{2}+X_{12},
\end{equation}
where the remainder is given by
\begin{equation}\label{eq:X12}
    X_{12}:=\int_{\mathcal{M}\times\mathcal{N}^{\ast}}s\,\mathcal{O}_{\lambda}(1)\,|z|^{2}
    +\sum_{i=1}^{d}\int_{\mathcal{M}_{i}^{\ast}\times\mathcal{N}^{\ast}}
    \Big(2s\lambda\gamma_{i}\xi\,\Delta_{\Gamma}\psi
    +s\,\mathcal{O}_{\lambda}(h^{2}+(sh)^{2})\Big)|D_{i}z|^{2}.
\end{equation}

\begin{remark}
The leading term of $I_{12}$ is
$2s\lambda^{2}\sum_{i}\gamma_{i}\xi\,|\nabla_{\Gamma}\psi|^{2}|D_{i}z|^{2}$,
which is positive and of order $s\lambda^{2}$.
When combined with the leading term of $I_{11}$
(which contains
$s\lambda^{2}\gamma_{i}\xi[2\gamma_{i}(\partial_{i}\psi)^{2}
-|\nabla_{\Gamma}\psi|^{2}]|D_{i}z|^{2}$),
the leading coefficient of $|D_{i}z|^{2}$ becomes
\begin{equation*}
    s\lambda^{2}\gamma_{i}\xi
    \Big[2\gamma_{i}(\partial_{i}\psi)^{2}
    -|\nabla_{\Gamma}\psi|^{2}+2|\nabla_{\Gamma}\psi|^{2}\Big]
    =s\lambda^{2}\gamma_{i}\xi
    \Big[2\gamma_{i}(\partial_{i}\psi)^{2}
    +|\nabla_{\Gamma}\psi|^{2}\Big],
\end{equation*}
which is positive. The sub-leading term $2s\lambda\gamma_{i}\xi\,\Delta_{\Gamma}\psi$
is of order $s\lambda$ and is absorbed for $\lambda$ large.
\end{remark}
\subsection{Estimate of $I_{13}$}
For $I_{13}$, let us set $\beta^{i}_{13}:=rA_{i}^{2}\rho$. Then, we have 
\begin{equation*}
I_{13}=\sum_{i=1}^{d}\int_{\mathcal{M}\times\mathcal{N}}\trm(\beta_{13}^{i}\,D_{i}(\gamma_{i}D_{i}z))D_{t}z=:\sum_{i=1}^{d}I^{i}_{13}.
\end{equation*}
Followed by a discrete integration by parts with respect to the difference operator $D_{i}$, $I_{13}$ can be written as
\begin{equation*}
\begin{split}
I^{i}_{13}=&-\int_{\mathcal{M}_{i}^{\ast}\times\mathcal{N}}\trm(\gamma_{i}D_{i}z)\,D_{i}(\trm\beta_{13}^{i}\,D_{t}z),
\end{split}
\end{equation*}
since $D_{t}z=0$ on $\partial\mathcal{M}$. We note that the product rule, \cite[Lemma 2.1]{LDOP-2021}, for the difference operator $D_{i}$ yields
\begin{equation*}
\begin{aligned}
    I_{13}^{i}=&-\int_{\mathcal{M}_{i}^{\ast}\times\mathcal{N}}\trm(\gamma_{i}D_{i}z)D_{i}(\trm\beta_{13}^{i})\,A_{i}D_{t}z\,-\int_{\mathcal{M}_{i}^{\ast}\times\mathcal{N}}\trm(\gamma_{i}D_{i}z)A_{i}\trm(\beta^{i}_{13})\,D_{it}^{2}z\\
    =:&I_{13}^{i(a)}+I_{13}^{i(b)}.
    \end{aligned}
\end{equation*}
Now, by using the identity $\trm zD_{t}z=\frac{1}{2}D_{t}(|z|^{2})-\frac{1}{2}\Delta t|D_{t}z|^{2}$, for $I_{13}^{(b)}$ we get
\begin{equation*}
\begin{split}
    I_{13}^{i(b)}=&-\frac{1}{2}\int_{\mathcal{M}_{i}^{\ast}\times\mathcal{N}}\trm(\gamma_{i}A_{i}\beta_{13}^{i})\,D_{t}(|D_{i}z|^{2})+\frac{\Delta t}{2}\int_{\mathcal{M}_{i}^{\ast}\times\mathcal{N}}\trm(\gamma_{i}A_{i}\beta_{13}^{i})\,|D_{it}^{2}z|^{2}.
    \end{split}
\end{equation*}
Using the integration by parts with respect to the time difference operator $D_{t}$ given by \cite[Proposition 2, equation (2.5)]{LMPZ:2023}, on the first integral of $I_{13}^{i(b)}$, we obtain
\begin{equation*}
\begin{aligned}
    I_{13}^{i(b)}=&\frac{1}{2}\int_{\mathcal{M}_{i}^{\ast}\times\mathcal{N}}\gamma_{i}D_{t}(A_{i}\beta_{13}^{i})\,\trp(|D_{i}z|^{2})-\frac{1}{2}\int_{\mathcal{M}_{i}^{\ast}\times\partial\mathcal{N}}\trp(\gamma_{i}A_{i}\beta_{13}^{i}|D_{i}z|^{2})n\\
    &+\frac{\Delta t}{2}\int_{\mathcal{M}_{i}^{\ast}\times\mathcal{N}}\trm(\gamma_{i}A_{i}\beta_{13}^{i})\,|D_{it}^{2}z|^{2}.
    \end{aligned}
\end{equation*}
By virtue of Theorem \ref{theo:fully:weight:estimates} we note that $A_{i}\trpm\beta_{13}= C_{\lambda}>0$ provided $\tau h(\delta T^{2})^{-1}\leq \varepsilon_{1}(\lambda)$ small enough, we thus have the following lower bound for $I_{13}^{i(b)}$
\begin{equation*}
\begin{split}
    I_{13}^{i(b)}\geq&\frac{1}{2}\int_{\mathcal{M}^{\ast}_{i}\times\mathcal{N}}\gamma_{i}D_{t}(A_{i}\beta_{13}^{i})\,\trp(|D_{i}z|^{2})-C_{\lambda}\int_{\mathcal{M}^{\ast}_{i}\times\partial\mathcal{N}^{+}}\gamma_{i}|D_{i}\trp z|^{2}+C_{\lambda}\Delta t \int_{\mathcal{M}_{i}^{\ast}\times\mathcal{N}}|D_{it}^{2}z|^{2},
    \end{split}
\end{equation*}
where we have dropped the left time boundary term since it is positive 
We note that provided $\frac{\Delta t \tau}{T^{3}\delta^{2}}\leq \frac{1}{2}$, from Theorem \eqref{theo:fully:weight:estimates}, we have 
$\gamma_{i}D_{t}(A_{i}\beta^{i}_{13})=T\trm(\theta(sh)^{2})\mathcal{O}_{\lambda}(1)+\left( \frac{\tau\Delta t}{\delta^{3}T^{4}}\right)\left( \frac{\tau h}{\delta T^{2}}\right)\mathcal{O}_{\lambda}(1)$. Then, we deduce
\begin{equation*}
\begin{split}
    I_{13}^{i(b)}\geq&-\int_{\mathcal{M}^{\ast}_{i}\times\mathcal{N}}\left( T\theta(sh)^{2}\mathcal{O}_{\lambda}(1)+\left( \frac{\tau\Delta t}{\delta^{3}T^{4}}\right)\left( \frac{\tau h}{\delta T^{2}}\right)\mathcal{O}_{\lambda}(1)\right)\,|D_{i}z|^{2}-C_{\lambda}\int_{\mathcal{M}_{i}^{\ast}\times\partial\mathcal{N}^{+}}|D_{i}\trp z|^{2}\\
    &-C_{\lambda}\Delta t\int_{\mathcal{M}_{i}^{\ast}\times\mathcal{N}}\left( \frac{\tau\Delta t}{\delta^{3}T^{4}}\right)\left( \frac{\tau h}{\delta T^{2}}\right)\,|D_{it}^{2}z|^{2}+C_{\lambda}\Delta t \int_{\mathcal{M}_{i}^{\ast}\times\mathcal{N}}|D_{it}^{2}z|^{2},
    \end{split}
\end{equation*}
 where we used $\trp(|D_{i}z|^{2})\leq 2\trm(|D_{i}z|^{2})+2(\Delta t)^{2}|D_{it}z|^{2}$. Hence, taking $\epsilon_{1}(\lambda)$ small enough, for $I_{13}^{i(b)}$ we have
\begin{equation}\label{eq:I13:a2}
\begin{split}
    I_{13}^{i(b)}\geq&-\int_{\mathcal{M}_{i}^{\ast}\times\mathcal{N}}\left( T\theta(sh)^{2}\mathcal{O}_{\lambda}(1)+\left( \frac{\tau\Delta t}{\delta^{3}T^{4}}\right)\left( \frac{\tau h}{\delta T^{2}}\right)\mathcal{O}_{\lambda}(1)\right)\,|D_{i}z|^{2}\\
    &-C_{\lambda}\int_{\mathcal{M}_{i}^{\ast}\times\partial\mathcal{N}^{+}}|D_{i}\trp z|^{2}+C_{\lambda}\Delta t \int_{\mathcal{M}_{i}^{\ast}\times\mathcal{N}}|D_{it}^{2}z|^{2}.
    \end{split}
\end{equation}
\par On the other hand, by virtue of Theorem \eqref{theo:weight:estimates}  we note that $D_{i}\beta_{13}^{i}=s\mathcal{O}_{\lambda}((sh)^{2})$. Then, this estimate and the Young's inequality for $I_{13}^{i(a)}$ yield 
\begin{equation*}
    \begin{split}
      |I_{13}^{i(a)}|\leq&C_{\lambda} \int_{\mathcal{M}^{\ast}_{i}\times\mathcal{N}}\trm(s^{-1}(sh)^{2}|D_{t}A_{i}z|^{2}+C_{\lambda}\int_{\mathcal{M}^{\ast}_{i}\times\mathcal{N}^{\ast}}s(sh)^{2}|D_{i}z|^{2},
    \end{split}
\end{equation*}
where in the last integral above we have used \cite[Proposition 2, equation (2.4)]{LMPZ:2023} to shift the discrete time variable. Now, using that $|D_{t}A_{i}z|^{2}\leq A_{i}(|D_{t}z|^{2})$ and an integration by parts with respect to the average operator $A_{i}$ (\cite[Lemma 2.2, equation (16)]{LDOP-2021}) on the first integral of the above expression, we obtain
\begin{equation}\label{eq:I13:a1}
    \begin{split}
      |I_{13}^{i(a)}|\leq&C_{\lambda} \left(\int_{\mathcal{M}\times\mathcal{N}}\trm(s^{-1}(sh)^{2})|D_{t}z|^{2}+\int_{\mathcal{M}_{i}^{\ast}\times\mathcal{N}^{\ast}}s(sh)^{2}|D_{i}z|^{2}\right),
    \end{split}
\end{equation}
since $D_{t}z=0$ on $\partial\mathcal{M}$. Thus, collecting \eqref{eq:I13:a2} and \eqref{eq:I13:a1} we obtain that $I_{13}^{i}$ can be bounded as $\displaystyle I_{13}^{i}\geq -W_{13}$ where
\begin{equation}
    \begin{split}
          W_{13}=&\int_{\mathcal{M}^{\ast}_{i}\times\mathcal{N}^{\ast}}\left( T\theta(sh)^{2}\mathcal{O}_{\lambda}(1)+\left( \frac{\tau\Delta t}{\delta^{3}T^{4}}\right)\left( \frac{\tau h}{\delta T^{2}}\right)\mathcal{O}_{\lambda}(1)\right)\,|D_{i}z|^{2}\\
        &-C_{\lambda}\Delta t \int_{\mathcal{M}_{i}^{\ast}\times\mathcal{N}}|D_{it}^{2}z|^{2}+C_{\lambda}\int_{\mathcal{M}_{i}^{\ast}\times\mathcal{N}^{\ast}}s(sh)^{2}|D_{i}z|^{2}. 
    \end{split}
\end{equation}
Combining the estimates for $I_{13}^{i(a)}$ and $I_{13}^{i(b)}$,
and summing over $i=1,\ldots,d$, we obtain
\begin{equation}\label{eq:I13:final}
\begin{aligned}
I_{13}\geq&-\sum_{i=1}^{d}\int_{\mathcal{M}_{i}^{\ast}\times\mathcal{N}}
\bigg(T\theta(sh)^{2}\,\mathcal{O}_{\lambda}(1)
+\frac{\tau\Delta t}{\delta^{3}T^{4}}\cdot\frac{\tau h}{\delta T^{2}}\,
\mathcal{O}_{\lambda}(1)
+s(sh)^{2}\,\mathcal{O}_{\lambda}(1)\bigg)|D_{i}z|^{2}\\
&-\mathcal{O}_{\lambda}((sh)^{2})
\int_{\mathcal{M}\times\mathcal{N}}|D_{t}z|^{2}
+C_{\lambda}\Delta t\sum_{i=1}^{d}
\int_{\mathcal{M}_{i}^{\ast}\times\mathcal{N}}|D_{it}^{2}z|^{2}\\
&-C_{\lambda}\sum_{i=1}^{d}
\int_{\mathcal{M}_{i}^{\ast}\times\partial\mathcal{N}^{+}}
|D_{i}\trp z|^{2}.
\end{aligned}
\end{equation}
The positive $\Delta t\int|D_{it}^{2}z|^{2}$ term from $I_{13}^{i(b)}$
is kept on the right-hand side as a good term that can be
moved to the left-hand side of the Carleman estimate.

\subsection{Estimate of $I_{21}$}
Let us recall that  
\begin{equation*}
\begin{split}
    A_{2}z:=\sum_{i=1}^{d}\gamma_{i}\trm(rD_{i}^{2}\rho A_{i}^{2}z),\, \text{ and }\,
    B_{1}z:=2\sum_{i=1}^{d}\gamma_{i}\trm(rD_{i}A_{i}\rho D_{i}A_{i}z).
    \end{split}
\end{equation*}
Then, setting $\beta_{21}^{ij}:=2\gamma_{i}\gamma_{j}r^{2}D_{i}^{2}\rho D_{j}A_{j}\rho$  we have after shifting the time variable
\begin{equation*}
    I_{21}=\sum_{i,j=1}^ {d}\int_{\mathcal{M}\times\mathcal{N}^{\ast}}\beta_{21}^{ij}\,A_{i}^{2}z\,D_{j}A_{j}z=:\sum_{i,j=1}^{d}I_{21}^{ij}.
\end{equation*}
\par Let us estimate $I_{21}^{ij}$ for $i\ne j$. Integrating by parts with respect to the average operator $A_{i}$, and using that $D_{j}A_{j}z=0$ on $\partial_{i}\mathcal{M}$, we have 
\begin{equation*}
    \begin{split}
I_{21}^{ij}=\int_{\mathcal{M}_{i}^{\ast}\times\mathcal{N}^{\ast}}A_{i}(\beta_{21}^{ij}\,D_{j}A_{j}z)\,A_{i}z.
    \end{split}
\end{equation*}
Notice that $A_{i}(\beta_{21}^{ij}\,D_{j}A_{j}z)=A_{i}\beta_{21}^{ij}\,A_{ij}^{2}D_{j}z+\frac{h^{2}}{4}D_{i}\beta_{21}^{ij}\,D_{ij}^{2}A_{j}z$, thanks to \cite[Lemma 2.1, equation (14)]{LDOP-2021}. Then, $I_{21}^{ij}$ can be written as
\begin{equation}
    \begin{split}
        I_{21}^{ij}=&\int_{\mathcal{M}_{i}^{\ast}\times\mathcal{N}^{\ast}}A_{i}\beta_{21}^{ij}\,A_{ij}^{2}D_{j}z\,A_{i}z+\frac{h^{2}}{4}\int_{\mathcal{M}_{i}^{\ast}\times\mathcal{N}^{\ast}}D_{i}\beta_{21}^{ij}\,D_{ij}^{2}A_{j}z\,A_{i}z=:I_{21}^{(a)}+I_{21}^{(b)}.
    \end{split}
\end{equation}
\par The next step is to estimate $I_{21}^{(a)}$ and $I_{21}^{(b)}$ as follows. We apply an integration by parts on $I_{21}^{(a)}$ with respect to the average operator $A_{j}$, and use that $A_{i}z=0$ on $\partial_{j}\mathcal{M}$, to obtain $\displaystyle
    I_{21}^{(a)}=\int_{\overline{\mathcal{M}}_{ij}\times\mathcal{N}^{\ast}}A_{j}(A_{i}\beta_{21}^{ij}\,A_{i}z)\,A_{i}D_{j}z$. Then, by using \cite[Lemma 2.1, equation (13)]{LDOP-2021} we have
\begin{equation}
    I_{21}^{(a)}=\frac{1}{2}\int_{\overline{\mathcal{M}}_{ij}\times\mathcal{N}^{\ast}}A^{2}_{ij}\beta_{21}^{ij}\,D_{j}(|A_{i}z|^{2})+\frac{h^{2}}{4}\int_{\overline{\mathcal{M}}_{ij}\times\mathcal{N}^{\ast}}D_{j}A_{i}\beta_{21}^{ij}\,|A_{i}D_{j}z|^{2},
\end{equation}
Thus, applying an integration by parts with respect to the difference operator $D_{j}$ on the first integral in the above expression, we obtain
\begin{equation*}
\begin{split}
    I_{21}^{(a)}=&-\frac{1}{2}\int_{\mathcal{M}^{\ast}_{i}\times\mathcal{N}^{\ast}}D_{j}A^{2}_{ij}\beta_{21}^{ij}\,|A_{i}z|^{2}
    +\frac{h^{2}}{4}\int_{\overline{\mathcal{M}}_{ij}\times\mathcal{N}^{\ast}}D_{j}A_{i}\beta_{21}^{ij}\,|A_{i}D_{j}z|^{2},
    \end{split}
\end{equation*}
where we have used that $A_{i}z=0$ on $\partial_{j}\mathcal{M}$. Moreover, by using $|A_{i}z|^{2}=A_{i}(|z|^{2})-\frac{h^{2}}{4}|D_{i}z|^{2}$ we have
\begin{equation}
\begin{split}
    I_{21}^{(a)}=&-\frac{1}{2}\int_{\mathcal{M}^{\ast}_{i}\times\mathcal{N}^{\ast}}D_{j}A^{2}_{ij}\beta_{21}^{ij}\,A_{i}(|z|^{2})+\frac{h^{2}}{8}\int_{\mathcal{M}^{\ast}_{i}\times\mathcal{N}^{\ast}}D_{j}A^{2}_{ij}\beta_{21}^{ij}\,|D_{i}z|^{2}+\frac{h^{2}}{4}\int_{\overline{\mathcal{M}}_{ij}\times\mathcal{N}^{\ast}}D_{j}A_{i}\beta_{21}^{ij}\,|A_{i}D_{j}z|^{2}.
    \end{split}
\end{equation}
Applying an integration by parts, with respect to the average operator $A_{i}$ on the first term in the above expression, gives
\begin{equation}
\begin{split}
    I_{21}^{(a)}=&-\frac{1}{2}\int_{\mathcal{M}\times\mathcal{N}^{\ast}}A_{i}D_{j}A^{2}_{ij}\beta_{21}^{ij}\,|z|^{2}+\frac{h^{2}}{8}\int_{\mathcal{M}^{\ast}_{i}\times\mathcal{N}^{\ast}}D_{j}A^{2}_{ij}\beta_{21}^{ij}\,|D_{i}z|^{2}+\frac{h^{2}}{4}\int_{\overline{\mathcal{M}}_{ij}\times\mathcal{N}^{\ast}}D_{j}\beta_{21}^{ij}\,|A_{i}D_{j}z|^{2}.
    \end{split}
\end{equation}
Recalling that $\beta_{21}^{ij}:=2\gamma_{i}\gamma_{j}r^{2}D_{i}^{2}\rho D_{j}A_{j}\rho$, from Theorem \ref{theo:weight:estimates} we have
\begin{equation}\label{estimation:beta:21}
    \begin{split}
    \beta_{21}^{ij}=&-2s^{3}\lambda^{3}\xi^{3}\gamma_{i}\gamma_{j}(\partial_{i}\psi)^{2}\partial_{j}\psi+s^{2}\mathcal{O}_{\lambda}(1)+s^{3}\mathcal{O}_{\lambda}((sh)^{2}),\\
        D_{j}\beta_{21}^{ij}=&-6s^{3}\lambda^{4}\xi^{3}\gamma_{i}\gamma_{j}(\partial_{i}\psi)^{2}(\partial_{j}\psi)^{2}+6s^{2}\lambda^{3}\xi^{2}\mathcal{O}_{\lambda}(1)+s^{2}\mathcal{O}_{\lambda}(1)+s^{3}\mathcal{O}_{\lambda}((sh)^{2})=s^{3}\mathcal{O}_{\lambda}(1),\\
        A_{i}D_{j}A_{ij}^{2}\beta_{21}^{ij}=&-6s^{3}\lambda^{4}\xi^{3}(\partial_{i}\psi)^{2}\gamma_{i}\gamma_{j}(\partial_{j}\psi)^{2}+6s^{2}\lambda^{3}\xi^{2}\mathcal{O}_{\lambda}(1)+s^{2}\mathcal{O}_{\lambda}(1)+s^{3}\mathcal{O}_{\lambda}((sh)^{2}).
    \end{split}
\end{equation}
Thus, by using the estimates \eqref{estimation:beta:21}, we get
\begin{equation}\label{eq:J_{21a}}
        I_{21}^{(a)}=\int_{\mathcal{M}\times\mathcal{N}^{\ast}}3s^{3}\lambda^{4}\xi^{3}\gamma_{i}\gamma_{j}(\partial_{i}\psi)^{2}(\partial_{j}\psi)^{2}|z|^{2}+W_{21}^{ij(a)},
\end{equation}
where
\begin{equation}
    \begin{split}
        W_{21}^{ij(a)}:=&\int_{\mathcal{M}\times\mathcal{N}^{\ast}}\left( s^{2}\lambda^{3}\xi^{2}\mathcal{O}_{\lambda}(1)+s^{2}\mathcal{O}_{\lambda}(1)+s^{3}\mathcal{O}_{\lambda}((sh)^{2})\right)|z|^{2}\\
        &+\int_{\mathcal{M}_{i}^{\ast}\times\mathcal{N}^{\ast}}s\mathcal{O}_{\lambda}((sh)^{2})|D_{i}z|^{2}+\int_{\overline{\mathcal{M}}_{ij}\times\mathcal{N}^{\ast}}s\mathcal{O}_{\lambda}((sh)^{2})\,|A_{i}D_{j}z|^{2}.
        \end{split}
\end{equation}
\par Now, we estimate $I_{21}^{(b)}$. By using the estimates \eqref{estimation:beta:21} and applying  Young's inequality we note that
\begin{equation}
\begin{split}
    |I_{21}^{(b)}|\leq&\int_{\mathcal{M}_{i}^{\ast}\times\mathcal{N}^{\ast}}s\mathcal{O}_{\lambda}((sh)^{2})|D_{ij}^{2}A_{j}z|^{2}+\int_{\mathcal{M}_{i}^{\ast}\times\mathcal{N}^{\ast}}s\mathcal{O}_{\lambda}((sh)^{2})|A_{i}z|^{2}.
    \end{split}
\end{equation}
Using inequality $|D_{ij}^{2}A_{j}z|^{2}\leq A_{j}(|D_{ij}^{2}z|^{2})$ it follows that
\begin{equation}
|I_{21}^{(b)}|\leq\int_{\mathcal{M}_{i}^{\ast}\times\mathcal{N}^{\ast}}s\mathcal{O}_{\lambda}((sh)^{2})A_{j}(|D_{ij}^{2}z|^{2})+\int_{\mathcal{M}_{i}^{\ast}\times\mathcal{N}^{\ast}}s\mathcal{O}_{\lambda}((sh)^{2})A_{i}(|z|^{2}).
\end{equation}
Integrating by parts with respect to the average operators $A_{j}$ and $A_{i}$ in the above expression yield
\begin{equation}
|I_{21}^{(b)}|\leq\int_{\overline{\mathcal{M}}_{ij}\times\mathcal{N}^{\ast}}s\mathcal{O}_{\lambda}((sh)^{2})\,|D_{ij}^{2}z|^{2}+\int_{\partial_{j}\mathcal{M}_{i}^{\ast}\times\mathcal{N}^{\ast}}s\mathcal{O}_{\lambda}((sh)^{2})t_{r}^{j}(|D_{ij}^{2}z|^{2})+\int_{\mathcal{M}\times\mathcal{N}^{\ast}}s\mathcal{O}_{\lambda}((sh)^{2})|z|^{2}.
\end{equation}
Hence, when $i\ne j$ we have
\begin{equation}\label{eq:I21:ij:final}
I_{21}^{ij}=3s^{3}\lambda^{4}\int_{\mathcal{M}\times\mathcal{N}^{\ast}}\gamma_{i}\gamma_{j}\xi^{3}(\partial_{i}\psi)^{2}(\partial_{j}\psi)^{2}\,|z|^{2}+W_{21}^{ij},
\end{equation}
where 
\begin{equation}\label{eq:W21:ij:full}
\begin{aligned}
W_{21}^{ij}:=&\int_{\mathcal{M}\times\mathcal{N}^{\ast}}\Big(s^{3}\lambda^{3}\xi^{3}\mathcal{O}(1)+s^{2}\mathcal{O}_{\lambda}(1)\Big)|z|^{2}+\int_{\mathcal{M}_{i}^{\ast}\times\mathcal{N}^{\ast}}s\,\mathcal{O}_{\lambda}((sh)^{2})|D_{i}z|^{2}\\
&+\int_{\overline{\mathcal{M}}_{ij}\times\mathcal{N}^{\ast}}s\,\mathcal{O}_{\lambda}((sh)^{2})|A_{i}D_{j}z|^{2}+\int_{\overline{\mathcal{M}}_{ij}\times\mathcal{N}^{\ast}}s\,\mathcal{O}_{\lambda}((sh)^{2})|D_{ij}^{2}z|^{2}\\
&+\frac{h}{2}\int_{\partial_{j}\mathcal{M}_{i}^{\ast}\times\mathcal{N}^{\ast}}s\,\mathcal{O}_{\lambda}((sh)^{2})t_{r}^{j}(|D_{ij}^{2}z|^{2})+\frac{h}{2}\int_{\partial_{i}\mathcal{M}\times\mathcal{N}^{\ast}}s\,\mathcal{O}_{\lambda}((sh)^{2})t_{r}^{i}(|z|^{2}).
\end{aligned}
\end{equation}
\par On the other hand, for $j=i$ we have that
\begin{equation*}
I_{21}^{ii}:=\int_{\mathcal{M}\times\mathcal{N}^{\ast}}\beta_{21}^{ii}A_{i}^{2}zD_{i}A_{i}z.
\end{equation*}
Noting that $D_{i}(|A_{i}z|^{2})=2A_{i}^{2}zD_{i}A_{i}z$, a discrete integration by parts gives
\begin{equation*}
I_{21}^{ii}=-\frac{1}{2}\int_{\mathcal{M}^{\ast}_{i}\times\mathcal{N}^{\ast}}D_{i}\beta_{21}^{ii}|A_{i}z|^{2}+\frac{1}{2}\int_{\partial_{i}\mathcal{M}\times\mathcal{N}^{\ast}}\beta_{21}^{ii}t_{r}^{i}(|A_{i}z|^{2})n_{i}.
\end{equation*}
Substituting the identity $|A_{i}z|^{2}=A_{i}(|z|^{2})-\frac{h^{2}}{4}|D_{i}z|^{2}$ and applying an integration by parts with respect to the average operator $A_{i}$ we obtain
\begin{equation}\label{eq:I21:ii:expanded}
\begin{aligned}
I_{21}^{ii}=&-\frac{1}{2}\int_{\mathcal{M}\times\mathcal{N}^{\ast}}A_{i}(D_{i}\beta_{21}^{ii})\,|z|^{2}+\frac{h^{2}}{8}\int_{\mathcal{M}_{i}^{\ast}\times\mathcal{N}^{\ast}}D_{i}\beta_{21}^{ii}\,|D_{i}z|^{2}\\
&-\frac{h}{4}\int_{\partial_{i}\mathcal{M}\times\mathcal{N}^{\ast}}D_{i}\beta_{21}^{ii}\,t_{r}^{i}(|z|^{2})\,n_{i}+\frac{1}{2}\int_{\partial_{i}\mathcal{M}\times\mathcal{N}^{\ast}}\beta_{21}^{ii}\,t_{r}^{i}(|A_{i}z|^{2})\,n_{i}.
\end{aligned}
\end{equation}
Thanks to Theorem \ref{theo:weight:estimates}
\begin{equation}\label{eq:Di:beta21:ii:estimate}
D_{i}\beta_{21}^{ii}=-6s^{3}\lambda^{4}\gamma_{i}^{2}\xi^{3}(\partial_{i}\psi)^{4}+s^{3}\lambda^{3}\xi^{3}\mathcal{O}(1)+s^{2}\mathcal{O}_{\lambda}(1)+s^{3}\mathcal{O}_{\lambda}((sh)^{2}).
\end{equation}
and 
\begin{equation}\label{eq:Ai:Di:beta21:ii:estimate}
A_{i}(D_{i}\beta_{21}^{ii})=-6s^{3}\lambda^{4}\gamma_{i}^{2}\xi^{3}(\partial_{i}\psi)^{4}+s^{3}\lambda^{3}\xi\mathcal{O}(1)+s^{2}\mathcal{O}_{\lambda}(1)+s^{3}\mathcal{O}_{\lambda}((sh)^{2}).
\end{equation}
Thus, $I_{21}^{ii}$ is estimated as follows
\begin{equation}\label{eq:I21:ii:final}
\begin{aligned}
I_{21}^{ii}=&\;3s^{3}\lambda^{4}\int_{\mathcal{M}\times\mathcal{N}^{\ast}}\gamma_{i}^{2}\xi^{3}(\partial_{i}\psi)^{4}\,|z|^{2}\\
&+\int_{\mathcal{M}\times\mathcal{N}^{\ast}}\Big(s^{3}\mathcal{O}_{\lambda}(1)+s^{2}\mathcal{O}_{\lambda}(1)+s^{3}\mathcal{O}_{\lambda}((sh)^{2})\Big)|z|^{2}\\
&+s\,\mathcal{O}_{\lambda}((sh)^{2})\int_{\mathcal{M}_{i}^{\ast}\times\mathcal{N}^{\ast}}|D_{i}z|^{2}\\
&-\frac{h}{4}\int_{\partial_{i}\mathcal{M}\times\mathcal{N}^{\ast}}\mathcal{O}_{\lambda}(s^{3})\,t_{r}^{i}(|z|^{2})\,n_{i}+\int_{\partial_{i}\mathcal{M}\times\mathcal{N}^{\ast}}s^{3}\mathcal{O}_{\lambda}(1)\,t_{r}^{i}(|A_{i}z|^{2})\,n_{i}.
\end{aligned}
\end{equation}
Therefore, combining \eqref{eq:I21:ij:final} and \eqref{eq:I21:ii:final},
and using the algebraic identity
\begin{equation*}
\sum_{i\ne j}\gamma_{i}\gamma_{j}(\partial_{i}\psi)^{2}(\partial_{j}\psi)^{2}
+\sum_{i=1}^{d}\gamma_{i}^{2}(\partial_{i}\psi)^{4}
=\bigg(\sum_{i=1}^{d}\gamma_{i}(\partial_{i}\psi)^{2}\bigg)^{2},
\end{equation*}
we get
\begin{equation}
I_{21}\geq 3\lambda^{4}\int_{\mathcal{M}\times\mathcal{N}^{\ast}}s^{3}\xi^{3}\left(\sum_{i=1}^{d}\gamma_{i}(\partial_{i}\psi)^{2}\right)^{2}|z|^{2}-X_{21}-Y_{21},
\end{equation}
where
\begin{equation}\label{eq:X21:final}
\begin{aligned}
X_{21}:=&\int_{\mathcal{M}\times\mathcal{N}^{\ast}}
\Big(s^{3}\lambda^{3}\xi\mathcal{O}(1)+s^{2}\mathcal{O}_{\lambda}(1)\Big)|z|^{2}
+\sum_{i=1}^{d}
\int_{\mathcal{M}_{i}^{\ast}\times\mathcal{N}^{\ast}}s\,\mathcal{O}_{\lambda}((sh)^{2})|D_{i}z|^{2}\\
&+\sum_{i\ne j}
\int_{\overline{\mathcal{M}}_{ij}\times\mathcal{N}^{\ast}}s\,\mathcal{O}_{\lambda}((sh)^{2})|A_{i}D_{j}z|^{2}
+\sum_{i\ne j}
\int_{\overline{\mathcal{M}}_{ij}\times\mathcal{N}^{\ast}}s\,\mathcal{O}_{\lambda}((sh)^{2})|D_{ij}^{2}z|^{2},
\end{aligned}
\end{equation}
and the boundary remainder is
\begin{equation}\label{eq:Y21:final}
\begin{aligned}
Y_{21}:=&\sum_{i=1}^{d}\bigg[\frac{h}{4}
\int_{\partial_{i}\mathcal{M}\times\mathcal{N}^{\ast}}
s^{3}\mathcal{O}_{\lambda}(1)\,t_{r}^{i}(|z|^{2})\,n_{i}
+\int_{\partial_{i}\mathcal{M}\times\mathcal{N}^{\ast}}
s^{3}\mathcal{O}_{\lambda}(1)\,t_{r}^{i}(|A_{i}z|^{2})\,n_{i}\bigg]\\
&+\sum_{i\ne j}\bigg[\frac{h}{2}
\int_{\partial_{j}\mathcal{M}_{i}^{\ast}\times\mathcal{N}^{\ast}}
s\,\mathcal{O}_{\lambda}((sh)^{2})\,t_{r}^{j}(|D_{ij}^{2}z|^{2})
+\frac{h}{2}\int_{\partial_{i}\mathcal{M}\times\mathcal{N}^{\ast}}
s\,\mathcal{O}_{\lambda}((sh)^{2})\,t_{r}^{i}(|z|^{2})\bigg].
\end{aligned}
\end{equation}

\subsection{Estimate of $I_{22}$ \label{lem:I22}}
 We recall that 
\begin{equation}
    A_{2}z:=\sum_{i=1}^{d}\gamma_{i}\trm(rD_{i}^{2}\rho\,A_{i}^{2}z)
\,\text{  and  }\,
    B_{2}z:=-2\trm(s\Delta\xi\,z).
\end{equation}
Let us set $\beta_{22}^{i}:=-2s\gamma_{i}\Delta_{\Gamma} \xi\,rD_{i}^{2}\rho$. Then, shifting the time variable we have
\begin{equation}
I_{22}=\sum_{i=1}^{d}\int_{\mathcal{M}\times\mathcal{N}^{\ast}} \beta_{22}^{i}\,A_{i}^{2}z\,z=:\sum_{i=1}^{d}I_{22}^{i}.
\end{equation}
By using the identity $A_{i}^{2}z=z+\frac{h^{2}}{4}D_{i}^{2}z$ we rewrite $I_{22}^{ij}$ as
\begin{equation}
   I_{22}^{i}=\int_{\mathcal{M}\times\mathcal{N}^{\ast}} \beta_{22}^{i}\,z\,\left(z+\frac{h^{2}}{4}D_{i}^{2}z\right)
    =\int_{\mathcal{M}\times\mathcal{N}^{\ast}} \beta_{22}^{i}|z|^{2}+\frac{h^{2}}{4}\int_{\mathcal{M}\times\mathcal{N}^{\ast}} \beta_{22}^{i}\,z\,D_{i}^{2}z.
\end{equation}
Integration by parts with respect to the difference operator $D_{i}$ yields 
\begin{equation}
    \int_{\mathcal{M}\times\mathcal{N}^{\ast}}\beta_{22}^{i}zD_{i}^{2}z=-\int_{\mathcal{M}_{i}^{\ast}\times\mathcal{N}^{\ast}}D_{i}(\beta_{22}^{i}\,z)\,D_{i}z,
\end{equation}
where we have used that $z=0$ on $\partial_{i}\mathcal{M}$ for $i=1,\ldots,d$. By using product rule in the operator $D_{i}$ \cite[Lemma 2.1, equation (13)]{LDOP-2021} it follows that
\begin{equation}
    I_{22}^{i}=\int_{\mathcal{M}\times\mathcal{N}^{\ast}} \beta_{22}^{i}|z|^{2}-\frac{h^{2}}{4}\int_{\mathcal{M}_{i}^{\ast}\times\mathcal{N}^{\ast}}D_{i}\beta_{22}^{i}\,A_{i}z\,D_{i}z-\frac{h^{2}}{4}\int_{\mathcal{M}_{i}^{\ast}\times\mathcal{N}^{\ast}}A_{i}\beta_{22}^{i}\,|D_{i}z|^{2}.
\end{equation}
Then, using the identity $D_{(|z|^{2})}=2A_{i}zD_{i}z$ we rewrite the above expression as
\begin{equation}
I_{22}^{i}=\int_{\mathcal{M}\times\mathcal{N}^{\ast}} \beta_{22}^{i}|z|^{2}-\frac{h^{2}}{8}\int_{\mathcal{M}_{i}^{\ast}\times\mathcal{N}^{\ast}}D_{i}\beta_{22}^{i}\,D_{i}(|z|^{2})-\frac{h^{2}}{4}\int_{\mathcal{M}_{i}^{\ast}\times\mathcal{N}^{\ast}}A_{i}\beta_{22}^{i}\,|D_{i}z|^{2}.
\end{equation}
After integration by parts with respect to the difference operator $D_{i}$, we obtain
\begin{equation*}
    I_{22}^{i}=\int_{\mathcal{M}\times\mathcal{N}^{\ast}} \beta_{22}^{i}|z|^{2}+\frac{h^{2}}{8}\int_{\mathcal{M}\times\mathcal{N}^{\ast}}D_{i}^{2}\beta_{22}^{i}\,|z|^{2}-\frac{h^{2}}{4}\int_{\mathcal{M}_{i}^{\ast}\times\mathcal{N}^{\ast}}A_{i}\beta_{22}^{i}\,|D_{i}z|^{2}.
\end{equation*}
On the other hand, since $\beta_{22}^{i}=-2s\gamma_{i}\Delta_{\Gamma}\xi\cdot rD_{i}^{2}\rho$, with
\begin{equation*}
\Delta_{\Gamma}\xi=\sum_{j=1}^{d}\gamma_{j}\partial_{j}^{2}\xi=\lambda^{2}\xi\sum_{j=1}^{d}\gamma_{j}(\partial_{j}\psi)^{2}+\lambda\xi\sum_{j=1}^{d}\gamma_{j}\partial_{j}^{2}\psi,
\end{equation*}
and, by Theorem~\ref{theo:weight:estimates},
\begin{equation*}
rD_{i}^{2}\rho=s^{2}\lambda^{2}\xi^{2}(\partial_{i}\psi)^{2}+s\,\mathcal{O}_{\lambda}(1)+s^{2}\mathcal{O}_{\lambda}((sh)^{2}),
\end{equation*}
we obtain
\begin{equation}\label{eq:beta22:i:estimate}
\begin{aligned}
\beta_{22}^{i}=&\;-2s^{3}\lambda^{4}\gamma_{i}\xi^{3}(\partial_{i}\psi)^{2}\sum_{j=1}^{d}\gamma_{j}(\partial_{j}\psi)^{2}+s^{3}\lambda^{3}\xi^{3}\,\mathcal{O}(1)+s^{2}\mathcal{O}_{\lambda}(1)+s^{3}\mathcal{O}_{\lambda}((sh)^{2}).
\end{aligned}
\end{equation}
Indeed, the leading term arises from multiplying the dominant parts:
\begin{equation*}
-2s\gamma_{i}\,\lambda^{2}\xi\sum_{j=1}^{d}\gamma_{j}(\partial_{j}\psi)^{2}\, s^{2}\lambda^{2}\xi^{2}(\partial_{i}\psi)^{2}=-2s^{3}\lambda^{4}\gamma_{i}\xi^{3}(\partial_{i}\psi)^{2}\sum_{j=1}^{d}\gamma_{j}(\partial_{j}\psi)^{2},
\end{equation*}
while term $s^{3}\lambda^{3}\xi^{3}\,\mathcal{O}(1)$ comes from the cross products
\begin{equation*}
-2s\gamma_{i}\,\lambda\xi\sum_{j=1}^{d}\gamma_{j}\partial_{j}^{2}\psi\, s^{2}\lambda^{2}\xi^{2}(\partial_{i}\psi)^{2}=-2s^{3}\lambda^{3}\gamma_{i}\xi^{3}(\partial_{i}\psi)^{2}\sum_{j=1}^{d}\gamma_{j}\partial_{j}^{2}\psi=s^{3}\lambda^{3}\xi^{3}\,\mathcal{O}(1).
\end{equation*}
The remaining terms $s^{2}\mathcal{O}_{\lambda}(1)$ and $s^{3}\mathcal{O}_{\lambda}((sh)^{2})$ arise from the lower-order contributions in $rD_{i}^{2}\rho$ and the discrete error in the theorem. Consequently, 
\begin{equation*}
\sum_{i=1}^{d}\beta_{22}^{i}=-2s^{3}\lambda^{4}\xi^{3}\bigg(\sum_{i=1}^{d}\gamma_{i}(\partial_{i}\psi)^{2}\bigg)^{2}+s^{3}\lambda^{3}\xi^{3}\,\mathcal{O}(1)+s^{2}\mathcal{O}_{\lambda}(1)+s^{3}\mathcal{O}_{\lambda}((sh)^{2}).
\end{equation*}
Moreover, we also have
\begin{equation*}
D_{i}^{2}\beta_{22}^{i}=s^{3}\mathcal{O}_{\lambda}(1),\qquad A_{i}\beta_{22}^{i}=s^{3}\mathcal{O}_{\lambda}(1)+s\,\mathcal{O}_{\lambda}((sh)^{2}),
\end{equation*}
Thus, by using the above bounds and \eqref{eq:beta22:i:estimate}, $I_{22}$ can be estimated as 
\begin{equation}\label{eq:I22:final}
I_{22}=-2\lambda^{4}\int_{\mathcal{M}\times\mathcal{N}^{\ast}}s^{3}\xi^{3}\bigg(\sum_{i=1}^{d}\gamma_{i}(\partial_{i}\psi)^{2}\bigg)^{2}|z|^{2}+X_{22},
\end{equation}
where 
\begin{equation}\label{eq:X22}
\begin{aligned}
X_{22}:=&\int_{\mathcal{M}\times\mathcal{N}^{\ast}}\Big(s^{3}\lambda^{3}\xi^{3}\,\mathcal{O}(1)+s^{2}\mathcal{O}_{\lambda}(1)+s^{3}\mathcal{O}_{\lambda}((sh)^{2})\Big)|z|^{2}\\
&+\int_{\mathcal{M}\times\mathcal{N}^{\ast}}s\,\mathcal{O}_{\lambda}((sh)^{2})|z|^{2}+\sum_{i=1}^{d}\int_{\mathcal{M}_{i}^{\ast}\times\mathcal{N}^{\ast}}s\,\mathcal{O}_{\lambda}((sh)^{2})|D_{i}z|^{2}.
\end{aligned}
\end{equation}
\begin{remark}
The leading term of $I_{22}$ is $\displaystyle -2\lambda^{4}\int s^{3}\xi^{3}\bigg(\sum_{i}\gamma_{i}(\partial_{i}\psi)^{2}\bigg)^{2}|z|^{2}$, which is negative. When combined with the positive leading term of $I_{21}$, namely 
$$3s^{3}\lambda^{4}\int\xi^{3}\bigg(\sum_{i}\gamma_{i}(\partial_{i}\psi)^{2}\bigg)^{2}|z|^{2},$$ the contribution is
\begin{equation*}
I_{21}+I_{22}=s^{3}\lambda^{4}\int_{\mathcal{M}\times\mathcal{N}^{\ast}}\xi^{3}\bigg(\sum_{i=1}^{d}\gamma_{i}(\partial_{i}\psi)^{2}\bigg)^{2}|z|^{2}+\widetilde{W}_{21}+X_{22},
\end{equation*}
which provides the main positive $s^{3}\lambda^{4}|\nabla\psi|^{4}|z|^{2}$ term in the Carleman estimate (when $\gamma_{i}=\gamma$ for all $i$, the squared sum becomes $\gamma^{2}|\nabla\psi|^{4}$). The sub-leading terms in $X_{22}$ are absorbed for $\lambda$ large (the $s^{3}\lambda^{3}$ term) and for $sh$ small (the $(sh)^{2}$ terms).
\end{remark}
\subsection{Estimate of $I_{23}$}
Let us recall that
\begin{equation*}
    A_{2}z:=\sum_{i=1}^{d}\gamma_{i}\trm(rD_{i}^{2}\rho A_{i}^{2}z)\, \text{ and }B_{3}z:=D_{t}z.
\end{equation*}
Then, setting $\beta_{23}:=rD_{i}^{2}\rho$, we have to estimate
\begin{equation}
I_{23}=\sum_{i=1}^{d}\int_{\mathcal{M}\times\mathcal{N}}\gamma_{i}\trm(\beta_{23}A_{i}^{2}z)\,D_{t}z=:\sum_{i=1}^{d}I^{i}_{23}.
\end{equation}
Using the identity $A_{i}^{2}z=z+\frac{h^{2}}{4}D_{i}^{2}z$, $I^{i}_{23}$ can be rewritten as
\begin{equation}
\begin{split}
    I^{i}_{23}=&\int_{\mathcal{M}\times\mathcal{N}}\gamma_{i}\trm(\beta_{23})\trm(z)\,D_{t}z+\frac{h^{2}}{4}\int_{\mathcal{M}\times\mathcal{N}}\gamma_{i}\trm(\beta_{23}\,D_{i}^{2}z)\,D_{t}z\\
    =:&I_{23}^{i(a)}+I_{23}^{i(b)}.
    \end{split}
    \end{equation}
\par We note that the identity $\trm z\,D_{t}z=\frac{1}{2}D_{t}(|z|^{2})-\frac{\Delta t}{2}|D_{t}z|^{2}$, for $I_{23}^{i(a)}$, yields 
\begin{equation*}
\begin{split}
    I_{23}^{i(a)}=&\frac{1}{2}\int_{\mathcal{M}\times\mathcal{N}}\gamma_{i}\trm(\beta_{23})\,D_{t}(|z|^{2})-\frac{\Delta t}{2}\int_{\mathcal{M}\times\mathcal{N}}\gamma_{i}\trm(\beta_{23})\,|D_{t}z|^{2}.
    \end{split}
    \end{equation*}
A discrete integration by parts with respect to the time difference operator $D_{t}$, \cite[Equation (2.9) ]{LMPZ:2023} enables us to write
\begin{equation*}
\begin{split}
    I_{23}^{i(a)}=&-\frac{1}{2}\int_{\mathcal{M}\times\mathcal{N}}\gamma_{i}D_{t}(\beta_{23})\trp(|z|^{2})+\frac{1}{2}\int_{\mathcal{M}\times\partial\mathcal{N}}\gamma_{i}\trp(\beta_{23}|z|^{2})\,n\\
    &-\frac{\Delta t}{2}\int_{\mathcal{M}\times\mathcal{N}}\gamma_{i}\trm(\beta_{23})\,|D_{t}z|^{2}.
\end{split}
\end{equation*}
Thus, by using Theorem \ref{theo:fully:weight:estimates} and using the inequality $|\trp(|z|^{2})|\leq 2\trm(|z|^{2})+2(\Delta t)^{2}|D_{t}z|^{2}$, for $I_{23}^{i(a)}$ we obtain
\begin{equation}\label{ine:I:23:a}
    \begin{aligned}
    I_{23}^{i(a)}=&-\int_{\mathcal{M}\times\mathcal{N}}\trm\mu\,\trm(|z|^{2})-(\Delta t)^{2}\int_{\mathcal{M}\times\mathcal{N}}\trm(Ts^{2}\theta)|D_{t}z|^{2}\\
    &+\int_{\mathcal{M}\times\partial\mathcal{N}}\mathcal{O}_{\lambda}(1)\trp(s^{2}|z|^{2})\,n-\Delta t\int_{\mathcal{M}\times\mathcal{N}}\trm(s^{2})\mathcal{O}_{\lambda}(1)\,|D_{t}z|^{2},
    \end{aligned}
\end{equation}
where $\mu:=Ts^{2}\theta\mathcal{O}_{\lambda}(1)+\left(\frac{\tau^{2}\Delta t}{\delta^{4}T^{6}}\right)\mathcal{O}_{\lambda}(1)+\left(\frac{\tau \Delta t}{\delta^{3}T^{4}}\right)\left(\frac{\tau h}{\delta T^{2}}\right)^{2}\mathcal{O}_{\lambda}(1)$.
\par Let us now estimate $I_{23}^{i(b)}$. Using a discrete integral by parts for the operator $D_{i}$, and $D_{t}z=0$ on $\partial\mathcal{M}$, we obtain
\begin{equation*}
    I_{23}^{(b)}=-\frac{h^{2}}{4}\int_{\mathcal{M}_{i}^{\ast}\times\mathcal{N}}D_{i}(\gamma_{i}\trm\beta_{23}D_{t}z)\trm(D_{i}z).
\end{equation*}
By virtue of the product rule for the difference operator $D_{i}$ we write
\begin{equation*}
\begin{split}
    I_{23}^{(b)}=&-\frac{h^{2}}{4}\int_{\mathcal{M}_{i}^{\ast}\times\mathcal{N}}\trm(D_{i}(\gamma_{i}\beta_{23}))A_{i}D_{t}z\,\trm(D_{i}z)-\frac{h^{2}}{4}\int_{\mathcal{M}_{i}^{\ast}\times\mathcal{N}}\trm(A_{i}(\gamma_{i}\beta_{23}))D^{2}_{it}z\,\trm(D_{i}z)\\
    :=&I_{23}^{(b_{1})}+I_{23}^{(b_{2})}.
    \end{split}
\end{equation*}
We note that $D_{it}^{2}z\,\trm(D_{i}z)=\frac{1}{2}D_{t}(|D_{i}z|^{2})-\frac{\Delta t}{2}|D^{2}_{ti}z|^{2}$, thanks to the identity $\trm z\,D_{t}z=\frac{1}{2}D_{t}(|z|^{2})-\frac{\Delta t}{2}|D_{t}z|^{2}$. Then, $I_{23}^{(b_{2})}$ can be rewritten as
\begin{equation}
\begin{aligned}
    I_{23}^{(b_{2})}=&-\frac{h^{2}}{8}\int_{\mathcal{M}_{i}^{\ast}\times\mathcal{N}}\trm(A_{i}(\gamma_{i}\beta_{23}))D_{t}(|D_{i}z|^{2})+\frac{\Delta t\,h^{2}}{8}\int_{\mathcal{M}_{i}^{\ast}\times\mathcal{N}}\trm(A_{i}(\gamma_{i}\beta_{23})) 
    |D^{2}_{ti}z|^{2}.
    \end{aligned}
\end{equation}
Applying a discrete integration by parts for the discrete difference time operator $D_{t}$ on the first integral above leads to
\begin{equation}
\begin{split}
    I_{23}^{(b_{2})}=&\frac{h^{2}}{8}\int_{\mathcal{M}_{i}^{\ast}\times\mathcal{N}}A_{i}D_{t}(\gamma_{i}\beta_{23})\trp(|D_{i}z|^{2})-\frac{h^{2}}{8}\int_{\mathcal{M}_{i}^{\ast}\times\partial\mathcal{N}}\trp(A_{i}(\gamma_{i}\beta_{23})|D_{i}z|^{2})\,n\\
    &+\frac{\Delta th^{2}}{8}\int_{\mathcal{M}_{i}^{\ast}\times\mathcal{N}}\trm(A_{i}(\gamma_{i}\beta_{23}))\,|D^{2}_{ti}z|^{2}.
    \end{split}
\end{equation}
We note that using the estimates Theorem \ref{theo:fully:weight:estimates} and the estimates for the spatial variable it follows that
\begin{equation}\label{eq:I_{23}:}
    \begin{split}
        A_{i}D_{t}(\gamma_{i}\beta_{23})&=T\trm(s^{2}\theta)\mathcal{O}_{\lambda}(1)+\frac{\tau^{2}\Delta t}{\delta^{4} T^{6}}\mathcal{O}_{\lambda}(1)+\left( \frac{\tau \Delta t}{\delta^{3}T^{4}}\right)\left( \frac{\tau h}{\delta T^{2}}\right)^{2}\mathcal{O}_{\lambda}(1),\\
        A_{i}(\gamma_{i}\beta_{23})&=s^{2}\mathcal{O}_{\lambda}(1).
    \end{split}
\end{equation}
Combining the previous estimates and \eqref{eq:I_{23}:}, and using $\trp(|D_{i}z|^{2})\leq 2\trm(|D_{i}z|^{2})+2(\Delta t)^{2}|D_{ti}^{2}z|^{2}$, we have $I_{23}^{(b_{2})}$ can be estimated as 
\begin{equation}\label{ine:I:23:b2}
\begin{aligned}
|I_{23}^{(b_{2})}|\leq&\int_{\mathcal{M}_{i}^{\ast}\times\mathcal{N}}T\trm(\theta (sh)^{2})\trm(|D_{i}z|^{2})+(\Delta t)^{2}\int_{\mathcal{M}_{i}^{\ast}\times\mathcal{N}}T\trm(\theta (sh)^{2})|D^{2}_{it}z|^{2})\\
&+h^{2}\int_{\mathcal{M}^{\ast}\times\partial\mathcal{N}}\frac{\tau^{2}\Delta t}{\delta^{4} T^{6}}\mathcal{O}_{\lambda}(1)\trp(|D_{i}z|^{2})+\Delta th^{2}\int_{\mathcal{M}\times\mathcal{N}}\mathcal{O}_{\lambda}(1)\trm(s^{2})\,|D^{2}_{ti}z|^{2}\\
&+h^{2}\int_{\mathcal{M}_{i}^{\ast}\times \partial \mathcal{N}}\left( \frac{\tau \Delta t}{\delta^{3}T^{4}}\right)\left( \frac{\tau h}{\delta T^{2}}\right)^{2}\mathcal{O}_{\lambda}(1)\trp(|D_{i}z|^{2}).
    \end{aligned}
\end{equation}
\par On the other hand, we focus now on $I_{23}^{(b_{1})}$. We note that Young's inequality yields
\begin{equation}
    |I_{23}^{(b_{1})}|\leq C \int_{\mathcal{M}_{i}^{\ast}\times\mathcal{N}}h^{2}|\trm(s^{-1}D_{i}(\gamma_{i}\beta_{23}))|A_{i}D_{t}z|^{2}+\int_{\mathcal{M}_{i}^{\ast}\times\mathcal{N}}h^{2}|\trm(D_{i}(\gamma_{i}\beta_{23}))|\trm(s|D_{i}z|^{2}).
\end{equation}
Using that $|A_{i}D_{t}z|^{2}\leq A_{i}(|D_{t}z|^{2})$ on the first integral from the above expression and then a discrete integration by part for the average operator $A_{i}$ we have
\begin{equation}
\begin{split}
    |I_{23}^{(b_{1})}|\leq &C \int_{\mathcal{M}\times\mathcal{N}}h^{2}\trm(
    s^{-1}|A_{i}D_{i}(\gamma_{i}\beta_{23})|)\,|D_{t}z|^{2}+\int_{\mathcal{M}^{\ast}\times\mathcal{N}}h^{2}|\trm(D_{i}\alpha_{23})|\trm(s|D_{i}z|^{2}),
    \end{split}
\end{equation}
where we have used $D_{t}z=0$ on $\partial_{i}\mathcal{M}$. Finally, using Theorem \ref{theo:fully:weight:estimates}, $I_{23}^{(b_{1})}$ can be estimated as follows
\begin{equation}\label{ine:I:23:b1}
    |I_{23}^{(b_{1})}|\leq C \int_{\mathcal{M}\times\mathcal{N}}s^{-1}\mathcal{O}_{\lambda}((sh)^{2})\,|D_{t}z|^{2}+\int_{\mathcal{M}^{\ast}\times\mathcal{N}}\mathcal{O}_{\lambda}((sh)^{2})\trm(s|D_{i}z|^{2}).
\end{equation}
Thus, combining estimates \eqref{ine:I:23:a}, \eqref{ine:I:23:b1} and \eqref{ine:I:23:b2} we get
\begin{equation}\label{eq:I23:final}
\begin{aligned}
I_{23}^{i}=&-\int_{\mathcal{M}\times\mathcal{N}}\gamma_{i}\mu\,\trm(|z|^{2})+X_{23}^{i}+Y_{23}^{i},
\end{aligned}
\end{equation}
where $\mu:=Ts^{2}\theta\,\mathcal{O}_{\lambda}(1)+\frac{\tau^{2}\Delta t}{\delta^{4}T^{6}}\,\mathcal{O}_{\lambda}(1)+\frac{\tau^{4}h^{2}\Delta t}{\delta^{6}T^{10}}\,\mathcal{O}_{\lambda}(1)$, the volume remainder $X_{23}^{i}$ satisfies
\begin{equation}\label{eq:R23}
\begin{aligned}
X_{23}^{i}:=&-\bigg[(\Delta t)^{2}Ts^{2}\theta\,\mathcal{O}_{\lambda}(1)+\Delta t\,s^{2}\,\mathcal{O}_{\lambda}(1)+s^{-1}\mathcal{O}_{\lambda}((sh)^{2})\bigg]\int_{\mathcal{M}\times\mathcal{N}}|D_{t}z|^{2}\\
&+\bigg[T\theta(sh)^{2}\,\mathcal{O}_{\lambda}(1)+\mathcal{O}_{\lambda}((sh)^{2})\,s\bigg]\int_{\mathcal{M}_{i}^{\ast}\times\mathcal{N}}\trm(|D_{i}z|^{2})\\
&+(\Delta t)^{2}T\theta(sh)^{2}\,\mathcal{O}_{\lambda}(1)\int_{\mathcal{M}_{i}^{\ast}\times\mathcal{N}}|D_{ti}^{2}z|^{2}\\
&+\Delta t\,h^{2}\,s^{2}\,\mathcal{O}_{\lambda}(1)\int_{\mathcal{M}_{i}^{\ast}\times\mathcal{N}}\trm(|D_{ti}^{2}z|^{2}),
\end{aligned}
\end{equation}
and the boundary remainder $Y_{23}^{i}$ collects the contributions on $\partial\mathcal{N}$:
\begin{equation}\label{eq:B23}
\begin{aligned}
Y_{23}^{i}:=&\;\frac{1}{2}\int_{\mathcal{M}\times\partial\mathcal{N}}\gamma_{i}\,\mathcal{O}_{\lambda}(1)\,\trp(s^{2}|z|^{2})\,n-\frac{h^{2}}{8}\int_{\mathcal{M}_{i}^{\ast}\times\partial\mathcal{N}}\trp(A_{i}(\gamma_{i}\beta_{23})\,|D_{i}z|^{2})\,n.
\end{aligned}
\end{equation}
\par Summing over $i=1,\ldots,d$, the full estimate for $I_{23}$ reads

\begin{equation}\label{eq:I23:total}
\begin{aligned}
I_{23}=&-\sum_{i=1}^{d}\int_{\mathcal{M}\times\mathcal{N}}\gamma_{i}\mu\,\trm(|z|^{2})+\sum_{i=1}^{d}X_{23}^{i}+\sum_{i=1}^{d}Y_{23}^{i}.
\end{aligned}
\end{equation}
\begin{remark}
The terms in $X_{23}^{i}$ have the following roles in the Carleman estimate:
\begin{itemize}
    \item The $|D_{t}z|^{2}$ terms: the leading contribution $\Delta t\,s^{2}\,\mathcal{O}_{\lambda}(1)$ must be absorbed by the Carleman term $\frac{1}{s}\int|D_{t}z|^{2}$. The $(\Delta t)^{2}Ts^{2}\theta$ and $s^{-1}(sh)^{2}$ terms are lower order.
    \item The $\trm(|D_{i}z|^{2})$ terms: the leading contribution $T\theta(sh)^{2}$ is absorbed for $sh$ small, and $s(sh)^{2}$ is lower order.
    \item The $|D_{ti}^{2}z|^{2}$ terms: both carry factors of $(\Delta t)^{2}$ or $\Delta t\,h^{2}$, making them lower order.
\end{itemize}
\end{remark}

\subsection{Estimate of $I_{31}$}
Recall that 
$A_{3}z=-\tau\varphi\trm(\theta'z)$ and $B_{1}z=2\sum_{i=1}^{d}\gamma_{i}\trm(rD_{i}A_{i}\rho D_{i}A_{i}z)$. Let us set $\alpha_{31}:=-\tau\varphi\theta'\gamma_{i}rD_{i}A_{i}\rho$. Then, shifting the time variable for $I_{31}$ yields $$\displaystyle I_{31}=\sum_{i=1}^{d}\int_{\mathcal{M}\times\mathcal{N}^{\ast}}2\alpha_{31}z\,D_{i}A_{i}z.$$ A discrete integration by parts with respect to the average operator $A_{i}$ leads to
\begin{equation*}
I^{i}_{31}=\int_{\mathcal{M}_{i}^{\ast}\times\mathcal{N}^{\ast}}2A_{i}(\alpha_{31}z)D_{i}z.
\end{equation*}
From \cite[Lemma 2.1, equation (13)]{LDOP-2021} and the identity $2D_{i}zA_{i}z=D_{i}(|z|^{2})$, for $I_{31}^{(a)}$, we have
\begin{equation*}
\begin{split}
    I_{31}^{i}=&\int_{\mathcal{M}_{i}^{\ast}\times\mathcal{N}^{\ast}}A_{i}(\alpha_{31})\,D_{i}(|z|^{2})+\frac{h^{2}}{2}\int_{\mathcal{M}_{i}^{\ast}\times\mathcal{N}^{\ast}}D_{i}(\alpha_{31})\,|D_{i}z|^{2}.
    \end{split}
\end{equation*}
A discrete integration by parts with respect to the difference operator $D_{i}$ yields
\begin{equation*}
\begin{split}
    I_{31}^{(i)}=&-\int_{\mathcal{M}\times\mathcal{N}^{\ast}}D_{i}A_{i}(\alpha_{31})\,|z|^{2}+\frac{h^{2}}{2}\int_{\mathcal{M}_{i}^{\ast}\times\mathcal{N}^{\ast}}D_{i}(\alpha_{31})\,|D_{i}z|^{2},
    \end{split}
\end{equation*}
since $z=0$ on $\partial_{i}\mathcal{M}$ for $i=1,\ldots, d$. We note that $e^{s\varphi}\partial_{t}(e^{-s\varphi})=-\tau\varphi\partial_{t}\theta$ and by Theorem \ref{theo:weight:estimates} we have $rA_{i}D_{i}\rho=s\mathcal{O}_{\lambda}(1)$. Thus, for $I_{31}^{i}$ we have the following estimate
\begin{equation}
\begin{split}
I_{31}^{(i)}=&\int_{\mathcal{M}\times\mathcal{N}^{\ast}}Ts^{2}\theta\mathcal{O}_{\lambda}(1)\,|z|^{2}+\int_{\mathcal{M}_{i}^{\ast}\times\mathcal{N}^{\ast}}T\theta\mathcal{O}_{\lambda}((sh)^{2})\,|D_{i}z|^{2}.
    \end{split}
\end{equation}
Therefore, $I_{31}$ can be estimated as 
\begin{equation*}
I_{31}=\int_{\mathcal{M}\times\mathcal{N}^{\ast}}dTs^{2}\theta\mathcal{O}_{\lambda}(1)\,|z|^{2}+\sum_{i=1}^{d}\int_{\mathcal{M}_{i}^{\ast}\times\mathcal{N}^{\ast}}T\theta\mathcal{O}_{\lambda}((sh)^{2})\,|D_{i}z|^{2}.
\end{equation*}
\subsection{Estimate of $I_{32}$}
Let us set $\alpha_{32}:=2s\Delta_{\Gamma}\xi\tau\varphi\theta'$. We have shifting the discrete time variable $I_{32}$ is given by $\displaystyle I_{32}=\int_{\mathcal{M}\times\mathcal{N}^{\ast}}\alpha_{32}|z|^{2}$. Then, it follows that 
$$I_{32}\geq -\int_{\mathcal{M}\times\mathcal{N}^{\ast}}Ts^{2}\theta\mathcal{O}_{\lambda}(1)|z|^{2},$$
since $|\theta'|\leq CT\theta^{2}$, $\left\| \Delta_{\Gamma}\xi\right\|_{L^{\infty}_{h}(\mathcal{M}\times\mathcal{N}^{\ast})}\leq \mathcal{O}_{\lambda}(1)$ and $\left\| \varphi\right\|_{C(\Omega)}=\mathcal{O}_{\lambda}(1)$.
\subsection{Estimate of $I_{33}$}
We have to estimate
$\displaystyle I_{33}:=-\int_{\mathcal{M}\times\mathcal{N}}\tau\varphi\trm(\theta'\,z)D_{t}z$. Note that by using the identity $\trm(z)D_{t}z=\frac{1}{2}D_{t}(|z|^{2})-\frac{\Delta t}{2}|D_{t}z|^{2}$, $I_{33}$ can be rewritten as 
\begin{equation*}
I_{33}=-\frac{1}{2}\int_{\mathcal{M}\times\mathcal{N}}\tau\varphi\trm(\theta')\,D_{t}(|z|^{2})+\frac{\Delta t}{2}\int_{\mathcal{M}\times\mathcal{N}}\tau\varphi\trm(\theta')|D_{t}z|^{2}.
\end{equation*}
Now, using \cite[Equation (2.9)]{LMPZ:2023} on the first integral above we have 
\begin{equation*}
\begin{split}
    I_{33}=&\frac{1}{2}\int_{\mathcal{M}\times\mathcal{N}}\tau\varphi\,D_{t}(\theta')\,\trp(|z|^{2})-\frac{1}{2}\int_{\mathcal{M}\times\partial \mathcal{N}}\tau\varphi\trp(|z|^{2})\trp(\theta')n+\frac{\Delta t}{2}\int_{\mathcal{M}\times\mathcal{N}}\tau\varphi\trm(\theta')|D_{t}z|^{2}.
    \end{split}
\end{equation*}
We note that $\trp(\theta')\,n>0$ on $\mathcal{M}\times\partial\mathcal{N}$ since $\theta'=(2t-T)\theta^{2}$. Moreover, thanks to $\varphi<0$ in $\mathcal{M}$, it follows that
\begin{equation}
-\frac{1}{2}\int_{\mathcal{M}\times\partial\mathcal{N}}\tau\varphi\trp(|z|^{2})\trp(\theta')\,n>0.
\end{equation}
Thus, we can drop the boundary terms and $I_{33}$ can be estimated as
\begin{equation*}
    I_{33}\geq\frac{1}{2}\int_{\mathcal{M}\times\mathcal{N}}\tau\varphi\,D_{t}(\theta')\,\trp(|z|^{2})+\frac{\Delta t}{2}\int_{\mathcal{M}\times\mathcal{N}^{\ast}}\tau\varphi\trm(\theta')|D_{t}z|^{2}.
\end{equation*}    
Now, let us focus on the integral with the term $|\trp z|^{2}$. Using that $|\trp z|^{2}\leq C|\trm z|^{2}+C(\Delta t)^{2}|D_{t}z|^{2}$ and Theorem \ref{theo:fully:weight:estimates} we obtain
\begin{equation*}
    \begin{split}
        I_{33}\geq &-\int_{\mathcal{M}\times\mathcal{N}}\trm(\mu_{33})|\trm z|^{2}-\int_{\mathcal{M}\times\mathcal{N}}(\Delta t)^{2}\trm(\mu_{33})|D_{t}z|^{2}-\int_{\mathcal{M}\times\mathcal{N}}\trm(\nu_{33})|D_{t}z|^{2},
    \end{split}
\end{equation*}
with $\mu_{33}:=(\tau T^{2}\theta^{3}+\frac{\tau \Delta t}{\delta^{4}T^{5}})\mathcal{O}_{\lambda}(1)$ and $\nu_{33}:=\Delta t \tau T\theta^{2}\mathcal{O}_{\lambda}(1)$. We can shift the integral with the term $|\trm z|^{2}$.

\bibliographystyle{abbrv}
\bibliography{references}
\end{document}